\documentclass{amsart}
  \usepackage[numbers]{natbib}

  \usepackage{rotating}
  \usepackage{floatpag}
  \rotfloatpagestyle{empty}

 \usepackage{amsmath}% if you are using this package,
  \usepackage{amsthm}
  \usepackage{graphicx}

\usepackage{amssymb}
\usepackage{mathrsfs}
\usepackage[all]{xy}

\usepackage{zref-xr,hyperref,pb-diagram,amssymb,epic,eepic,verbatim,graphicx,graphics,epsfig,psfrag}
\usepackage{tikz-cd}
\hypersetup{colorlinks}
\usepackage{color}

\definecolor{darkred}{rgb}{0.5,0,0}
\definecolor{darkgreen}{rgb}{0,0.5,0}
\definecolor{darkblue}{rgb}{0,0,0.5}

\hypersetup{ colorlinks,
linkcolor=darkblue,
filecolor=darkgreen,
urlcolor=darkred,
citecolor=darkblue }

\theoremstyle{plain}

\theoremstyle{remark}
\newcommand\M{\mathcal{M}}

\renewcommand\M{\mathcal{M}}
\renewcommand\L{\mathcal{L}}

\newcommand{\cS}{\mathcal{S}}

\newcommand{\XX}{\mathcal{X}}

\newcommand{\R}{\mathbb{R}}

\newcommand{\C}{\mathbb{C}}
\newcommand{\cC}{\mathcal{C}}

\renewcommand{\ss}{\on{ss}}

\newcommand{\Z}{\mathbb{Z}}
\newcommand{\Q}{\mathbb{Q}}

\renewcommand{\P}{\mathbb{P}}
\newcommand{\cP}{\mathcal{P}}
\newcommand{\bA}{\mathbb{A}}

\newcommand{\Ct}{\mathbb{C}^{\times}}% 
\newcommand\lie[1]{\mathfrak{#1}}

\newcommand{\g}{\lie{g}}

\renewcommand{\t}{\lie{t}}

\newcommand{\on}{\operatorname}

\newcommand{\com}{\on{com}}
\newcommand{\coh}{\on{coh}}
\newcommand{\Jac}{\on{Jac}}

\newcommand{\Bl}{\on{Bl}}

\newcommand{\Ch}{\on{Ch}}
\newcommand{\vir}{{\on{vir}}}
\newcommand{\Td}{\on{Td}}

\newcommand{\Coh}{\on{Coh}}

\newcommand{\dual}{\vee}

\newcommand{\loc}{{\on{loc}}}
\newcommand{\Ver}{\on{Vert}}

\newcommand{\End}{\on{End}}

\newcommand{\fin}{\on{fin}}

\newcommand{\Aut}{ \on{Aut} } 
 
\newcommand{\Ob}{ \on{Ob} }

\newcommand{\aut}{ \on{aut} }

\newcommand{\Hom}{ \on{Hom}}

\newcommand{\Ind}{ \on{Ind}}
\newcommand{\fix}{{\on{fix}}}
\newcommand{\mov}{{\on{mov}}}

\newcommand{\Resid}{\on{Resid}}
\newcommand{\Restr}{\on{Restr}}

\newcommand\dirac{/\kern-1.2ex\partial} % Dirac operator
\newcommand\qu{/\kern-.7ex/} % Categorical quotients
\newcommand\lqu{\backslash \kern-.7ex \backslash} % Categorical
\newcommand\dr{r_+ \kern-.7ex - \kern-.7ex r_-}
 
\usepackage{amsmath, amsfonts,amssymb, latexsym, mathrsfs,braket}
\newcommand{\cO}{{\mathcal O}}
\newcommand{\ol}[1]{\overline{#1}}

\newcommand{\tensor}{\otimes}

\newcommand{\inject}{\hookrightarrow}

\renewcommand{\comment}[1]   {{}}
\newcommand{\labell}\label

\renewcommand{\d}{{\on{d}}}
\newcommand{\ovl}{\overline}

\newcommand\eps{\epsilon}

\newcommand{\lan}{\langle}
\newcommand{\ran}{\rangle}

\newcommand{\ti}{\tilde}

\newcommand\pt{\on{pt}}
\newcommand\Def{\on{Def}}

\newcommand\cE{\mathcal{E}}
\newcommand\cL{\mathcal{L}}

\newcommand\cF{\mathcal{F}}

\newcommand\mE{\mathcal{E}}

\newcommand\mT{\mathcal{T}}
\newcommand\MM{\mathfrak{M}}

\newcommand\Map{\on{Map}}
\newcommand\Sym{\on{Sym}}

\newcommand\rank{\on{rank}}
\newcommand\ev{\on{ev}}
\newcommand\Eul{\on{Eul}}
\newcommand\Pic{\on{Pic}}

\newcommand\ul{\underline}
\newcommand\mO{\mathcal{O}}

\newcommand\bdefn{\begin{definition}}
\newcommand\edefn{\end{definition}}
\newcommand\bea{\begin{eqnarray*}}
\newcommand\eea{\end{eqnarray*}}
\newcommand\bcv{\left[ \begin{array}{r} }
\newcommand\ecv{\end{array} \right] }

\newcommand\bma{\left[ \begin{array} }
\newcommand\ema{\end{array} \right]}
\newcommand\ben{\begin{enumerate}}
\newcommand\een{\end{enumerate}}
\newcommand\beq{\begin{equation}}
\newcommand\eeq{\end{equation}}
\newcommand\bex{\begin{example}}
\newcommand\bsj{\left\{ \begin{array}{rrr} }
\newcommand\esj{\end{array} \right\}}

\newcommand\Cone{\on{Cone}}

\newcommand\LL{\mathcal{L}}

\newcommand\eex{\end{example}}

\newcommand\Crit{{\on{Crit}}}

\newcommand\sx{*\kern-.5ex_X}

\renewcommand{\Def}{{\on{Def}}}
\newcommand{\Obs}{{\on{Obs}}}

\newcommand{\cT}{{\mathcal{T}}}

\def\mathunderaccent#1{\let\theaccent#1\mathpalette\putaccentunder}
\def\putaccentunder#1#2{\oalign{\(#1#2\)\crcr\hidewidth \vbox
to.2ex{\hbox{\(#1\theaccent{}\)}\vss}\hidewidth}}

  \theoremstyle{plain}% default 
  \newtheorem{theorem}{Theorem}[section]
  \newtheorem{lemma}[theorem]{Lemma}  

\newtheorem{proposition}[theorem]{Proposition}
  
  \newtheorem{assumption}[theorem]{Assumption}

  \newtheorem*{theorem*}{Theorem}

  \theoremstyle{definition}
  \newtheorem{definition}[theorem]{Definition}
  \newtheorem{example}[theorem]{Example}

  \theoremstyle{remark}
  \newtheorem*{remark}{Remark}

\begin{document}

%  \mainmatter 

%%%%%%%%%%%%%%%%%%%%%%%%%%%%%%%%%%%%%%%%%%%%%

%CONTRIBUTORS: edit the info below 

%%%%%%%%%%%%%%%%%%%%%%%%%%%%%%%%%%%%%%%%%%%%%
\author[E. Gonz\'alez and C. Woodward]{E. Gonz\'alez and C. Woodward}
%\email{asher.auel@dartmouth.edu}

%\contributor{Eduardo Gonz\'alez} 

%MSC 53D45

\address{
Department of Mathematics \\
University of Massachusetts Boston \\
100 William T. Morrissey Boulevard \\
Boston, MA 02125 \\
eduardo@math.umb.edu}

%\contributor{Chris T. Woodward}

\address{Mathematics-Hill Center \\
Rutgers University, 110 Frelinghuysen Road \\ Piscataway, NJ
08854-8019 \\
ctw@math.rutgers.edu}

\title{Quantum Kirwan for quantum K-theory}

\begin{abstract}  
  For $G$ a complex reductive group and $X$ a smooth projective or
  convex quasi-projective polarized $G$-variety we construct a formal
  map in quantum K-theory
  \[ \kappa_X^G: QK_G^0(X) \to QK^0(X \qu G) \]
  from the equivariant quantum K-theory $QK_G^0(X)$ to the quantum
  K-theory of the geometric invariant theory quotient $X \qu G$, assuming the quotient
  $X \qu G$ is a smooth Deligne-Mumford stack with projective coarse
  moduli space.  As an example, we give a presentation of the
  (possibly bulk-shifted) quantum K-theory of any smooth proper toric
  Deligne-Mumford stack with projective coarse moduli space,
  generalizing the presentation for \label{grammar} quantum K-theory of projective
  spaces due to Buch-Mihalcea \cite{buch:qk} and (implicitly) of
  Givental-Tonita \cite{giv:toric}.  \label{reclause} We also provide
  a wall-crossing formula for the K-theoretic gauged potential
\[ \tau_X^G: QK_G^0(X) \to \Lambda_X^G \] 
under variation of geometric invariant theory quotient, a proof of the invariance of
$\tau_X^G$ under (strong) crepant transformation assumptions, and a
proof of the abelian non-abelian correspondence relating $\tau_X^G $
and $\tau_X^T$ for $T \subset G$ a maximal torus. 
\end{abstract}

\thanks{Partially supported by NSF grants DMS-1510518 and DMS-1711070. 
  Any opinions, findings, and conclusions or recommendations expressed 
  in this material are those of the author(s) and do not necessarily 
  reflect the views of the National Science Foundation.  }
\maketitle
 % \listofcontributors 

  \tableofcontents

%\tableofcontents

\section{Introduction}

We aim to describe the behavior of quantum K-theory under the operation of
geometric invariant theory quotient.  Let \(X\)
be a smooth projective \(G\)-variety
with polarization \(\L\to X\).
The geometric invariant theory (git) quotient \(X \qu G\)
is a proper smooth Deligne-Mumford stack with projective coarse moduli
space.  Let \(K_G^0(X)\)
denote the even topological or algebraic \(G\)-equivariant
\(K\)-cohomology
of \(X\).
The {\em Kirwan map} is the ring homomorphism in \(K\)-theory
(``K-theoretic reduction'')
\begin{equation} \label{KR} \kappa_{X}^G: K_G^0(X) \to K^0(X \qu
  G), \quad [E] \mapsto [(E | X^{\ss}) / G ] 
\end{equation}
obtained by restricting a vector bundle $E$ to the semistable locus
\(X^{\ss}\)
and passing to the stack quotient.  If \(X\)
is projective then Kirwan showed that \eqref{KR} is surjective in
rational cohomology \cite{ki:coh}.  The analogous results in
\(K\)-theory
hold by work of Harada-Landweber \cite[Theorem 3.1]{har:surj} and
Halpern-Leistner \cite[Corollary 1.2.3]{hal:der}.  The Kirwan map can
often be used to compute the K-theory of a git quotient.  In
particular, the Kirwan map allows a simple presentation of the
K-theory of a smooth projective toric variety, equivalent to the
presentation given in Vezzosi-Vistoli \cite[Section 6.2]{vv}.

A quantum deformation of the K-theory ring was introduced by Givental
\cite{giv:wdvv} and Y.P. Lee \cite{lee:qk1}. 
 In this deformation, the
tensor product of vector bundles is replaced by a certain push-pull
over the moduli space of stable maps.  The virtual fundamental class
in cohomology is replaced by a {\em virtual structure sheaf}
introduced in \cite{lee:qk1}, and integrals over the moduli space of
stable maps in K-theory are called {\em \(K\)-theoretic
  Gromov-Witten invariants}.  For many purposes quantum K-theory is
expected to be more natural than the quantum cohomology ring, which
can be obtained as a limit; see for example \cite{ok}. \label{okref}  In
particular, the \(K\)-theoretic
Gromov-Witten invariants are integers.  Computations in quantum
K-theory have been rather rare; even the quantum K-theory of
projective space seems to have been computed only recently by
Buch-Mihalcea \cite{buch:qk}.  

We develop a quantum version of Kirwan's map in K-theory.  As
applications, we give presentations of the quantum K-theory of toric
varieties, generalizing a computation of Buch-Mihalcea \cite{buch:qk}
in the case of projective spaces.   Let
\[ \Lambda_X^G \subset \Map(H_2^G(X,\Q),\Q) \] 
denote the equivariant Novikov ring associated to the polarization,
with $q^d$ denoting for $q$ a formal variable the delta function at
$d \in H_2^G(X,\Q)$.  \label{qd} The quantum $K$-theory product is
defined by Givental-Lee \cite{giv:wdvv,lee:qk1} as a pull-push over
moduli spaces of stable maps; in order to make the products finite we
define the equivariant quantum K-theory as the completion of
$K^0_G \otimes \Lambda_X^G$ with respect to the ideal $I_G^X(c)$ of
elements $E \in K^0_G \otimes \Lambda_X^G$ with $\on{val}_q(E) > c$:
\begin{equation} \label{completion}
 QK_G^0(X) = \lim_{n \leftarrow} (K_G^0(X) \otimes \Lambda_X^G)/ I_G^X(c)^n.
\end{equation}
\label{defined}The main result is the following:

\begin{theorem} Let $G$ be a complex reductive group and $X$ be a
  smooth polarized projective (or convex quasiprojective) $G$-variety
  with locally free git quotient $X \qu G$.  There exists a canonical
  {\em Kirwan map} in quantum K-theory
\[ \kappa_X^G: QK_G^0(X) \to QK^0(X \qu G) \]
with the property that the linearization \(D_\alpha
\kappa_X^G\) is a homomorphism:
\[ 
  D_\alpha \kappa_X^G (\beta \star \gamma) = D_\alpha
\kappa_X^G (\beta) \star D_\alpha \kappa_X^G (\gamma) .
\]
If \(X \qu G\)
is a free quotient then \(\kappa_X^G\)
is surjective.\footnote{Computations suggest that the map $\kappa_X^G$
  might be surjective even for locally free quotients of proper free
  actions.}
\end{theorem} 

The convexity assumption is satisfied if, for example, the variety is
a vector space with a torus action such that all weights are properly
contained in a half-space.  For example, for toric varieties that may
be realized as git quotients we explicitly compute the kernel of the
map to obtain a presentation of the orbifold quantum K-theory at a
point determined by the presentation, generalizing the presentation of
ordinary K-theory of non-singular toric varieties due to
Vezzosi-Vistoli \cite[Section 6.2]{vv} and the case of quantum
K-theory of projective spaces due to Buch-Mihalcea \cite{buch:qk}: Let
$G$ be a complex torus and $X$ a vector space with weight spaces
$X_1,\ldots, X_k$ and weights $\mu_1,\ldots, \mu_k$ define the {\em
  completed equivariant quantum K-theory} $\widehat{QK}_G^0(X)$ to be
the ring with generators $X_1^{\pm 1},\ldots, X_k^{\pm 1}$ formally
completed by the ideal generated by
$ (1 - X_j^{-1}), j = 1,\ldots,k$.\footnote{In the case of orbifold
  quotients, a more complicated formal completion is necessary, see
  Definition \ref{formal}.}  The K-theoretic Batyrev (or quantum
K-theoretic Stanley-Reisner) ideal is the ideal $QKSR_X^G$ generated
by the relations
\begin{equation} \label{qsrr} \prod_{(\mu_j,d) \ge 0} (1 - X_j^{-1})^{\mu_j(d)} = q^d
\prod_{(\mu_j,d) < 0} (1 - X_j^{-1})^{-\mu_j(d)} .\end{equation}

\begin{theorem} \label{tmain} Suppose that $G$ is a torus with Lie
  algebra $\g$, that $X$ is a $G$-vector space with weights
  $\mu_1,\ldots,\mu_k \in \g_\R^\dual$ contained in an open half-space
  in $\g_\R^\dual$, and that $X$ is equipped with a polarization so
  that $X \qu G$ is a non-singular proper toric Deligne-Mumford stack
  with projective coarse moduli space.  Let $T = (\C^\times)^k/G$
  denote the residual torus.  Then the quantum $K$-theory ring
  $QK^0(X \qu G)$ with bulk deformation $\kappa_X^G(0)$ is isomorphic to
  the quotient $\widehat{QK_G^0(X)}/QKSR_X^G$.
\end{theorem}

\begin{example} \label{wps} {\rm (Weighted projective spaces)} Suppose
  that \(G = \C^\times\)
  acts on \(X = \C^k\)
  with weights \(\mu_1,\ldots, \mu_k \in \Z\),
  so that \(X \qu G\)
  is the weighted projective space \(\P(\mu_1,\ldots, \mu_k)\).
  Then the \(T\)-equivariant
  quantum K-theory of \(X \qu G\)
  has canonical presentation with generators and a single relation (in
  this case the formal completion is not necessary):
\begin{equation} \label{present} 
QK^0(X \qu G) \cong
\frac{ \Lambda_X^G[ X_1^{ \pm 1}, \ldots, X_k^{\pm 1} ] } { 
 \big\lan \prod_j (1 - X_j^{-1})^{\mu_j(d)} -  q  \big\ran } .\end{equation}
In this case, the bulk deformation \(\kappa_X^G(0)\) turns
out to vanish, see Lemma \ref{vanish} below, and \(X_j\) is
the class of the line bundle associated to the weight
\(\mu_j\).  \end{example} 

\begin{example} {\rm (\(B\Z_2\))}
  This is a sub-example of the previous Example \ref{wps}; we include
  it to emphasize the importance of working over the equivariant
  Novikov ring.  Suppose that \(G = \C^\times\)
  acts on \(X = \C\)
  with weight two.  Then the quantum \(K\)-theory
  of \(X \qu G = B\Z_2\)
  has generators \(X^{\pm 1}\)
  with single relation 
\[ QK(B\Z_2) = \frac{ \Z[ X^{\pm 1}, q  ] }{ 
\big\lan  ( 1 - X^{-1})^2  -  q \big\ran }  . \] 
On the other hand, without the equivariant Novikov ring
\(K(B\Z_2;\Z)\)
is simply \label{simply} the group ring on \(\Z_2\)
via the identification of representations with their characters.
Let \label{deltapm}
\[ \delta_{\pm 1} \in K(B\Z_2;\Z) \] 
be the delta functions at the group elements \(\pm 1 \in \Z_2\).
Proposition \ref{kernelcontains} below shows that \((1 - X^{-1})\)
maps to \(\sqrt{q} \delta_{-1}\)
under $D_0 \kappa_X^G$.  \label{noQ} This implies the relation
\[ (\sqrt{q} \delta_{-1})^2 - q = q \delta_{(-1)(-1)} - q = 0 \]
since \(\delta_1\) is the identity.  This matches with the relation
\[ \delta_{-1}^2 = \delta_1 \]
in $K(B\Z_2)$, the group algebra of $\Z_2$ since the product is given
by convolution. \label{convolution}
\end{example} 

Since the relations are essentially the same as those in quantum
cohomology, one obtains a generalization of \label{grammar2} the
isomorphism between K-theory and cohomology of toric varieties induced
by identifying cohomological and K-theoretic first Chern classes of
divisors in Vezzosi-Vistoli \cite[Section 6.2]{vv}: The quantum
\(K\)-theory
ring at bulk deformation \(\kappa_X^G(0) \in QK^0(X \qu G)\)
is canonically isomorphic to the quantum cohomology ring $QH(X \qu G)$
at bulk deformation \( \kappa_X^{G,\coh}(0) \in QH(X \qu G) \)
(where \(\kappa_X^{G,\coh}: QH_G(X) \to QH(X \qu G)\)
is the cohomological quantum Kirwan map) via a map defined on
generators by
\[
QK^0(X \qu G) \to QH(X \qu G), \quad 
D_\alpha \kappa_X^G(X_j- 1) 
\mapsto D_\alpha \kappa_X^{G,\coh} (c_1(X_j)) .
\]
In particular, the quantum K-theory of toric Deligne-Mumford stacks is
generically semisimple.

We thank Ming Zhang for helpful comments and an anonymous referee for
pointing out an important omission in the orbifold case.

\section{Equivariant quantum K-theory} 

We recall the following basics of equivariant quantum K-theory,
following Buch-Mihalcea \cite{buch:qk} and Iritani-Milanov-Tonita
\cite{iri:qkt}.  Let \(X\)
be a \(G\)-variety.
The {\em equivariant K-homology group} \(K_0^G(X)\)
is the Grothendieck group of coherent \(G\)-sheaves
on \(X\),
that is, the free Abelian group generated by isomorphism classes of
\(G\)-equivariant
coherent sheaves modulo relations whenever there exists an equivariant
exact sequence: For $G$-equivariant sheaves
$\cF,\cF',\cF''$ \label{threeFs} we have
an implication
\[ 
0 \to
\cF' \to \cF \to \cF'' \to 0 \quad \implies \quad 
[\cF] = [\cF'] +
[\cF''] . \]   
The K-homology $K_G^0(X)$ is naturally a module over the {\em
  equivariant \(K\)-cohomology
  ring} \(K_G^0(X)\)
of \(G\)-equivariant
vector bundles on \(X\).
Both the multiplicative structure of \(K_G^0(X)\)
and the module structure of \(K^G_0(X)\)
are given by tensor products.  If \(X\)
is non-singular then the map from $K_G^0(X) $ to $ K^G_0(X)$
that sends a vector bundle to its sheaf of sections is an isomorphism.

Equivariant K-theory has the following functoriality properties.
Given a \(G\)-equivariant
morphism \(f: X \to Y\)
between varieties $X,Y$ there is a ring homomorphism
\[ f^* : K^0_G(Y) \to K_G^0(X), \quad [E] \mapsto [f^* E] \]
given by pull-back of bundles.  If \(f\)
is proper then there is a pushforward map
\[f_*: K^G_0(X) \to K^G_0(Y), \quad f_* [\cF] \mapsto \sum_{i \ge 0}
(-1)^i [R^i f_* \cF] .\]
This map is a homomorphism of \(K^0_G(Y)\)
modules by the projection formula.

The quantum product in K-theory is defined by incorporating
contributions from moduli spaces of stable maps.  Let \(X\)
be a smooth projective \(G\)-variety.
For integers \(g,n\ge 0\) and a class \(d \in H_2(X,\Z)\) let
\[ \ovl{\M}_{g,n}(X,d) = \Set{( u: C \to X, \ul{z} \in C^n ) \ |
\begin{array}{l} \#\Aut(u,\ul{z}) < \infty,  \\ g(C) = g, \ u_*[C] =
  d \end{array} } \]
denote the moduli stack of stable maps to \(X\)
with \(n\)
markings, genus \(g\),
and homology class \(d\).  The evaluation maps are denoted
\[ \ev = (\ev_1,\ldots,\ev_n): \ovl{\M}_{g,n}(X,d) \to X^n .\]
Recall that a \emph{perfect obstruction theory \(E^\bullet\)
  admitting a global resolution by vector bundles} on a stack $\M$ is
a pair $(E,\phi)$ consisting of 
an object 
of the bounded derived category of coherent sheaves on \(\M\)
that can be presented as a two term complex
\[ E = [E^{-1}\to E^{0}]\in D^{[-1,0]} (\on{Coh} (\M)) \] 
of vector bundles, together with a morphism
\[ \phi:E \to L_\M, \quad h^0(\phi) \ \text{iso}, \quad h^{-1}(\phi) \
\text{epi} \]
to the ( \(L^{ \geq -1 } \)
truncation of the) cotangent complex of \(L_\M\),
satisfying that (c.f.  \cite{bf:in}, \cite{gr:loc}) \(h^0(\phi)\)
is an isomorphism and \(h^{-1}(\phi)\)
is an epimorphism. The perfect obstruction theory defines the
\emph{virtual tangent bundle}
\[
T^\vir_\M=
\on{Def}-\on{Obs} 
\in
K^0(\M),  \quad \Def :=  [(E^{0})^{\dual}] ,  \quad \Obs := [(E^{-1})^{\dual}]  \]
as in \cite[p. 21]{coates:thesis}.  That is, the virtual tangent
bundle is the K-theoretic difference of the deformation space and the
obstruction space.  The virtual normal cone
\(C \inject E_1=(E^{-1})^{\dual}\)
induces the \emph{virtual structure sheaf } \cite{lee:qk1} as the
derived tensor product
\[
  \cO^{\vir}_\M := \cO_\M \bigotimes^L_{\cO_{E_1}} \cO_C 
\in \Ob( D \Coh(\M)), 
\]
whose class in \(K(\M)\) is 
\[
[\cO^{\vir}_\M] = \sum_{i=0}^\infty (-1)^{i} \on{Tor}^i_{\cO_{E_1}} (\cO_\M, \cO_{C})
\in K(\M) .
\]
For any class \(\alpha\in K(\M)\)
we define the \emph{virtual Euler characteristic}
\[
\chi^\vir(\M;\alpha)=\chi (\M; \alpha\otimes \cO^{\vir}_\M) \in \Z
\]
the Euler characteristic after twisting by \(\cO^\vir_\M\).
These constructions admit equivariant generalizations, so that for any
genus $g \ge 0$ and markings $n \ge 0$ with $\M = \ovl{\M}_{g,n}(X,d)$
the {\em virtual structure sheaf} \(\cO^{\vir}_\M \)
is an object in the $G$-equivariant bounded derived category for
\(\ovl{\M}_{g,n}(X,d)\)
introduced by Y.P. Lee \cite{lee:qk1}.  It defines a class
\([\cO^\vir]\)
in the equivariant \(K\)-theory
of \(\ovl{\M}_{g,n}(X,d)\).
For classes $\alpha_1,\dots, \alpha_n \in K^0_G(X)$ and a class
$\beta \in K(\ol{\M}_{g,n})$ define the equivariant \(K\)-theoretic
Gromov-Witten invariants
\begin{equation} \label{kgw}
  \langle \alpha_1,\dots, \alpha_n ; \beta \rangle_{g,n,d} := \chi_G(
  \ev_1^* \alpha_1 \otimes \ldots \otimes \ev_n^* \alpha_n \otimes f^*
  \beta \otimes [\cO^{\vir}]), \end{equation} 
  where \(\chi_G\)
  is the equivariant Euler characteristic.  In fact, because \(X\)
  is smooth, one may replace the algebraic K-cohomology group above by
  the topological equivariant K-cohomology, that is, the Grothendieck
  group of equivariant complex vector bundles on \(X\),
  as in \cite[Section 4]{lee:qk1}.  In this case the equivariant Euler
  characteristic is then replaced by a proper push-forward in
  topological K-theory.  Define {\em descendant invariants} by
\[ \langle \alpha_1 L^{d_1} ,\dots, \alpha_n L^{d_n} ; \beta \rangle_{g,n,d}
\in K^0_G(\pt)   \] 
defined by insertion $d_i$ cotangent lines $L$ at the $i$-th marking
$z_i$.

The K-theoretic Gromov-Witten invariants can be organized into a
potential as follows.  Let \(\Lambda_X \subset \Map(H_2(X),\Q)\)
denote the \emph{Novikov ring} associated to the ample line bundle
\(\LL\to X\).
The elements of the Novikov ring 
 are formal combinations
\[ \Lambda_X = \left\{ \sum_{d \in H_2(X)} c_d q^d  \right\}  \] 
such that for any \(E > 0\)
the number of coefficients \(c_d\)
with \((d , c_1(\LL)) < E\)
is finite.  Define the {\em equivariant quantum K-theory} as the
completion \eqref{completion}.  The K-theoretic {\em genus zero
  Gromov-Witten potential} with insertions is the formal function
which we write informally
\begin{equation}\label{eq:eqpot} \mu_X: QK_G^0(X) \to QK_G^0(\pt), \quad \alpha \mapsto \sum_{d
    \in H_2(X)} \sum_{n \ge 0}
  \lan \alpha,\ldots,\alpha ; 1 \ran_{0,n,d} \frac{q^d}{n!} ;
\end{equation}
what this means is that each Taylor coefficient of $\mu_X$ is
well-defined.  In the following expressions involving $\mu_X$ will be
understood in this sense.  For any element $\sigma \in QK_G^0(\pt)$ we
denote by $\partial_\sigma \mu_X $ the differentiation of $\mu_X $ in
the direction of $\sigma$.  The {\em quantum K-theory pairing} at
\(\alpha \in QK_G^0(X)\) is for $\sigma,\gamma \in QK_G^0(X)$
\begin{equation} \label{pairing} B_\alpha(\sigma,\gamma) =
  \partial_1 \partial_\sigma \partial_\gamma \mu_X(\alpha) \in QK_G^0(\pt) 
\end{equation}
where the identity in \(QK_G^0(X)\) is the structure sheaf
\(\cO_X\). This recovers the usual pairing
\(\chi(\sigma\tensor \gamma)\) when
\(\alpha=0, q=0\). Note that
the corresponding pairing in quantum cohomology is the
classical pairing; the existence of quantum corrections in
the K-theoretic pairing is due to a modification in the contraction
axioms \cite[Section 3.7]{lee:qk1}, this is an important new
feature of quantum K-theory. The {\em quantum K-theory
product} on \(QK_G^0(X)\) with bulk deformation \(\alpha\) is the
formal product defined by
\[ B_\alpha(\sigma \star_\alpha \gamma, \kappa) =
\partial_\sigma \partial_\gamma \partial_\kappa \mu_X(\alpha)
.\]
As in the product in quantum cohomology, for each choice of
\(\alpha \in QK_G^0(X)\) we obtain a formal Frobenius algebra
structure on \(QK_G^0(X)\), by an argument of Givental
\cite{giv:wdvv}.  Notably, the product does not satisfy a
divisor axiom.  However, quantum K-theory has better
properties in other respects.  For example, the small
quantum K-theory (product at \(\alpha = 0\)) is defined over
the integers, since the virtual Euler characteristics are
virtual representations.

Later we will need a slight reformulation of the
quantum-corrected inner product in \eqref{pairing}.  Let
\(\ev_{n+1}^d\) denote the restriction of the evaluation map
\[ \ev_{n+1}: \ovl{\M}_{0,n+1}(X) \to X \] 
to \(\ovl{\M}_{0,n+1}(X,d)\).  Define the (formal) {\em Maurer-Cartan map}
\begin{multline} \label{MCd}
   \mathcal{MC}_X^G: QK_G^0(X) \to QK_G^0(X) \\
  \alpha
                                         \mapsto  \sum_{n \ge 1,d \in H_2(X)} \ev_{n+2,*}^d \left( \ev_1^* \alpha
                                                   \otimes \ldots \otimes\ev_n^* \alpha  \otimes ev_{n+1}^* 1
                                                   \right)\frac{q^d}{n!}   \end{multline}
where the push-forward is defined using the virtual structure sheaf.
Then if $\ul{B}$ denotes the classical Mukai pairing we have
\label{mukai} 
\[ B_\alpha(\sigma,\gamma) = \ul{B}( D_\alpha \mathcal{MC}_X^G
(\sigma), \gamma), \quad \ul{B}^{-1} B_\alpha = D_\alpha
\mathcal{MC}_X^G,\]
where \(D_\alpha \mathcal{MC}_X^G\)
denotes the linearization of \(\mathcal{MC}_X^G\) at \(\alpha\)
\begin{equation*} \sigma \mapsto  \sum_{n \ge 1,d
    \in H_2(X)} \ev_{n+2,*}^d \left( \ev_1^*\sigma \otimes\ev_2^*
    \alpha \otimes \ldots \otimes\ev_n^* \alpha \otimes \ev_{n+1}^* 1
  \right)\frac{q^d}{n!}.
\end{equation*}

One can also consider twisted K-theoretic Gromov-Witten invariants as
in Tonita \cite{to:twisted}: Let
\begin{equation*}
\begin{tikzcd}[every arrow/.append style={-latex}]
  \ovl{\cC}_{g,n}(X,d) \arrow{r}{e} \arrow{d}{p}  & X  \\
\ovl{\M}_{g,n}(X,d). &
\end{tikzcd}	
\end{equation*}
denote the universal curve.  If \(E \to X\) is a
\(G\)-equivariant vector bundle then the index class is
defined by
\[ \Ind(E) := [Rp_* e^* E] \in K^0(\ovl{\M}_{g,n}(X,d)) .\]
Its Euler class \([\Eul(\Ind(E))]\)
is well-defined in $K_{\C^\times}^0(\ovl{\M}_{g,n}(X,d))$ after
localizing the equivariant parameter for the action of $\C^\times$ by
scalar multiplication on the fibers of $E$ at roots of unity.  
The genus $g = 0$ sum 
\begin{multline} \label{twisted} \varphi_{n,d}^E: K^0_{G \times
    \C^\times}(X)^{\otimes n} \to K^{0,\loc}_{G \times
    \C^\times}(\pt), \\ (\alpha_1,\ldots,\alpha_n) \mapsto \chi_{G
    \times \C^\times}( \ev_1^* \alpha_1 \otimes \ldots \otimes \ev_n^*
  \alpha_n \otimes [\cO^{\vir}] \otimes [\Eul(\Ind(E))])
\end{multline}
produces the {\em \(E\)-twisted quantum K-theory}, again a
Frobenius manifold.

An analog of the quantum connection in quantum K-theory was introduced
by Givental \cite{giv:wdvv}:With
$m(\alpha)(\cdot) = \alpha \star \cdot$ quantum multiplication 
 define a connection \label{nabla} 
\[ \nabla_\alpha^q = (1-{z}) \partial_\alpha +
m(\alpha) 
 \in \End(QK^0_{\C^\times}(X)) .\]
By Givental \cite{giv:wdvv} the quantum connection is flat. 
\comment{ Let
\(\{\Phi_i\}_{i=0,\dots,N}\)
be a basis for \(K(X)\)
(suppose for the moment we are speaking of topological K-theory) with
\(\Phi_0=\cO_X\)
and let \(\{t^i\}\) denote the dual coordinates.  Define a matrix
\[ T_{i j} = g_{i j} + \sum_{d \in H_2(X), n \ge 0}
\frac{q^d}{n!} \left\langle \Phi_i, t,\ldots, t, \frac{\Phi_j}{1 -
  {z} L} \right\rangle_{0,n+2,d} .\]
where \(g_{ij}=\chi(\Phi_i\otimes \Phi_j)\).
We write \(\ovl{T}_{i j}\)
for the function replacing \({z}\)
with \({z}^{-1}\).  Introduce the endomorphism-valued function
\[ T :
QK^0_{\C^\times}(X) \to \End(QK^0_{\C^\times}(X)), \quad 
B(\Phi_i, T \Phi_j) := \ovl{T}_{i j} .\]
By Givental \cite{giv:wdvv} and Lee \cite{lee:qk1} , the endomorphism
\(T\) gives a fundamental solution to the quantum differential
equation
\begin{equation} \label{fsol} 
 (1 - {z}) \partial_i T + \Phi_i \star T = 0
 .\end{equation} 
} One of the goals of this paper is to give (somewhat non-explicit)
formulas for its fundamental solutions.

\section{Quantum K-theoretic Kirwan map} 

In this section we extend the definition of the quantum Kirwan map,
defined in \cite{qkirwan} to K-theory.  Let $G$ be a complex reductive
group as in the introduction and let \(X\)
be a smooth polarized projective \(G\)-variety
with \(G\)-polarization,
that is, ample $G$-line bundle, \(\LL \to X\).
Suppose that \(G\)
acts with finite stabilizers on the semistable locus, defined as the
locus of points with non-vanishing invariant sections of some positive
power of the polarization:
\[ X^{\ss}=X^{\ss}(\LL) = \{ x \in X \ | \ \exists k > 0, s \in
H^0(X,\LL^{\otimes k})^G, \ s(x) \neq 0 \} \subset X . \]
Equivalently, suppose that every orbit $Gx \subseteq X^{\ss}$ \label{xss} for
$x \in X^{\ss}$ is closed.  Denote the stack-theoretic quotient
\[X \qu G := X^{\ss}/G\] 
which is necessarily a smooth proper Deligne-Mumford stack with
projective coarse moduli space.  By definition $X/G$ is the category
whose objects
\[ \Ob( X / G) = \{ (P \to C, u: P \to X) \} \]
are pairs consisting of principal $G$-bundles $P$ over some base $C$
and equivariant maps $u: P \to X$, and whose morphisms are the natural
commutative diagrams; $X^{\ss}/G$ is the sub-category of $X/G$ whose
objects $(P,u)$ have the property that $u$ takes values in the
semi-stable locus $X^{\ss}$.

\subsection{Affine gauged maps} 

The quantum Kirwan map is defined by push-forward (in cohomology by
integration) over moduli spaces of {\em affine gauged maps}.

\begin{definition} \label{affgauge} {\rm (Affine gauged maps)}  
  An affine gauged map to \(X/G\)
  is a datum $(C, \ul{z}, \lambda, u: C \to X/G)$ consisting of
\begin{itemize} 
\item {\rm (Curve)} a possibly-nodal projective curve \(p: C \to S\)
  of arithmetic genus $0$ over an algebraic space \(S\);
\item {\rm (Markings)} sections \( \ul{z} = (z_0,\ldots, z_n : S \to
  C) \) disjoint
  from each other and the nodes;
\item {\rm (One-form)} a section \(\lambda : C \to \P(\omega_{C/ S}
  \oplus \C)\) of the projective dualizing sheaf \(\omega_{C/S}\);
\item {\rm (Map)} a map \(u: C \to X/G\) to the quotient stack \(X/G\);
\end{itemize} 
satisfying the following conditions:
\begin{itemize} 
\item {\rm (Scalings at markings)} \(\lambda(z_0) = \infty\) and
  \(\lambda(z_i)\) is finite for \(i = 1,\ldots, n\);
\item {\rm (Monotonicity)} on any component $C_v \subset C$ on which
  \(\lambda | C_v \)
  is non-constant, \(\lambda |C_v \)
  has a single double pole, at the node $w \in C_v$ closest to
  \(z_0\);
\item {\rm (Map stability for infinity scaling)} \(u\) takes values in
  the semistable locus \(X \qu G\) on \(\lambda^{-1}(\infty)\);
\item {\rm (Bundle triviality for zero scaling)} The bundle
  $u^* (X \to X/G)$ is trivializable on \(\lambda^{-1}(0)\),
  or equivalently, \(u\)
  lifts to a map to \(X\) on \(\lambda^{-1}(0)\).
\end{itemize} 
The monotonicity assumption gives an affine structure near the double
pole, thus the term {\em affine}.  An affine gauged map given by a datum
\((u:C \to X/G, z_0,\ldots, z_n,\lambda)\)
is {\em stable} if any component $C_v$ of \(C\)
on which \(u\)
is trivializable has at least three special points, if
\(\lambda | C_v \)
is zero or infinite, or two special points, if \(\lambda | C_v \)
is finite and non-zero.  In the case that \(X \qu G\)
is only locally free, that is, has some finite but non-trivial
stabilizers we also allow orbifold twistings at the nodes of $C$ where
\(\lambda = \infty\),
as in orbifold quantum cohomology.  The {\em homology class} of an
affine gauged map is the class \(u_*[C] \in H_2^G(X,\Q)\).
\end{definition} 

We introduce the following notation for moduli stacks. 
Let \(\ovl{\M}_n^G(\bA,X)\) be the moduli stack of stable affine gauged
maps to \(X\) and \(\ovl{\M}_n^G(\bA,X,d)\) the locus of homology class
\(d\).  Each \(\ovl{\M}_n^G(\bA,X,d)\) is a proper Deligne-Mumford stack
equipped with a perfect obstruction theory.  The relative perfect
obstruction theory on \(\ovl{\M}_n^G(\bA,X) \) has complex dual to \(R p_*
e^* T_{X/G}\), where \(p,e\) are maps from the universal curve
\(\ovl{\cC}_n^G(\bA,X)\) as in the diagram
\begin{equation*}
\begin{tikzcd}[every arrow/.append style={-latex}]
  \ovl{\cC}^G_{n}(\bA,X) \arrow{r}{e} \arrow{d}{p}  & X/G  \\
\ovl{\M}^G_{n}(\bA,X).&
\end{tikzcd}	
\end{equation*}
As in the construction of Y.P. Lee \cite{lee:qk1}, the
perfect obstruction theory determines a virtual structure
sheaf \(\cO^\vir_\M \) in the bounded derived category of coherent
sheaves on \(\ovl{\M}_n^G(\bA,X)\).  It defines a class
\([\cO^\vir_\M ]\) in the rational \(K\)-theory of
\(\ovl{\M}_n^G(\bA,X)\).  Let \(\ovl{I}_{X \qu G}\) denote the
rigidified inertia stack of \(X \qu G\) and 
\[ \ev = (\ev_0,\ev_1,\ldots, \ev_n): \ovl{\M}_n^G(\bA,X)
\to \ovl{I}_{X \qu G} \times (X/G)^n \]
denote the evaluation maps at \(z_0,\ldots, z_n\).
If \(X\)
is smooth projective then the moduli stack \(\ovl{\M}^G_n(\bA,X)\)
is proper.  Properness also holds in certain other situations, such as
if \(X\)
is a vector space, \(G\)
is a torus, and the weights of \(G\)
are contained in an open half-space in $\g^\dual$. For details on the
proof of properness we refer the reader to \cite{reduc}.

The moduli stack of affine gauged maps also admits a forgetful map to
a stack of domain curves.  Denote by \(\ovl{\M}_n(\bA)\)
the moduli stack \(\ovl{\M}_n^G(\bA,X)\)
in the case that \(X\)
and \(G\)
are points, which we call the stack of {\em affine scaled curves}.
There is a forgetful morphism
\[ f: \ovl{\M}_n^G(\bA,X) \to \ovl{\M}_n(\bA) \]
defined by forgetting the morphism to $X/G$ and collapsing all
components that become unstable.

\begin{example} The moduli stack \(\ovl{\M}_2(\bA)\)
  of twice-marked affine scaled curves is isomorphic to the projective
  line via the map
\[\ovl{\M}_2(\bA) \to \P^1, \quad (C,z_0,z_1,z_2,\lambda) \mapsto
\int_{z_1}^{z_2} \lambda ;\]
more precisely, the map above identifies the locus of affine scaled
curves with irreducible domain with \(\C^\times\).  The
compactification adds the two distinguished divisors
\[D_{ \{ 1,2 \}}, D_{ \{ 1 \} , \{ 2 \}} \subset \ovl{\M}_2(\bA) \]
corresponding to loci where the markings \(z_1,z_2 \in C\)
are on the same component $C_v \subseteq C, v \in \Ver(\Gamma)$ with
zero scaling $\lambda | C_v = 0$ resp. different components
$C_{v_1}, C_{v_2} \subset C, v_1\neq v_2 \in \Ver(\Gamma)$.
\end{example}

\subsection{Affine gauged invariants and quantum Kirwan map}

K-theoretic affine gauged Gromov-Witten invariants are defined as
virtual Euler characteristics over the moduli stack of affine gauged
maps.  We introduce an equivariant version of the Novikov ring.
Denote the equivariant polarization \(\LL \to X\).
We denote the equivariant Novikov ring
\begin{equation}\label{eq:novikov}
  \Lambda_X^G =\Set{ f(q) = \sum_{d \in H_2^G(X,\Q)} c_d q^d | 
    \forall E, \# \on{Supp}^E(f) < \infty } \end{equation} 
where 
\[ \on{Supp}^E(f)  = \Set{d \in H_2^G(X) | c_d \neq 0, \lan d,c^G_1(\LL)\ran <
  E }  .\] 
The ring $\Lambda_X^G$ depends on the (class of the) polarisation
\(\LL\),
but we omit \(\LL\)
from the notation when it is clear which polarisation we are
using. Redefine
\begin{eqnarray*}
 QK_G^0(X) &=& \lim_{n \leftarrow} K_G^0(X) \otimes \Lambda^G_X/I_X^G(c)^n,
\\ QK^0(X \qu G) &=&
 \lim_{n \leftarrow}K^0(X \qu G) \otimes \Lambda^G_X/I_X^G(c)^n ;\end{eqnarray*}
in other words, from now on we work over the Novikov ring
\(\Lambda^G_X\).
Let \(\ev_0^d\)
denote the restriction of
\[ \ev_0: \ovl{\M}_n^G(\bA,X) \to \ovl{I}_{X \qu G}\]
to \(\ovl{\M}_n^G(\bA,X,d)\).
Recall the formal map $ \mathcal{MC}_{X \qu G}$ from $QK(X \qu G) $ to
$ QK(X \qu G)$ from \eqref{MCd}.  The map $\mathcal{MC}_{X \qu G}$ is
formally invertible near \(0\).
Its linearization at \(0\)
is the identity modulo higher order terms, involving positive powers
of \(q\),
hence it has a formal inverse $\mathcal{MC}_{X \qu G}^{-1}$ with
$\mathcal{MC}_{X \qu G}^{-1}(0) = 0$.

\begin{definition} The {\em quantum Kirwan map} in quantum K-theory
  with insertions $\beta_n \in K(\ol{\M}_{0,n})$ is the formal map
\begin{multline}
  \kappa_X^G: QK_G^0(X) \to QK^0(X \qu G), \\ \quad \alpha \mapsto
  \mathcal{MC}_{X \qu G}^{-1} \sum_{d \in H_2(X \qu G,\Q),n \ge 0}
  \frac{q^d}{n!} \ev_{0,*}^d (\ev_1^* \alpha \otimes \ldots \otimes
  \ldots \ev_n^* \alpha \otimes f^* \beta_n) .\end{multline}
The {\em linearized quantum Kirwan map} is obtained from the
linearization of \(\kappa_X^G\)
and correction terms arising from the quantum corrections in the inner
product:
\begin{multline} \label{Dkappa} D_\alpha \kappa_X^G: QK_G^0(X) \to QK^0(X
  \qu G), \\ \quad \sigma \mapsto (D_{\kappa_X^G(\alpha)}
  \mathcal{MC}_{X \qu G})^{-1} \sum_{d,n} \frac{q^d}{(n-1)!}
  \ev_{0,*}^d (\ev_1^* \sigma \otimes \ev_2^* \alpha \otimes \ldots
  \ev_n^* \alpha).  \end{multline}
\end{definition} 

\begin{theorem} \label{each}  Each linearization of \(\kappa_X^G\) is a \(\star\)-homomorphism:
\[D_\alpha \kappa_X^G(\sigma \star_\alpha \gamma) = D_\alpha
  \kappa_X^G(\sigma) \star_{\kappa_X^G(\alpha)} D_\alpha
  \kappa_X^G(\gamma) \]
for any \(\alpha \in QK_G^0(X)\).
\end{theorem} 

\begin{proof} The proof is a consequence of an equivalence of divisor
  classes.  As in the proof of associativity of quantum K-theory by
  Givental \cite{giv:wdvv},  consider the forgetful map
\[ f_2: \ovl{\M}^G_n(\bA,X) \to \ovl{\M}_2(\bA) \cong \P^1 \]
forgetting all but the first and second markings and scaling.  The
inverse image \(f_2^{-1}(\infty)\)
consists of configurations
\[ (u:C \to X/G, \lambda,\ul{z}) \in \Ob(\ovl{\M}_n^G(\bA,X)) \]  
where the first two incoming markings $z_1,z_2 \in C$ are on different
components of the domain $C$ is a union of boundary divisors:
\[ f_2^{-1}(\infty) = \bigcup_\Gamma \ovl{\M}_{\Gamma}^G(\bA,X) \]
where \(\Gamma\)
ranges over combinatorial type of colored tree with \(r\)
colored vertices $v_1,\ldots, v_r \in \Ver(\Gamma)$ and one
non-colored vertex $v_0 \in \Ver(\Gamma)$, with the edge labelled
\(1\)
attached to the first vertex $v_1$ and the edge labelled \(2\)
attached to the second colored vertex $v_2$.  We digress briefly
to recall that if 
\[ D = \cup_{i=1}^n D_i \] 
is a divisor with normal crossing singularities on a variety $Y$ then
the class of the structure sheaf $\cO_D$ is
\[ [\cO_D] = \sum_{I \subset \{ 1, \ldots n \}} (-1)^{|I|} [\cO_{D_I}]
\in K( Y ), \quad D_I = \bigcap_{i \in I} D_i. \]
    The corresponding property for virtual structure sheaves in
    \(\ovl{\M}_{g,n}(X)\)
    is proved \label{grammar3} by Y.P. Lee \cite[Proposition 11]{lee:qk1}.  The intersection of
    any two strata \(\ovl{\M}_{\Gamma_1}^G(\bA,X)\),
    \(\ovl{\M}_{\Gamma_2}^G(\bA,X)\)
    of codimension one is a stratum \(\ovl{\M}_{\Gamma_3}^G(\bA,X)\)
    of lower codimension; there is an exact sequence of sheaves whose
    \(i\)-th
    term is the union of structure sheaves of strata of codimension
    \(i\).
    Thus the structure sheaf of \(f_2^{-1}(\infty)\)
    is identified in \(K\)-theory
    with the alternating sum of structure sheaves
\[ [\mO_{f_2^{-1}(\infty)}] = \sum_{k_1,k_2 \ge 0}
(-1)^{k_1+k_2} \sum_{\Gamma \in \cT_\infty(k_1,k_2)} [
\cO_{\ovl{\M}_{\Gamma}^G(\bA,X,d)} ] \]
where \(\cT_\infty(k_1,k_2)\)
is the set of combinatorial types of affine scaled gauged maps with
\(k_1,k_2\)
rational curves connecting the components containing \(z_1,z_2\),
with finite and non-zero scaling, and the component with infinite
scaling containing \(z_0\).
On the other hand, the structure sheaf of \(f_2^{-1}(0)\)
is the alternating sum of structure sheaves
\[ [ \mO_{f_2^{-1}(0)}] = \sum_{k \ge 0} (-1)^k \sum_{\Gamma \in
\cT_0(k)} [ \cO_{\ovl{\M}_{\Gamma}^G(\bA,X,d)} ] \]
where \(\cT_0(k)\)
is the set of combinatorial types with \(z_1,z_2\)
on one component, \(z_0\)
on another, and these two components related by a chain of \(k\)
rational curves.  The contribution of $\cT_\infty(k)$ to the
push-forward over $f_2^{-1}(0)$ is
\[ (D_\alpha \mathcal{MC}_{X \qu G} - I)^k \sum_{d \in H_2^G(X,\Q),n
  \ge 1} \frac{q^d}{(n-1)!} \ev_{0,*}^d (\ev_1^* \sigma \otimes \ev_2^*
\alpha \otimes \ldots \ev_n^* \alpha).  \]
The inverse of the linearization of the map
\(\mathcal{MC}_{X \qu G}\),
\begin{equation} \label{contributes} 
 D_\alpha \mathcal{MC}_{X \qu G}^{-1} = (I + (D_\alpha \mathcal{MC}_{X \qu G} - I))^{-1} = \sum_{k \ge 0}
 (-1)^k (D_\alpha \mathcal{MC}_{X \qu G} - I)^k .\end{equation} 
Putting everything together and using \eqref{Dkappa} gives
\begin{multline} \nonumber
(D_{\kappa_X^G(\alpha)} \mathcal{MC}_{X \qu G})  (D_\alpha \kappa_X^G)  (\sigma \star_\alpha \gamma) =
\\ ( D_{\kappa_X^G(\alpha)} \mathcal{MC}_{X \qu G})   (( ( D_\alpha \kappa_X^G) \sigma) \star_{\kappa_X^G(\alpha)} ((D_\alpha \kappa_X^G)
  \gamma )) .\end{multline}
Since \(D_{\kappa_X^G(\alpha)} \mathcal{MC}_{X \qu G}\) is invertible, this implies the result.
\end{proof} 

\begin{remark}  {\rm (Inductive definition)} 
For classes 
\[ \aleph_0 \in QK^0(X \qu G), \quad \aleph_1,\ldots, \aleph_n \in
QK_G^0(X) \] 
denote by 
\begin{multline} \label{IGA2}
  m_{n,d}(\aleph_0,\aleph_1,\ldots,\aleph_n) \\ := \chi(
  \ovl{\M}_n^G(\bA,X,d), \ev_0^* \aleph_0 \otimes \ev_1^* \aleph_1
  \otimes \ldots \ev_n^* \aleph_n \otimes [\cO^{\vir}] ) \in
  \Z \end{multline}
the virtual Euler characteristic.  For classes
\[  \aleph_0 \in QK^0(X \qu G), \quad \alpha, \aleph_1,\ldots,
\aleph_k \in QK_G^0(X) \]
define
\[ m_{k,d}^\alpha(\aleph_0,\ldots,\aleph_k) = \sum_{n \ge 0}
  \frac{1}{n!}  m_{k+n,d}(\aleph_0,\ldots, \aleph_k,\alpha,\ldots,
  \alpha) \]
  and similarly define $ \varphi_{n,d}^{\kappa_X^G(\alpha)}$ by summing
  over all possible numbers of insertions of $\kappa_X^G(\alpha)$.
Expanding the definition of the inner product we have (in topological $K$-theory)
\begin{multline}  \label{expanding}
 \ul{B}(D_\alpha \kappa_X^G(\sigma) ,\gamma) =
 \sum_{d_0,\ldots,d_r,\aleph_1,\ldots,\aleph_r} (-1)^r q^d
 m_{2,d_0}^{\alpha}(\sigma,\aleph_1^\dual) \\ \left( \prod_{i=1}^{r-1}
 \varphi_{2,d_i}^{\kappa_X^G(\alpha)}(\aleph_i, \aleph_{i+1}^\dual) \right)
 \varphi_{2,d_r}^{\kappa_X^G(\alpha)}(\aleph_r,\gamma) \end{multline}
where the sum is over all sequences of non-negative classes
\((d_0,\ldots,d_r)\) such that \(\sum d_i = d\) and \(d_i > 0 \) for \(i > 0\)
and \(\aleph_1,\ldots,\aleph_r\) range over a basis for \(QK(X \qu G)\).
Equivalently, \(D_\alpha \kappa_X^G(\sigma)\) can be defined by the inductive
formula
\begin{multline} \label{inductive} 
 \ul{B}(D_\alpha \kappa_X^G(\sigma),\gamma) = \sum_{d > 0 } q^d
 m_{2,d}^{\alpha}(\sigma,\gamma) \\ -\sum_{\aleph, e > 0} q^e
 \ul{B}(D_\alpha \kappa_X^G(\sigma),\aleph)
 \varphi_{2,e}^{\kappa_X^G(\alpha)}(\aleph^\dual,\gamma) .\end{multline}
\end{remark} 

\begin{definition} 
  The {\em canonical bulk deformation} of \(QK(X \qu G)\)
  is the value $ \kappa_X^G(0) \in QK(X \qu G)$ of $\kappa_X^G$ at
  $0$.  
\end{definition}

Note that $\kappa_X^G$ depends on the choice of presentation of
\(X \qu G\)
as a git quotient.  We give the following criterion for the canonical
bulk deformation to vanish.

\begin{lemma} \label{vanish} Suppose that the coarse moduli spaces
  of \( \ovl{\M}^G_0(\bA,X,d)\)
  and \( \ol{\M}_{0,n}(X \qu G, d) \)
  are smooth, rational and connected with virtual fundamental sheaf
  equal to the usual structure sheaf.  Suppose further that
  $H_2^G(X) \cong H_2(X \qu G)$ with \( \ovl{\M}^G_0(\bA,X,d)\)
  non-empty iff \( \ol{\M}_{0,n}(X \qu G, d) \)
  is non-empty.  Then $\kappa_X^G(0) = 0$.
\end{lemma}

\begin{proof} Under the conditions in the Lemma, a result of
  Buch-Mihalcea \cite[Theorem 3.1]{buch:qk} implies that the
  push-forward of the structure sheaf over any moduli space of maps in
  the Lemma under any evaluation map is the structure sheaf of the
  target:
 \begin{multline} 
 \ev_{0,*}^d [\mO_{\ovl{\M}_n^G(\bA,X,d)}] 
= [\mO_{X \qu G}], \quad 
    \ev_{1,*}^d [\mO_{\ovl{\M}_{0,n}(X \qu G,d)}] = [\mO_{X\qu G}], \\ \quad
      \forall d \in H_2^G(X) \cong H_2(X \qu G), n \ge 0. \end{multline}
Thus 
\begin{eqnarray*} 
 \sum_{d} q^d \ev_{0,*}^d [\mO_{\ovl{\M}_n^G(\bA,X,d)}] 
&=& \left(\sum_{d} q^d \right) [\mO_{X \qu G}], \\
\sum_{d,n} q^d \ev_{1,*}^d  [\mO_{\ovl{\M}_{0,1}(X,d)}] 
 &=& \left(\sum_{d} q^d \right) [\mO_{X \qu G}] \end{eqnarray*}
where both sums are over $d$ such that
$\ol{\M}^G_n(\bA,X,d), \ol{\M}_{0,n}(X \qu G,d)$ are non-empty.  It
follows that
\[  \mathcal{MC}_{X \qu G}(0) = \left(\sum_{d} q^d \right) [\mO_{X \qu G}], 
\quad \kappa_X^G(0) = 0  \] 
as claimed.
\end{proof} 

Under the conditions of the lemma, one obtains a map of {\em small}
quantum K-rings given by the linearized quantum Kirwan map
\[ D_0 \kappa_X^G : QK_G^0(X) \to QK^0(X \qu G) .\]

\section{Quantum K-theory of toric quotients} 
\label{toric}

In this section we use the K-theoretic quantum Kirwan map to give a
presentation of the quantum K-theory of any smooth proper toric
Deligne-Mumford stack with projective coarse moduli space.  The case
of projective spaces was treated in Buch-Mihalcea \cite{buch:qk} and
Iritani-Milanov-Tonita \cite{iri:qkt}, and general toric varieties
were treated in Givental \cite{giv:toric}.  The toric stacks we
consider are obtained as git quotients for actions of tori on vector
spaces.  Let \(G\)
be a torus with Lie algebra $\g$.  Let \(X\)
be a vector space with a representation of \(G\)
such that the weights 
\[ \mu_i \in \g^\dual_\R, i = 1,\ldots, k \] 
of the action are contained in the interior of a half-space in
$\g^\dual_\R$.  For generic polarization, the git quotient \(X \qu G\)
is smooth.  We assume that $X \qu G$ is non-empty, and for simplicity
that the generic stabilizer is trivial.  We have
\[ QK_G^0(X) \cong QK_G^0(\pt) = R(G) \] 
where $R(G)$ denotes the Grothendieck group of finite-dimensional
representations of \(G\).
Denote by \(X_k \subset X\)
the representation given by the \(k\)-th
weight space.  For any class \(d \in H_2^G(X) \cong \g_\Z\),
define elements of $QK^0_G(X)$ by
\[ \zeta_+(d) = \prod_{\mu_j(d) \ge 0} (1 - X_j^{-1})^{\mu_j(d)}  ,
\quad  \zeta_-(d) = q^d \prod_{\mu_j(d) \leq 0} (1 -
X_j^{-1})^{-\mu_j(d)} .\]
Define the $d$-th {\em K-theoretic Batyrev element} 
\begin{equation} \label{zetad}
 \zeta(d) = \zeta_+(d) - \zeta_-(d) \in QK_G^0(\pt) \cong QK_G^0(X)
 .\end{equation} 

\begin{proposition} \label{kernelcontains} The kernel of
  \(D_0 \kappa_X^G\)
  contains the elements \(\zeta(d), d \in H_2^G(X)\).
\end{proposition}

\begin{proof} An argument using divided difference operators is given
  later in Example \ref{tdiff}; here we give a geometric proof.  The
  target \(X\)
  itself defines an element of \(K_G^0(X)\)
  via pull-back under $X \to \pt$.  The pull-back \([\ev_j^* X]\)
  is a class in \(K(\ovl{\M}_1^G(\bA,X) )\)
  for $j = 1,\ldots, n$.  Define sections
\begin{equation} \label{derivs}
 \sigma_{i,j}: \ovl{\M}_n^G(\bA,X) \to \ev_j^* X,  \quad i \ge 0 \end{equation} 
by composing the map $\ovl{\M}_1^G(\bA,X) \to \ev_1^* X$ taking the
\(i\)-th
derivative of the map \(u: C \to X/G\)
at the marking \(z_j\)
with the forgetful morphism
$\ovl{\M}_n^G(\bA,X) \to \ovl{\M}_1^G(\bA,X)$.  More precisely,
suppose that \(u: C \to X/G\)
is given by a bundle \(P \to C\)
and a section \(v: C \to P \times_G X\).
In a local trivialization near \(z_j\)
the section is given by a map \(v: \C \to X\).
Furthermore, the scaling \(\lambda\)
on \(C\)
necessarily pulls back to a non-zero scaling on \(\C\),
since there are no components of \(C\)
with zero scaling.  (There are no holomorphic curves in \(X\),
hence all curves with zero scaling are constant, and there is only one
marking, hence no constant components with three special points and
zero scaling.)  Choose a coordinate \(z\)
so that \(\lambda = \d z\).
Let \(\sigma_{i,j}([u]) \in \ev_1^* X_j \)
denote the \(i\)-th
derivative of \(v\) at \(z_j\) with respect to the coordinate \(z\).

We apply these canonical sections to the following Euler class
computation.  Each factor \(X_j\)
defines a corresponding class \([X_j]\)
in \(K_G^0(X)\);
we often omit the square brackets to simplify notation.  Define
bundles \(E_\pm\to \ol{\M}_1^G(\bA,X)\)
\[E_\pm := \bigoplus_{\pm \mu_j(d) \ge 0} \ev_1^* X_j^{\oplus
  \mu_j(d)} .\]
The Euler class of \(E_\pm\) is
\[ \Eul(E)_\pm = \bigotimes_{ \pm \mu_j(d) \ge 0} (1 - \ev_1^*
X_j)^{\otimes \mu_j(d)}  \in K(\ol{\M}_1^G(\bA,X)) .\]
Given a section \(\sigma_j\)
of \(\ev_1^* X_j\)
transverse to the zero section, there is a canonical isomorphism of
the structure sheaf $\mO_{\sigma_j^{-1}(0)}$ of \(\sigma_j^{-1}(0)\)
with 
\[  [\Eul( \ev_1^* X_j^\dual)] = [1 - \ev_1^* X_j^\dual] \in  K(
\ovl{\M}_1^G(\bA,X,d')) .\]   
This isomorphism is defined by the exact sequence
\[ 0 \to \ev_1^* X_j^\dual \to \cO \to \cO_{\sigma_j^{-1}(0)} \to 0
.\]
Extending this by direct sums, any section \( \sigma:
\ovl{\M}_1^G(\bA,X,d') \to E_\pm\) transverse to the zero section
defines an equality
\[[\cO_{\sigma^{-1}(0)}] = [\Eul(E)_\pm^\dual] \in K(
\ovl{\M}_1^G(\bA,X),d') .\]
In particular, let \(\sigma\)
denote the section of \(E_+\)
given by the derivatives
\[ \sigma_{i,j}, i = 1,\dots, d_j := \min(\mu_j(d),\mu_j(d')) \] 
defined in \eqref{derivs}.  We construct a diagram
\begin{equation*}
\begin{tikzcd}[every arrow/.append style={-latex}]
  \sigma^{-1}(0) \arrow{r}{\iota} \arrow{d}{\delta}  & \ovl{\M}_1^G(\bA,X,d')  \\
.\ovl{\M}_1^G(\bA,X,d'-d) &
\end{tikzcd}	
\end{equation*}
as follows.  The map \(\iota\)
is the inclusion.  To construct \(\delta\),
note that \(\sigma^{-1}(0) \subset \ol{\M}_1^G(\bA,X)\)
consists of maps $u$ whose \(j\)-th
component $u_j$ vanishes to order \( d_j\)
at the marking \(z_1\).
Therefore, for any \([u] \in \sigma^{-1}(0)\)
define a map of degree \( d' - d\)
by dividing by the \(j\)-th
component of \(u:C \to X/G\)
on the component of \(C\)
containing \(z_1\)
by \((z-z_1)^{d_j}\)
on the component containing \(z_1\),
to obtain a map denoted \((z - z_1)^{-d} u \).
The other components of \(C\)
all map to \(X \qu G\),
and the action of \((z - z_1)^{-d}\)
on the other components does not change the isomorphism class of
\(u\).  It follows that there is a canonical map
\begin{equation} \label{Dmap} \delta: \sigma^{-1}(0) \to \ovl{\M}_1^G(\bA,X,d'-d), \quad [u] \mapsto
[u / (z - z_1)^d] .\end{equation} 
The normal bundle to $\delta$ has Euler class the product of factors \label{factors}
$(1 - X_j^{-1})^{\min(\mu_j(-d), \mu_j(d' - d))}$ over $j$ such that
$\mu_j(-d) \ge 0$.

The remaining factors are explained by the difference in obstruction
theories.  We denote by \(p^{d'}\)
the restriction of the projection \(p\)
to maps of homology class \(d'\).
To compute the difference in classes we note that \(\delta\)
lifts to an inclusion of universal curves and (if
\(e^{d'}, e^{d' - d}\) denote the universal evaluation maps)
\begin{eqnarray*} 
 \iota^* [ R p^{d'}_{*} e^* T_{X/G}] - \delta^* [Rp^{d' - d}_* e^* T_{X/G}]
 &=& \iota^* [  R p^{d'}_*
     \bigoplus_j (\mO_{z_1} (\mu_j(d')) ] \\ && 
- \delta^* [Rp^{d' - d} \bigoplus_j \mO_{z_1}(\mu_j(d' - d)) ] \\ &=&
                                                                      \iota^*[E_+]- \rho^*[ E_-]
   .\end{eqnarray*}
Hence for any class \(\alpha_0 \in K(X \qu G)\) we obtain
\begin{multline} 
 \chi^{\vir}( \ovl{\M}_1^G(\bA,X,d') , \ev_0^* \alpha_0 \otimes \ev^*
\zeta_+(d)) \\ = \chi^{\vir}( \ovl{\M}_1^G(\bA,X,d' - d) , \ev_0^*
\alpha_0 \otimes \ev^* \zeta_-(d)) .\end{multline} 
That is, 
\[m_{1,d'}(\alpha_0,\zeta_+(d)) = m_{1,d'- d}(\alpha_0,\zeta_-(d)) .\]
By definition of the quantum Kirwan map this implies 
\[ D_0 \kappa_X^G (\zeta_+(d)) = q^d D_0 \kappa_X^G (\zeta_-(d)) .\qedhere\]
\end{proof}  

We wish to show that the elements in the lemma above generate, in a
suitable sense, the kernel of the K-theoretic quantum Kirwan map.
Define the {\em quantum \(K\)-theoretic
  Stanley-Reisner ideal} \(QKSR_X^G\)
to be the ideal in \(QK_G^0(X)\)
spanned by the K-theoretic Batyrev elements \(\zeta(d), d \in H_2(X,\Z)\)
of \eqref{zetad}.  In general there are additional elements in the
kernel of the linearized quantum Kirwan map.  In order to remove these
one must pass to a formal completion.  

\begin{definition} \label{formal}
\begin{enumerate}
\item In the case that $G$ acts freely on the semistable locus in $X$,
  for \(l = (l_1,\ldots, l_k) \in \Z_{\ge 0}^k \) define a filtration
\[ QK_G^0(X)^{\ge l} := \prod_{j=1}^k (1 - X_j^{-1})^{l_j} 
QK_G^0(X)
\subset QK_G^0(X) .\]
\item More generally suppose that $G$ acts on the semistable locus in
  $X$ with finite stabilizers $G_x, x \in X^{\ss}$ and let 
\[ \cF(X)  = \bigcup_{j=1}^k \cF_j(X) \subset \C \] 
denote the set of roots of unity 
  \[ \cF_j(X) := \{ g^{\mu_j} = \exp(2 \pi i \mu_j(\xi)) \in \C^\times
  \ | \ g = \exp(\xi) \in G_x, x \in X^{\ss} \} \]
  representing the roots of unity for the action of
  $g \in G_x, x \in X^{\ss}$ on $X_j$.   Define 
  \[ QK_G^0(X)^{\ge l} := \prod_{j=1}^k \prod_{\zeta \in \cF_j(X)} (1
  - \zeta X_j^{-1})^{l_j} . \]
\end{enumerate}
Let \(\widehat{QK}_G(X)\)
denote the completion with respect to this filtration
\[ \widehat{QK}_G(X) = \lim_{\leftarrow l} QK_G^0(X)/QK_G^0(X)^{\ge l} .\]
\end{definition}

\begin{lemma} \label{formal2} For any \(d,n\)
  the K-theoretic Gromov-Witten invariants in the definition of
  $\kappa_X^G$ vanish for
  \(\alpha_0,\ldots, \alpha_n \in QK_G^0(X)^{\ge l}\)
  for \(l_1 + \ldots + l_k \) sufficiently large. \end{lemma}

\begin{proof} 
  We apply Tonita's virtual Riemann-Roch theorem \cite{to:rr}: If
  $\M = \ol{\M}_n^G(\bA,X,d)$ embeds in a smooth global quotient
  stack and $I_\M$ its inertia stack then for any vector
  bundle $V$ \label{norigid} 
  \[ \chi^\vir(\M, V) = \int_{ [I_\M]^{\vir} } m^{-1} \Ch(V \otimes
  \cO_{I_\M}^{\vir}) \Td (T I_\M) \Ch( \Eul( \nu_\M^\dual)) \]
  where $\nu_\M$ is the normal bundle of $I_\M \to \M$ and
  $m: I_\M \to \Z$ is the order of stabilizer function.  The pullback
  $ \ev_j^* ( 1- \zeta X_j^{-1} ) $ restricts on a component of the
  inertia stack corresponding to an element $h \in G$ to the K-theory
  class $1 - \chi_j(h)^{-1} X_j^{-1}$, where $\chi_j$ is the character
  of $X_j$.  The Chern character of each
  $1 - \chi_j(h)^{-1} \zeta X_j^{-1}$ has degree at least two for
  $\zeta = \chi_j(h)$.  In this case if the sum $l_j, j\in I(h)$ of
  weights for which $\chi_j(h)$ is trivial is larger than the virtual
  dimension for the component of $I\M$ corresponding to $h$, the
  virtual integral over $I\M$ (which embeds in a smooth global
  quotient stack, see \cite[Part 3, Proposition 9.14]{qkirwan} and
  \cite[Proposition 3.1 part (d)]{wall}) \label{correctref} vanishes.
\end{proof} 

Lemma \ref{formal2} implies that the quantum Kirwan map admits a
natural extension to the formal completion:
\[ \widehat{\kappa_X^G}: \widehat{QK}_G(X) \to QK(X \qu G) .\]

\begin{theorem} (Theorem \ref{tmain} from the Introduction.)  The
  completed quantum Stanley-Reisner ideal is the kernel of the
  linearized quantum Kirwan map $D_0 \kappa_X^G$: We have an exact
  sequence
\[ 0 \to \widehat{QKSR_X^G} \to \widehat{QK_G^0(X)} \to QK(X \qu G) \to
0 .\]
\end{theorem} 

\begin{proof} In \cite[Theorem 2.6]{gw:surject} we proved a version of
  Kirwan surjectivity for the cohomological quantum Kirwan map.  The
  arguments given there hold equally well in rational topological K-theory as in
  cohomology.  We address first the surjectivity of the right arrow in
  the sequence.  By \cite[Proposition 2.9]{gw:surject} for $d$ such
  that $\mu_j(d) > 0$ for all $j$ we have
\begin{equation} \label{D0}
 D_0 \kappa_X^G \left( \prod_i^k (1 - X_j^{-1})^{s\lceil  \mu_i(d)
  \rceil} \right) = q^d [\mO_{I_{X \qu G}(\exp(d))}] + h.o.t. \end{equation}
where $h.o.t.$ denotes terms higher order in $q$.  Since divisor
intersections $[D_I ] = \cap_{i \in I} [D_i]$ generate the cohomology
$H(I_{X \qu G})$ of any $I_{X \qu G}$, the classes of their structure
sheaves $[\mO_{D_i}]$ generate the rational K-theory $K(I_{X \qu G})$.
It follows that $D_0 \widehat{\kappa_X^G}$ is surjective. 

To show exactness of the sequence, it suffices to show the equality of
dimensions
\[  \dim( \widehat{QK}_G(X)/QKSR_X^G ) = \dim( QK( X \qu G)) . \]  
We recall the argument in the case that the generic stabilizer is
trivial.  Let \(T = (\C^\times)^k/G\)
denote the residual torus acting on \(X \qu G\).
The moment polytope of \(X \qu G\) may be written
\[ \Delta_{X \qu G} = \{ \mu \in \t^\dual \ | \ (\mu, \nu_j) \ge c_j,
j = 1,\ldots, k \} \]
where \(\nu_j\) are normal vectors to the facets of \(\Delta_{X \qu G}\),
determined by the image of the standard basis vectors in \(\R^k\) in
\(\t\) under the quotient map, and \(c_j\) are constants determined by the
equivariant polarization on \(X\).  

The quantum cohomology may be identified with the Jacobian ring of a
{\em Givental potential} defined on the dual torus.  Let
$\Lambda = \exp^{-1}(1) \subset \t$ denote the coweight lattice, and
$\Lambda^\dual \subset \t^\dual$ the weight lattice.  The dual torus
and  {\em Givental potential} are 
\[ T^\dual = \t^\dual/\Lambda^\dual , \quad W: T^\dual \to
\C[q,q^{-1}], \quad y \mapsto \sum_{j=1}^k q^{c_j} y^{\nu_j} .\]
The quotient of $QK_G(X)$ by the Batyrev ideal maps to the ring
$\Jac(W)$ of functions on $\Crit(W)$ by
$ (1 - X_j^{-1}) \mapsto y^{\nu_j}$.   
Denote by \(\widehat{\Crit(W)}\)
the intersection of $\Crit(W)$ with a product $U$ of formal disks
around $q = 0, X_j \in \cF_j(X)$:
\[ \widehat{\Crit(W)} = \Crit(W) \cap U . \]
Under the identification of the Jacobian ring, $\widehat{\Crit(W)}$ is
the scheme of critical points $y(q)$ such that each $y^{\nu_j}$
approaches an element of $ \cF_j(X)$ as $q \to 0$.  This definition
differs from that in \cite{gw:surject} in that we allow in theory
critical points that converge to non-trivial roots of unity $\cF_j(X)$
in the limit $q \to 0$.

To see that this gives the same definition as in \cite{gw:surject} we
must show that there are no families of critical points that converge
to non-zero values of $y^{\nu_j}$ as $q \to 0$.  Let $y(q)$ be such a
family. Necessarily the $q$-valuation $\on{val}_q(y(q))$ of $y(q)$
lies in some face of the moment polytope.  Since the moment polytope
is simplicial, the normal vectors $\nu_j$ of facets containing
$\on{val}_q(y(q))$ cannot be linearly dependent.  Taking such $\nu_j$
as part of a basis for the Lie algebra of the torus one may write
$W(y) = y_1 + \ldots + y_k + h.o.t$ and taking partials with respect
to $y_1,\ldots, y_k$ shows that these variables must vanish.  Thus
$y(q)$ converges to zero as $q \to 0$.

The dimension of \(\widehat{\Crit(W)}\)
can be computed using the toric minimal model program \cite[Lemma
4.15]{gw:surject}: under the toric minimal model program each flip
changes the dimension of \(\widehat{\Crit(W)}\)
in the same as way as the dimension of \(QK(X \qu G)\).
\label{dimensionof}.  

Similarly for a Mori fibration one has a product formula
representing \(\dim(QK(X \qu G))\)
as the product of dimension of the base and fiber, and similarly for
the Jacobian ring.  It follows that $ \dim (\widehat{\Crit(W)})$ is
equal to $ \dim(\widehat{QK}_G(X) / QKSR_X^G) = \dim( QK(X \qu G)) .$
\end{proof}

\begin{remark}  \label{special}
\begin{enumerate} 
\item The presentation above specializes to Vezzosi-Vistoli presentation
  \cite[Theorem 6.4]{vv} by setting  \(q = 1\) in the case of smooth
projective toric varieties.  See Borisov-Horja
  \cite{bor:kth} for the case of smooth Deligne-Mumford stacks.
\item The presentation above restricts Buch-Mihalcea presentation
  \cite{buch:qk} in the case of projective spaces.  In the case of
  projective (or more generally weighted projective spaces realized as
  quotients of a vector space by a \(\C^\times\)
  action) we have \( m_{0,d}(\alpha_0) = 1\)
  for any \(d > 0 \).
  This implies \(\kappa_X^G(0) = 0\),
  by the arguments in Buch-Mihalcea \cite{buch:qk}: The moduli stack
  is non-singular, has rational singularities, the evaluation map
\[ \ev_0(d) : \ovl{\M}_0^G(\bA,X,d) \to X^{\exp(d)} \qu G \]   
is surjective and has irreducible and rational fibers.  By
\cite[Theorem 3.1]{buch:qk}, the push-forward of the structure sheaf
$\mO_{ \ovl{\M}_0^G(\bA,X,d) }$ is a multiple of the structure sheaf
on \( X^{\exp(d)} \qu G\).
It follows from the inductive formula \eqref{inductive} for
\(\kappa_X^G(0)\)
that \(\kappa_X^G(0)\) is the structure sheaf on $I_{X \qu G}$.

\item The linearized quantum Kirwan map has no quantum corrections in
  the case of circle group actions on vector spaces with positive
  weights.  To see this, note that for \(d > 0\)
  the push-forward of \(\ev_1^* (1 - X_j)\)
  is the pushforward of the structure sheaf \(\cO_{\sigma_j^{-1}(0)}\)
  to \(X \qu G\).
  By the argument in the previous item, this push-forward is equal to
  \([\cO_{X \qu G}]\).  For similar reasons, for \(d > 0 \) we have
\[ m_{2,d}( \ev_1^* (1 - X_j), [\cO_{\pt}]) =  1 .\]
The formula \eqref{inductive} then implies that
\( D_0 \kappa_X^G( 1- X_j ) \)
has no quantum corrections, hence neither does
\(D_0 \kappa_X^G(X_j)\).
It follows that in the presentation \eqref{present} the class \(X_j\)
may be taken to be the line bundle associated to the weight \(\mu_j\)
on the weighted projective space \(X \qu G = \P(\mu_1,\ldots,\mu_k)\).
It seems to us at the moment that even in the case of Fano toric
stacks one might have \(\kappa_X^G(0) \neq 0\)
and so the bulk deformation above may be non-trivial.
\end{enumerate} 
\end{remark} 

\section{K-theoretic gauged Gromov-Witten invariants} 

In this section we define gauged K-theoretic Gromov-Witten invariants
by K-theoretic integration over moduli stacks of Mundet-semistable
maps to the quotient stack, and prove an adiabatic limit Theorem
\ref{largearea} relating the invariants.

\subsection{The K-theoretic gauged potential}

In our terminology, a gauged Gromov-Witten invariant is an integral
over gauged maps, by which we mean maps to the quotient stack.
Let $C$ be a smooth projective curve.  

\begin{definition} \label{gmap} {\rm (Gauged maps)}  A gauged map  from  \(C\) to \(X/G\) consists of
\begin{itemize} 
\item {\rm (Curve)} a nodal projective curve \(\hat{C} \to S\) over an
  algebraic space \(S\);
\item {\rm (Markings)} sections \(z_0,\ldots, z_n : S \to \hat{C}\)
  disjoint from each other and the nodes;
\item {\rm (One-form)} a stable map \(\hat{C} \to C\) of homology class \([C]\);
\item {\rm (Map)} a map \(\hat{C} \to X/G\)
  to the quotient stack \(X/G\),
  corresponding to a bundle \(P \to \hat{C}\)
  and section \(u: \hat{C} \to P(X) := (P \times X)/G\)
  which we required to be pulled back from a map \(C \to BG\).
\end{itemize} 
\end{definition}

A gauged map is {\em stable} it satisfies a slope condition introduced
by Mundet \cite{mund:corr} which combines the slope conditions in
Hilbert-Mumford and Ramanathan for $G$-actions and principal
$G$-bundles respectively.  Given a gauged map
\[ (u:\hat{C} \to C \times X/G, z_0,\ldots, z_n) \]
let
\[ \sigma: C \to P/R \]
be a parabolic reduction of \(P\)
to a parabolic subgroup $R \subset P$. Let 
\[ \lambda \in \lie{l}(P)^\dual  \] 
be a central weight of the  Levi subgroup $L(P)$ of $R$.
By twisting the bundle and section by the one-parameter subgroup
$z^\lambda, z \in \C$ we obtain a family of gauged maps
\[ u_\lambda: \hat{C} \times \C^\times \to C \times X/G .\]
By Gromov compactness the limit $z \to 0$ gives rise to an {\em
  associated graded} gauged map
\[ u_\infty: \hat{C}_\infty \to X/ G \]
equipped with a canonical infinitesimal automorphism
\[ \lambda_\infty: \hat{C}_\infty \to \aut(P_\infty) \]
where \(P_\infty\)
is the \(G\)-bundle
corresponding to \(u_\infty\).
The automorphism naturally acts on the determinant line bundle
\(\det \aut(P_\infty)\)
as well as on the line bundle induced by the linearization
\(u_\infty^* (P_\infty(\cL) \to P(X))\).
The action on the first line bundle is given by a {\em Ramanathan
  weight} while the second is the {\em Hilbert-Mumford weight}
\[ 
\lambda_\infty \cdot \delta_\infty = i \mu_\lambda^R(u) 
\delta_\infty
\quad 
\lambda_\infty \cdot \ti{u}_\infty = i \mu_\lambda^{HM}(u) 
\ti{u}_\infty \] 
for points $\delta_\infty$ resp. $\ti{u}_\infty$ in the fiber
of $\det \aut(P_\infty)$ resp. 
\(u_\infty^* (P_\infty(\cL) \to P(X))\).
Let $\rho > 0$ be a real number.  The {\em Mundet weight} is
combination of the Ramanathan and Hilbert-Mumford weights with {\em
  vortex parameter} $\rho \in \R_{> 0}$
\begin{equation} \label{rhosum} \mu^M(\sigma,\lambda)= \rho
  \mu^R(\sigma,\lambda) + \mu^{HM}(\sigma,\lambda) .\end{equation}

\begin{definition} A gauged map \(u:\hat{C} \to C \times X/G\)
  is {\em Mundet semistable} if
\begin{enumerate} 
\item  $\mu^M(\sigma,\lambda)\leq 0 $
for all pairs \((\sigma,\lambda)\)
and 
\item each component $C_j$ of \(\hat{C}\)
  on which \(u\)
  is trivializable (as a bundle with section) has at least three
  special points $z_i \in C_j$. 
\end{enumerate} 
The map $u$ is {\em stable} if, in addition, there are only finitely
many automorphisms in $\Aut(u)$.
\end{definition} 

Gauged K-theoretic Gromov-Witten invariants are defined as virtual
Euler characteristics over moduli stacks of Mundet-semistable gauged
maps.  Denote by \(\ovl{\M}^G(C,X)\)
the moduli stack of Mundet semistable gauged maps.  Assume that the
semistable locus is equal to the stable locus, in which case
\(\ovl{\M}^G(C,X)\)
is a Deligne-Mumford stack with a perfect obstruction theory, proper
for fixed numerical invariants \cite{qkirwan}.   Restriction to the
sections defines an evaluation map 
\[ \ev:  \ovl{\M}^G(C,X) \to (X/G)^n  .\] 
Forgetting the map and stabilizing defines a morphism 
\[ f: \ovl{\M}^G(C,X) \to \ovl{\M}_n(C), \quad ( C, u ) \mapsto
C^{\on{st}} \]
where $ \ovl{\M}_n(C)$ is the moduli stack of stable maps to $C$ of
class $[C]$.  For classes \(\alpha_1,\ldots, \alpha_n \in K_G^0(X)\),
\( \beta_n \in \ol{\M}_n(C) \) and \(d \in H_2^G(X)\) we denote by
\begin{multline} \label{IGA}
  \tau_{X,n,d}^G(C,\alpha_1,\ldots,\alpha_n; \beta) := \\ \chi^{\vir}(
  \ovl{\M}_n^G(C,X,d), \ev_0^* \alpha_0 \otimes \ev_1^* \alpha_1
  \otimes \ldots \ev_n^* \alpha_n ) \otimes f^* \beta_n \in
  \Z \end{multline}
the virtual Euler characteristic. Define the {\em gauged
  K-theoretic Gromov-Witten potential} as the formal sum
\begin{multline} \tau_{X}^G: QK_G^0(X) \times K(\ol{\M}_n(C)) \to
  \Lambda_X^G, \\
(\alpha,\sigma) \mapsto \sum_{n \ge 0,d \in H_2^G(X,\Z)} \frac{q^d}{n!}
\tau_{X,n,d}^G(C,\alpha,\ldots,\alpha;\beta_n) .\end{multline}

In the case of domain the projective line the gauged potential can be
further localized as follows.  Let \(C = \P^1\)
be equipped with the standard \(\C^\times\)
action with fixed points \(0,\infty \in \P^1\).
Denote by \({z}\)
the equivariant parameter corresponding to the \(\C^\times\)-action.
As in \cite[Section 9]{qkirwan} let
$\ol{\M}^G_{n+1}(\C_+,X,d)^{\C^\times}$ denote the stack of
$n+1$-marked Mundet-semistable gauged maps
$P \to \P^1, u: \P^1 \to P(X)$ with the following properties: the data
$(P,u)$ are fixed up to automorphism by the $\C^\times$ action, and
with one marking at $0$ and the remaining markings mapping to
components attached to $0$, and the pair $(P,u)$ is trivializable in a
neighborhood of $\infty \in \P^1$.  The moduli space
$\ol{\M}^G_{n+1}(\C_-,X,d)^{\C^\times}$ 
\label{reversing} is defined similarly but
replacing $0$ with $\infty$ and vice-versa.  It is an observation of
Givental (in a more restrictive setting) that the $\C^\times$-fixed
locus in $\ol{\M}^G(\P^1,X,d)$ for sufficiently small vortex parameter
$\rho$ is naturally a union of fiber products
\begin{multline} \ol{\M}^G_{n_- + n_+}(\P^1,X)^{\C^\times} = \\
\bigcup_{ \stackrel{n_- + n_+ = n }{ d_- + d_+ = d}} 
\ol{\M}^G_{n_- + 1}(\C_-,X,d_-)^{\C^\times}
\times_{\ol{I}_{X \qu G}} \ol{\M}^G_{n_+ + 1}(\C_+,X,d_+ )^{\C^\times} \end{multline}
Indeed the bundle $P \to \P^1$ is given via the clutching construction 
by a transition map corresponding to an element $d \in \g$ and the map 
$u$ on $\C_\pm$ is given by an orbit of the one-parameter subgroup 
$u(z) = \exp(zd) x$ generated by $d$, for some $x \in X$.

Integration over the factors in this fiber product define {\em
  localized gauged graph potentials} \(\tau_{X,\pm}^G\)
as follows.  The stack $\M^G(\C_\pm,X,d)^{\C^\times}$ has a natural
equivariant perfect obstruction theory, as a fixed point stack in
$\M^G(\P^1, X,d)$.  The perfect obstruction theory for
$\M^G(\C_\pm,X,d)$ on the fixed locus splits in the \emph{fixed} and
\emph{moving} parts.  A perfect obstruction theory for
$\M^G(\C_\pm,X,d)^{\C^\times}$ can be taken to be the fixed part.  Let
\(N_\pm\)
denote the virtual normal complex of $\M^G(\C_\pm,X,d)^{\C^\times}$ in
$\M^G(\P^1,X,d)$.  Define
\begin{eqnarray*}
  \tau_{X,\pm}^G: QK_G^0(X) &\to& QK(X \qu G)[z^{\pm 1},{z}^{\mp 1}]], \\
  \alpha &\mapsto& \sum_{n \ge 0, d \in H_2^G(X,\Q)} \frac{q^d}{n!}
  \ev_{\infty,*}^d \frac{\ev^* \alpha^{\otimes n} }{
    \Eul(N^\dual_\pm)}.
 \end{eqnarray*}

 \begin{example} \label{Ifunction} The gauged graph potentials for
   toric quotients are $q$-hypergeometric functions described in
   Givental-Lee \cite{givlee:qk}.  Let \(G\)
   be a torus acting on a vector space \(X\)
   is a vector space with weights \(\mu_1,\ldots,\mu_k\)
   and weight spaces \(X_1,\ldots,X_k\)
   with free quotient \(X \qu G\).
   For any given class \(\phi \in H_2^G(X,\Z) \cong \g_\Z\),
   we have (omitting the classical map $K^G(X) \to K(X \qu G)$ from
   the notation)
  \begin{equation} \label{G0} \tau_{X,-}^G(0) = \sum_{d \in H_2^G(X)}
    q^d \frac{ \prod_{j=1}^k \prod_{m = -\infty}^{0} (1 - {z}^mX_j ^{-1}) }{
      \prod_{j=1}^k \prod_{m=-\infty}^{\mu_j(d)} (1 - {z}^m X_j^{-1} ) }.
\end{equation}
Note that the terms with \(X \qu_d G = \emptyset\)
contribute zero in the above sum since in this case the factor in the
numerator \( \prod_{\mu_j(d) < 0} (1 - X_j^{-1}) \) vanishes.

Arbitrary values of the gauged potential can be computed as follows,
using a result of Y.P. Lee \cite{lee:chi} on Euler characteristics on
the moduli spaces of genus zero marked curves.  Since there are no
non-constant holomorphic spheres in \(X\),
the evaluation maps \(\ev_1,\ldots, \ev_{n}\)
are equal on \(\ovl{\M}^G_n(\C_\pm,X)^{\C^\times}\).
Let 
\[ L_i \to \ovl{\M}^G_n(\C_\pm,X), \quad (L_i)_{u: C \to X/G, \lambda,
  \ul{z}} = T^\dual_{z_i} C \]
denote the cotangent line at the $i$-th marked point.  We compute the
push-pull as follows: On the component of
$\ovl{\M}^G_n(\C_\pm,X)^{\C^\times}$ corresponding to maps of degree
$d$ the pushforward is given by 
\begin{multline} \label{integral} \ev_{\infty,*}^d 
  \frac{\ev^* \alpha^{n}}{ \mp (1 - {z}^{\pm 1}) (1 - {z}^{\pm 1} L_{n+1})}
   \\ = \frac{\Psi_d(\alpha)^{\otimes n} }{ (1 - {z}^{\pm 1})^{2}} \chi
  \left(\ovl{\M}_{0,n+1},\sum_d ( L_{n+1} {z}^{\pm 1})^d \right)
  \end{multline}
  where 
\[ (\Psi_d \alpha)(g) = \alpha(gz^d)  .\]   
The integral \eqref{integral} can be computed using a result of
Y.P. Lee \cite[Equation (3)]{lee:chi} on Euler characteristics over
\(\ovl{\M}_{0,n+1}\)
(note the shift by \(1\)
in the variable \(n\) to relate to Lee's conventions):
\[ \chi \left(\ovl{\M}_{0,n+1},\sum_d ( ({z}^{\pm 1} L_{n+1})^{\otimes
    d}) \right)
= (1 - {z}^{\pm})^{n-1} .\]
This implies that for \(\alpha \in K_G^0(X)\)
\begin{equation} \label{Ifun} \tau_{X,-}^G(\alpha) = \sum_{d \in
    H_2^G(X)} q^d \exp \left( \frac{\Psi_d (\alpha)}{1 - {z}^{-1}}
  \right) \frac{ \prod_{j=1}^k \prod_{m = -\infty}^{0} (1 - X_j^{-1} {z}^m
    ) }{ \prod_{j=1}^k \prod_{m=-\infty}^{\mu_j(d)} (1 - X_j^{-1} {z}^m )
  }.
\end{equation} 
This is a version of Givental's K-theoretic {\em \(I\)-function},
see Givental-Lee \cite{givlee:qk} and Taipale \cite{taipale}.
\end{example} 

\subsection{The adiabatic limit theorem} 

In the limit that the linearization tends to infinity, the gauged
Gromov-Witten invariants are related to the Gromov-Witten invariants
of the quotient in K-theory.  Let $\C^\times$ act on $\M^G_X(\bA,X)$
via the weight $1$ action on $\bA$.   The quantum Kirwan map then has
a natural $\C^\times$-equivariant extension 
\[ \kappa_X^G: K_G^0(X) \to K^0_{\C^\times} (X \qu G) \] 
defined by push-forward using the $\C^\times$-equivariant virtual
fundamental sheaf.  The following is a K-theoretic version of a result
of Gaio-Salamon \cite{ga:gw}:

\begin{theorem}[Adiabatic Limit Theorem]
\label{largearea} 
Suppose that \(C\) is a smooth projective curve and \(X\) a polarized
projective \(G\)-variety such that stable=semistable for gauged maps
from \(C\) to \(X\) of sufficiently small \(\rho\).   Then
\begin{equation} \label{largerel} {\tau}_{X \qu G} \circ {\kappa_X^G}
  = \lim_{\rho \to 0} {\tau}_X^G: QK_G^0(X) \to
  \Lambda_X^G \end{equation}
in the following sense: For a class $\beta \in K(\ol{\M}_{n,1}(C))$
let
\[ \sum_{k=1}^l \beta^k_{\infty} \otimes \beta^k_1 \otimes \ldots
\beta^k_r , \quad \beta_0 \]
be its pullbacks to
\[ K \left( \ol{ \M}_r(C) \times \prod_{j=1}^r \ol{\M}_{|I_j|}(\C)
\right) , \quad \text{resp.}  \ K( \ol{\M}_n(C) )\]
respectively.  Then
\[ 
\sum_{I_1 \cup \ldots \cup I_r = \{ 1 ,\ldots, n \} }
\sum_{k=1}^l 
\tau_{ X \qu G}^r (\alpha,\beta_\infty^k) \circ \kappa_X^{G,|I_j|}(
\alpha, \beta_j^k) = \lim_{\rho \to 0} \tau_X^{G,n}( \alpha, \beta_0)
. \]
Similarly for the localized graph potentials (without insertions of
classes on the source moduli spaces)
\[ \tau_{X \qu G,\pm} \circ \kappa_X^G = 
\tau_{X,\pm}^G : QK_G^0(X) \to QK^0_{\C^\times}(X \qu G) .\]
\label{Jresults} 
\label{Jlargerel} \label{Jlargearea}
\end{theorem} 

\noindent In other words, the diagram 
\begin{equation} 
  \label{classdiag}
\begin{tikzcd}[every arrow/.append style={-latex}]
  QK_G^0(X) \arrow{dr}[swap]{\tau_X^G}
  \arrow{rr}{{\kappa}_X^G} & &QK^0_{\C^\times}( X \qu G)
  \arrow{dl}{\tau_{X \qu G}} \\ 
  & \Lambda_X^G &
\end{tikzcd}
\end{equation}
commutes in the limit \(\rho \to 0 \).

\begin{proof}[Sketch of Proof] The proof is similar to that in the
  cohomology case in \cite{qkirwan}: the proof only used an
  equivalence of divisor classes in the moduli stacks of scaled gauged
  maps. Let \(\ovl{\M}_{n,1}(\P^1)\)
  denote the moduli space {\em scaled} maps to \(\P^1\),
  that is, the space of maps $\phi: C \to \P^1$ of class \([\P^1]\)
  equipped with sections \(\lambda\)
  of the projectivized relative dualizing sheaf from \cite{qkirwan};
  this means that some component $C_0$ mapped isomorphically onto $\P^1$
  while the remaining components maps to points; either
  $\lambda | C_0$ is finite, in which case the remaining components
  $C_v \subset C$ have $\lambda | C_v= 0 $,  \label{CV} or $\lambda | C_0$ is
  infinite in which case there are a collection of bubble trees
  $C_1,\ldots, C_k \subset C$ attached to $C_0$ of the form described
  in \ref{affgauge}.  In particular, if there are no markings $n = 0$
  then there are no bubble components and there exists an isomorphism
  \[ \ovl{\M}_{0,1}(\P^1) \cong \P \]
  corresponding to the choice of section of the projectivized trivial
  sheaf.  Denote by
 \begin{equation} \label{f00} 
 f_{0,0}: \ovl{\M}_{0,n}(\P^1) \to \ovl{\M}_{0,0}(\P^1) \cong
 \P \end{equation} 
  the forgetful morphism forgetting the markings $z_1,\ldots, z_n$ but
  remembering the scaling $\lambda$.  We have the relation
\[ [\mO_{\cup_{\cP} D_{\cP}} ] =
[\mO_{f_{0,0}^{-1}(0) }] \in K(\ol{\M}_n(\P^1)) \]
where 
\[ D_{\cP} \cong \ol{\M}_r(\P^1) \times \prod_{i=1}^r
\ol{\M}_{i_j}(\bA) \]
is a divisor corresponding to the unordered partition
\[ \cP = \{ \cP_1,\ldots, \cP_r \}, \quad | \cP_j| = I_j \]
of the markings $\cP$ in groups of size $i_1,\ldots, i_r$.  The class
$\cup_{\cP} D_{\cP}$ is locally the union of prime divisors in a toric
variety and the standard resolution gives the equality in $K$-theory
\[ [\mO_{\cup_{\cP} D_{\cP}} ] = \sum_{\cP_1,\ldots, \cP_k} (-1)^{k-1}
  [\mO_{D_{\cP_1} \cap \ldots \cap D_{\cP_k} } ] .\]
  Two divisors $D_{\cP_1}, \ldots, D_{\cP_k}$ intersect if and only
  the partitions $\cP_1, \ldots, \cP_k$ have a common refinement (take
  a curve $(C,\ul{z})$ in the intersection and the partition
  determined by the equivalence class given by two markings are
  equivalent if they lie on the same irreducible component) and any
  type $\M_\Gamma(\P^1)$ in the intersection corresponds to some common
  refinement $\cP$.  Thus in the case of a non-empty intersection
  $D_{\cP_1} \cap \ldots \cap D_{\cP_k} $ we may assume that each
  $\cP$ refines each $\cP_j$.  

  Define a moduli space of maps with scaling as follows.  If
  $(C,\lambda,\ul{z})$ is a scaled curve and $u: C \to X/G$ a morphism
  we say that the data $(C,\lambda,\ul{z})$ is stable if either
  $\lambda | C_0$ is finite and is Mundet semistable or
  $\lambda | C_0 $ is infinite and each bubble tree is stable in the
  sense of \ref{affgauge}.  By \cite{reduc}, the moduli stack of
  scaled gauged maps $\ol{\M}_{n,1}^G(\P^1,X)$ is proper with a perfect
  obstruction theory.  Virtual Euler characteristics over
  $\ol{\M}_{n,1}^G(\P^1,X)$ define invariants
\[   K_G^0(X) \otimes K(\ol{\M}_{n,1}(\P^1)) \to \Z \] 
with the property that $\alpha \otimes [\mO_{\pt}]$ maps to
$\tau_G^X(\alpha)$.  We have a natural forgetful morphism 
\[ f: \ol{\M}_{n,1}^G(\P^1,X) \to \ol{\M}_{0,1}(\P^1) \cong \P \] 
and the inverse image of $\infty$ is the union of divisor classes
corresponding to partitions according to which markings lie on which
bubble tree.

The contributions of intersections of divisors correspond to the terms
in the Taylor expansion of the inverse of the Maurer-Cartan map, using
the {\em tree inversion formula} of Bass-Connell-Wright \cite[Theorem
4.1]{bass}.  Consider the Taylor expansion
 \begin{multline}
 \mathcal{MC}(\alpha) = \on{Id} + \sum_{k_1,\ldots, k_r \ge 0}
 \frac{ \mathcal{MC}_{(i_1,\ldots, i_r;)}
 \alpha_1^{i_1} \ldots
  \alpha_r^{i_r} } { i_1! \ldots i_r! },\\
  \mathcal{MC}_{(i_1,\ldots, i_r)} \in QK(X \qu G)
 . \end{multline}
  The tree formula for the formal inverse $\mathcal{MC}^{-1}$ of
  $\mathcal{MC}$ reads
\[ \mathcal{MC}^{-1} =\sum_{\Gamma} |\on{Aut}(\Gamma)|^{-1} \prod_{v \in \Ver(\Gamma)} (-
\mathcal{MC}_v) \]
where the sum over $\{ 1,\dots, \dim(QH(X \qu G)) \}$-labelled trees \label{ltrees}
$\Gamma$,
\[ \mathcal{MC}_v ( \sigma_1,\ldots, \sigma_{k(v)} ) = \sum_{i_1,\ldots,
  i_{k(v)}} \mathcal{MC}_{i_1(v),\ldots, i_r(v)} ( \sigma_1 , \ldots,
\sigma_{r}) \]
is the $|v|$-th Taylor coefficient in $\mathcal{MC} - \on{Id}$ for the
labels incoming the vertex $v$ considered as a symmetric polynomial in
the entries; and composition is taken on the tensor algebra using the
tree structure on $\Gamma$.  For example, for the tree corresponding
to the bracketing $(12)3$ the contribution of $\mathcal{MC}_2$ to
$\mathcal{MC}^{-1}$ is
$(- \mathcal{MC}_2) (- \mathcal{MC}_2 \otimes \on{Id})$.  See Wright
\cite{wright} for the extension of \cite{bass} to power series, and
also Kapranov-Ginzburg \cite[Theorem 3.3.9]{ginzburg}.  The argument
for the localized gauged potential is similar, by taking fixed point
components for the $\C^\times$-action on $\ol{\M}^G(\P^1, X)$.
\end{proof}

\subsection{Divided difference operators}

A result of Givental and Tonita \cite{giv:hrr} shows the existence of
a difference module structure on quantum K-theory for arbitrary target
which gives rise to relations in the quantum K-theory.  We follow the
treatment in \cite[Section 2.5, esp. 2.10-2.11]{iri:qkt}.  Let
\[ \lambda_1,\ldots, \lambda_k \in K(X) \] 
be classes of nef line bundles corresponding to a basis of \(H^2(X,\Z)\)
and $m(\lambda_i)$ the corresponding endomorphisms of $QK(X)$ given by
multiplication.  Define endomorphisms
\[ \cE_i = \tau( m(\lambda_i)^{-1} {z}^{q_i \partial_{q_i}} \tau^{-1} )  \in
\End(QK^0_{\C^\times}(X)) \]
where $\tau$ is a fundamental solution to $\nabla_\alpha^q \tau = 0$.
Define
\[ \cE_{i,\com} = \cE_i |_{{z} = 1} \in \End(QK(X)) .\]
Then any difference operator annihilating the \(J\)-function
defines a relation in the quantum K-theory: Following \cite[Remark
2.11]{iri:qkt} define
\[ \ti{\tau}_{X,\pm}^G = \left( \prod_{i=1}^r \lambda_i^{ - \ln(q_i)/\ln({z})}
\right) \tau_{X,\pm}^G .\]
The results in \cite[Section 2]{iri:qkt} give relations corresponding
to operators annihilating the fundamental solution.  Working with
topological $K$-theory we define coordinates $t_1,\ldots, t_r$ by
  \[ \aleph = t_1 \aleph_1 + \ldots t_r \aleph_r, \quad
  \aleph_1,\ldots, \aleph_r \text{\ a basis of \ } QK(X \qu G) .\]

\begin{theorem} \label{ddrel} For any divided-difference operator
\[  \Box \in
 \Q[{z},{z}^{-1}][[q,t]] \langle {z}^{q_i \partial_{q_i}}, (1 -
   {z}^{-1}) \partial_\alpha , \alpha \in QH(X \qu G) \rangle \] 
   (where angle brackets denote the sub-ring of differential operators
   generated by these symbols) we have \label{subring}
\[ \Box({z},{q},t,{z}^{q_i \partial_{q_i}}, (1 -
 {z}^{-1} )\partial_\alpha) \ti{\tau}_{X,-}^G = 0 \implies
 \Box({z},{q},t,\mE_i, \nabla_\alpha^{z}) 1 = 0 \]
 and setting \({z} = 1\)
 gives rise to the relation involving quantum multiplication
 $m(\alpha)$ by $\alpha$
 \[ \Box(1, q,t, \cE_{i,\com}, m(\alpha)) = 0 \]
 in the quantum K-theory $QK(X \qu G)$.
\end{theorem} 

\begin{example} \label{tdiff} {\rm (Toric varieties)} In the toric
  case, we recover some of the relations on quantum K-theory from
  Theorem \ref{tmain} as follows.  The operators
  \[ \mathcal{D}_{i,k} := 
1 - X_i ({z}^{q_i \partial_{q_i}} + 
z^{-k} (1 -
  {z}^{-1})\partial_i ) 
\]
satisfy  
\begin{equation}  \label{relation}
 \left(\prod_{\mu_i(d) \ge 0} \prod_{k=0}^{\mu_i(d)-1} \mathcal{D}_{i,k}
 - q^d 
\prod_{\mu_i(d) < 0} \prod_{i=1}^{-\mu_i(d)} 
\mathcal{D}_{i,k} 
\right) \hat{\tau}_{X,-}^G = 0 .\end{equation}
Using the expression \eqref{Ifun} and $\Psi_d(X_i) = X_i z^{\mu_i(d)}$
we have
\begin{multline} \label{using}
(1 - {z})\partial_i 
 \tau_{X,-}^G \\ = \sum_{d \in H_2^G(X)} q^d 
X_i z^{\mu_i(d)} \exp \left( \frac{ \Psi_d(\alpha)}{1 - {z}^{-1}}
\right) \frac{ \prod_{j=1}^k \prod_{m = -\infty}^{0} (1 - X_j^{-1} {z}^m ) 
}{ \prod_{j=1}^k \prod_{m=-\infty}^{\mu_j(d)} (1 - X_j^{-1} {z}^m ) }. \end{multline}
Hence 
\begin{multline} 
  \prod_{\mu_i(d) \ge 0} \prod_{k=0}^{\mu_i(d)-1} \mathcal{D}_{i,k}
  \tau_{X,-}^G(\alpha,q,z) \\ = \sum_{d \in H_2^G(X)} q^d
  \prod_{k=0}^{\mu_i(d)-1} (1 - X_i z^{\mu_i(d) - k}) \exp \left(
    \frac{ \Psi_d(\alpha)}{1 - {z}^{-1}} \right) \\ \frac{
    \prod_{j=1}^k \prod_{m = -\infty}^{0} (1 - X_j^{-1} {z}^m ) }{
    \prod_{j=1}^k \prod_{m=-\infty}^{\mu_j(d)} (1 - X_j^{-1} {z}^m ) }
  .\end{multline}
This equals 
\[ 
q^d \prod_{\mu_i(d) < 0} \prod_{k=1}^{- \mu_i(d)} \mathcal{D}_{i,k}
 \tau_{X,-}^G  \]
 hence the relation \eqref{relation}.  By \ref{ddrel} one obtains
 quantum Stanley-Reisner relations as in \eqref{qsrr}.\end{example}

\section{Wall-crossing in K-theory}

In this section we recall results on the dependence of the
\(K\)-theoretic
invariants of the quotient on the choice of polarization due to
Kalkman \cite{ka:co}, in the case of cohomology, and Metzler
\cite{metzler}, in the case of K-theory.
 
\subsection{The master space and its fixed point loci.} % (fold)
\label{ss:masterspace}

The geometric invariant theory quotients for two different
polarizations may be written as quotients by a circle group action on
a {\em master space}.  Let \(\LL_\pm \to X\)
two \(G\)-polarizations of \(X\).  Define
\begin{equation}
  \label{eq:tpolar}
\LL_t := \LL_-^{(1-t)/2}\otimes
\LL_+^{(1+t)/2}, \quad t \in (-1,1)\cap \Q
\end{equation}
the one-parameter family of rational polarizations 
given by interpolation.  

\begin{definition} The wall-crossing datum $(\LL_-,\LL_+)$ is
  \emph{simple} if the following conditions are satisfied:
  \begin{enumerate}
  \item The only \emph{singular value} is \(t=0\),
    this is, 
\[ X^{\ss}(\pm):=X^{\ss}(\LL_t) , \quad 0<\pm t \leq 1\] 
is constant.  We assume that stable equals semi-stable for
\(t\neq 0\),
that is, there are no positive-dimensional stabilizer subgroups of
points in the semi-stable locus.
\item The {\em strictly semistable set}
  \(X^{\ss}(\LL_0) \setminus (X^{\ss}(+)\cup X^{\ss}(-)) \)
  is connected.
  \item The only infinite stabiliser subgroup is a circle group: 
    \[ G_x \cong \C^\times, \quad \forall x\in X^{\ss}(\LL_0)
    \setminus (X^{\ss}(+)\cup X^{\ss}(-)), \quad \dim(G_x) > 0 . \]
\end{enumerate}
\end{definition} 

Given a simple wall-crossing, choose a one-parameter subgroup
\(\lambda:\Ct\to G\),
a connected component \(Z_\lambda\)
of the semi-stable fixed points \(X^{\ss}(\LL_0)^{\lambda}\).

\begin{definition}\label{def:master}
  The {\em master space} introduced Dolgachev-Hu \cite{do:va} and
  Thaddeus \cite{th:fl} is defined as follows.  The projectivization
  \(\P(\LL_- \oplus \LL_+) \to X \)
  of the direct sum \(\LL_- \oplus \LL_+ \to X\)
  is itself a $G$-variety, and has a natural polarization given by the
  relative hyperplane bundle \(\mO_{\P(\LL_- \oplus \LL_+)}(1)\)
  of the fibers, with stalks
\begin{equation} \label{relhyp} \mO_{\P(\LL_- \oplus
  \LL_+)}(1)_{[l_-,l_+]} = \on{span}(l_- + l_+)^\dual
  .\end{equation} 
Let \(\pi: \P(\LL_- \oplus \LL_+) \to X\) denote the
projection.  The group \(\C^\times\) acts on \(\P(\LL_- \oplus
\LL_+)\) by rotating the fiber \(w [l_-,l_+] = [l_-, wl_+]\).
The space of sections of \(\mO_{\P(\LL_- \oplus \LL_+)}(k)\)
has a decomposition under the natural \(\C^\times\)-action
with eigenspaces given by the sections of 
\[ \pi^* \LL_-^{k_-} \otimes \pi^* \LL_+^{k_+}, 
\quad k_- + k_+ = k, k_\pm \ge 0 .\]
The \(G\)-semistable
locus in \(\P(\LL_- \oplus \LL_+)\)
is the union of loci of invariant eigensections and so the union of
\(\pi^{-1}(X^{\ss,t})\)
where \(X^{\ss,t} \subset X\)
is the semistable locus for
\[ \LL_t=\LL_-^{(1-t)/2}\otimes \LL_+^{(1+t)/2}, t \in [-1,1] .\]  
Let
\[ \ti{X} := \P(\LL_- \oplus \LL_+) \qu G \] 
denote the geometric invariant theory quotient, by which we
mean the stack-theoretic quotient of the semistable locus.
The assumption on the action of the stabilizers implies that
the action of \(G\) on the semistable locus in \(\P(\LL_-
\oplus \LL_+)\) is locally free, so that stable=semistable
for \(\P(\LL_- \oplus \LL_+)\).  It follows that \(\ti{X}\) is a
proper smooth Deligne-Mumford stack.  The quotient \(\ti{X}\)
contains the quotients of \(\P(\LL_\pm) \cong X\) with respect
to the polarizations \(\LL_\pm\), that is, \(X \qu_\pm G\).
Moreover the quotient \(\ti{X}^{\ss}(t)/\Ct\) with respect to
the \(\Ct\)-linearisation \(\cO(t)\) is isomorphic to
\(X^{\ss}(\LL_t)/ G\).
This ends the definition.
\end{definition} 

Next we recall from \cite{wall} the fixed point sets for the circle
action on the master space.

\begin{lemma} The fixed point set $\ti{X}^{\C^\times}$ is the union of
  the quotients $X \qu_\pm G$ and the locus in $X^{\C^\times,\ss}$ in
  $X^{\C^\times} $ that is semistable for some $t \in (-1,1)$.  The
  normal bundle of $X^{\C^\times,\ss}$ in $\ti{X}$ is isomorphic to
  the normal bundle in $X$, while the normal bundles of
  $X \qu_{\pm} G$ are isomorphic to $\cL_\pm^{\pm 1}$.
\end{lemma}

\begin{proof} 
  Any fixed point is a pair \([l_-,l_+]\)
  with a positive dimensional stabiliser under the action of
  \(\C^\times\).
  Necessarily the points with \(l_- = 0, l_+ = 0\)
  are fixed and they correspond to the quotients \(X\qu_{\pm} G\).
  However there are other fixed points when \(l_-,l_+\)
  are both non-zero when the projection to \(X\)
  is fixed by a one-parameter subgroup, as we now explain.  For any
  \(\zeta \in \g\),
  we denote by \(G_\zeta \subset G\)
  the stabiliser of the line \(\C\zeta\)
  under the adjoint action of \(G\).
  If \([l] \in \ti{X}^{\C^\times}\),
  with \(l \in \P(\LL_- \oplus \LL_+)\)
  then \([l] \in \ti{X}^\xi\),
  where \(\xi\)
  is a generator of the Lie algebra of \(\C^\times\)
  and \(\ti{X}^\xi\)
  is the zero set of the vector field \(\xi_{\ti{X}}\)
  generated by \(\xi\).
  Since \(\ti{X}\)
  is the quotient of \(\P(\LL_- \oplus \LL_+)\)
  by \(G\),
  if \(\xi_L\)
  denotes the vector field on \(\P(\LL_- \oplus \LL_+)\)
  generated by \(\xi\)
  then \( \xi_L(l) = \zeta_L(l)\)
  for some \(\zeta \in \g\).
  Since \(G\)
  acts locally freely \(\zeta\)
  must be unique.  Integrating gives \(z\cdot l = z^\zeta l\)
  for all \(z \in \C^\times\).
  By uniqueness of \(\zeta\),
  this holds for every point in the component of \(\ti{X}^\xi\)
  containing \([l]\).
  Thus for any fixed point \( \ti{x} \in \ti{X}^{\C^\times}\)
  with \(\ti{x} = [l]\)
  for some \(l \in \P(\LL_- \oplus \LL_+)\),
  there exists \(\zeta \in \g\) such that
\[ \forall z \in \C^\times, \quad  z\cdot l = z^\zeta l, \]
where \(z\mapsto z^\zeta\) is the one-parameter subgroup generated by
\(\zeta\).
By the definition of semistability, the argument above and 
our wall-crossing assumption, any fixed point \( \ti{x} \in
\ti{X}^{\C^\times}\) is in the fibre over \(x \in X\) that
is \(0\)-semistable  and has stabilizer given by 
the one-parameter subgroup generated by 
\(\zeta \in \g\), that is, 
the weight of the one-parameter subgroup generated by
\(\zeta\) on \((\LL_t)_x\)
vanishes:%
\[ 
 z^\zeta l = l, \quad \forall l \in (\LL_t)_x.
\]
Denote by \(X^\zeta \qu_0 (G_\zeta/ \C^\times_\zeta)\)
the quotient \([X^{\ss}(\LL_0)^\zeta/(G_\zeta/\C^\times_\zeta)]\).
Conversely, taking any lift gives a morphism
\[ \iota_\zeta: X^\zeta \qu_0 (G_\zeta/ \C^\times_\zeta) 
\to \ti{X}^{\C^\times}.
\]
The argument above shows that any $(l_-,l_+)$ with both non-zero is in
the image 
of some $\iota_\zeta$.
The normal bundle \(\nu_{\ti{X}^{\Ct}}\) of
\(\ti{X}^{\C^\times}\) in \(\ti{X}\) restricted to the image of
\(\iota_\zeta\) is isomorphic to the quotient of the normal
bundle of \(G \times_{G_\zeta} \P(\LL_- \oplus \LL_+)^{\xi
+ \zeta}\) by \(G\), which in turn is isomorphic via
projection to the quotient of the normal bundle of
\(\P(\LL_- \oplus \LL_+)^{\xi + \zeta}\) by \(\g/\g_\zeta\),
and then by the induced action of the smaller group
\(G_\zeta/\C^\times_\zeta\). 
\(G_\zeta/\C^\times_\zeta\)
acts canonically on the normal bundle 
\(\nu_{X^\zeta}/ (\g/\g_\zeta)\) and
induces a bundle \(( \nu_{X^\zeta}/ (\g/\g_\zeta)) \qu
 (G_\zeta/\C^\times_\zeta)\) over \(X^\zeta \qu_0
  (G_\zeta/C^\times_\zeta)\).
  The pull-back of the normal bundle
  \(\iota_\zeta^*( \nu_{\ti{X}^{\C^\times}})\)
  is canonically isomorphic to the quotient of
  \(\nu_{X^\zeta} / (\g/\g_\zeta)\)
  by \(G_\zeta / \C^\times_\zeta\),
  by an isomorphism that intertwines the action of \(\C^\times_\zeta\)
  on
  \((\nu_{X^\zeta} / (\g/\g_\zeta) ) \qu (G_\zeta / \C^\times_\zeta)\)
  with the action of \(\C^\times\)
  on \(\nu_{\ti{X}^{\C^\times}}\).
  The final claim is left to the reader.
\end{proof}

\begin{definition}  We introduce the following notation.
	\label{def:fp}
Denote by \(X^{\zeta,0}
\subset X^\zeta\) the fixed point loci
\(X^{\ss}(\LL_0)^{\zeta}\) that are \(0\)-semistable.  Denote by \(\nu_{X^{\zeta,0}}\) the \(\Ct_\zeta\)-equivariant normal bundle of
\(X^{\zeta,0}\) modulo \(\g/\g_\zeta\), quotiented by
\(G_\zeta/\C^\times_\zeta\). 	
\end{definition}

A natural collection of K-theory classes on the master space is
obtained by pull back.  The projection
\(\P(\LL_- \oplus \LL_+) \to X\)
is \(G\)-equivariant
and \(\C^\times\)-invariant
by construction.  Consider the canonical map
\begin{equation}\label{eq:deltamap}
\delta: K_G^0(X) \to K^0_{\C^\times}(\ti{X}) .
\end{equation}
obtained by composition of the natural pull-back
\[
K_G^0(X) \to K^0_{G \times \C^\times}( \P(\LL_- \oplus \LL_+))
\] 
with the Kirwan map
\[
K^0_{G \times \C^\times}( \P(\LL_- \oplus \LL_+)) \to
K_{\Ct}(\P(\LL_- \oplus \LL_+)\qu G)=K^0_{\C^\times}(\ti{X}).
\]
The composition of \(\delta\) with the Kirwan map
\[
  \ti{\kappa}_{X,t}^{\C^\times}: K^0_{\C^\times}(\ti{X}) 
  \to K(\ti{X} \qu_t \C^\times) = K(X \qu_t G)
\]
agrees with the pull-back to the \(\LL_t\)-semistable
locus \(X^{\ss}(\LL_{t})\).   It follows that 
\[ 
  \ti{\kappa}_{X,t}^{\C^\times} \circ \delta = 
  \kappa_{X,t}^G: K_G^0(X) \to K(X
\qu_t G) ,
\] 
is the Kirwan map for the geometric invariant theory  quotient of \(X\) with
respect to the polarisation \(\LL_t.\)
In particular, \(\delta(\alpha) \in
K^0_{\C^\times}(\ti{X})\) 
restricts
to \(\kappa^G_{X,\pm} \alpha\) on the two distinguished 
fixed point
components \(X \qu_\pm G \subset \ti{X}^{\C^\times}\).

The restrictions of these classes to fixed point sets are described as
follows.  After passing to a finite cover, we may assume that
$G_\zeta$ splits as the product
$G_\zeta/\C^\times_\zeta \times \C^\times_\zeta$.  For any
\(\alpha \in K_G^0(X)\),
the pull-back of
\(\ti{\kappa}_{X,0}^{\Ct}|_{\ti{X}^{\C^\times}}(\alpha)\)
under \(\iota_\zeta\)
is equal to the image of \(\alpha\)
under the restriction map
\(K_G^0(X) \to K_{\C^\times_\zeta}(X^\zeta \qu_0
(G_\zeta/\C^\times_\zeta))\).
Indeed \(TX|_{X^\zeta}\)
is the quotient of \(T\P(\LL_- \oplus \LL_+) |_{\pi^{-1}(X^\zeta)}\)
by \(\C^\times\).
Hence the action of \(G_\zeta\)
on \(T\P(\LL_- \oplus \LL_+)|_{\pi^{-1}(X^\zeta)}\)
induces an action of \(G_\zeta/\C^\times\)
on \(TX |_{X^\zeta}\).
After identifying \(\C^\times \cong \C^\times_\zeta\),
we have that \(\iota_\zeta \circ \ti{\kappa}_{X,0}^{\Ct}\)
is the pullback to \(X^\zeta \qu_0 (G_\zeta/\C^\times_\zeta)\).
For \(\alpha\in K_G^0(X)\) we will denote by
\[ \alpha_0 \in K_G( X^\zeta \qu_0
(G_\zeta/\C^\times_\zeta) ) \] 
its restriction.

\subsection{The Atiyah-Segal localization formula}
\label{ss:euler}

In this section we review the Atiyah-Segal localization formula in
equivariant K-theory \cite{as}, which expresses Euler characteristics
as a sum over fixed point loci.  Recall the definition of the Euler
class in $K$-theory.  Suppose \(Y\)
is a smooth variety. For a vector bundle $E \to Y$ and a formal
variable \(q\) denote the graded exterior power by
\begin{equation} \label{bigwedge} {\textstyle\bigwedge_q} E=
  \sum_{k=0}^\infty q^k {\textstyle\bigwedge^k}E \in K^0(Y)[q].
\end{equation} 
The K-theoretic Euler class of \(E\) is defined by 
\[
\Eul(E)=\textstyle \bigwedge_{-1} E^\vee= 1-E^{\vee} + \bigwedge^2
E^{\vee} - \dots \in K^0(Y) .
\]
Suppose now that \(T=\C^\times\)
acts trivially on \(Y\),
and that \(a=\C_{(1)}\)
is the weight 1 one-dimensional representation. Thus the equivariant
K-theory of \(Y\) is given by
\[ K_T^0(Y)=K^0(Y)\otimes K_T^0(\pt)=K^0(Y)\otimes\Z[z,z^{-1}] . \] 
If \(E\)
is a coherent \(T\)-equivariant
sheaf, its decomposition into isotypical components will be denoted
\begin{equation}\label{eq:isotypical1}
  E=\bigoplus_{i=1}^k z^{\mu_i} E_i
\end{equation}
where \(\mu_i\in \Z\) is the weight of the action on \(E_i\). 
The K-theoretical equivariant Euler class of \(E\) is given by 
\[
  \Eul_{T} (E)=\prod_i  \Eul (z^{\mu_i} E_i) \in K_T(Y) .
\]

The localization formula involves an integral over fixed point
components with insertion of the inverted Euler class.  Suppose that
\(T=\Ct\)
acts on \(X\)
(non-trivially) and with fixed point set \(X^T\).
The previous paragraph discussion applies to the \(T\)-equivariant
normal bundle \(\nu_{X^T}\), that is, the Euler class \label{eulerclass}
\[ \Eul_{T} (\nu_{X^T})\in K_T(X^T)=K(X^T)[z,z^{-1}] \]
has been formally inverted through localisation.  Denote by
\(K^{\loc}_T(X)\),
the equivariant K-theory ring localized at the ideal of $R(T)$
generated by the Euler classes of representations $\C_{\mu_i}$ with
weights $\mu_i, i =1,\ldots, k$ (or more alternatively, one could
localize at the roots of unity in the equivariant
parameter).\footnote{In the orbifold case, one must localize at the
  roots of unity that appear in the denominators in the Riemann-Roch
  formula as in Tonita \cite{to:rr}.}  The \emph{K-theoretic
  localisation formula} of Atiyah-Segal \cite{as}, see also
\cite[Chapter 5]{chriss:rep}) states that in \(K^{\loc}_T(X)\)
\begin{equation}\label{eq:loc}
  [ \cO_X ]  = [ \iota_* \left( \cO_{X^T} \otimes
    \Eul_T (\nu_{X^T}^\dual)^{-1}\right) ] 
  = \left[ \sum_{F\to X^{T}}  \cO_{F} \otimes
    \Eul_T(\nu_F^\dual)^{-1} \right] 
\end{equation}
for the inclusion \(\iota:X^T \to X\).
Here the sum runs over all fixed components $F$ of $X^T$. In terms of
Euler characteristics we have
\[
  \chi(X; \cF) = \sum_{F \to X^T}\chi(F; \cF\otimes
  \Eul_T(\nu_F^\dual)^{-1}) \in R(T).
\]

The Atiyah-Segal localization formula implies a wall-crossing formula
for K-theoretic integrals under variation of geometric invariant
theory quotient due to Metzler \cite{metzler}.  Metzler's formula uses
the following notion of {\em residue}, related to a power series
expansion in a localized ring.  For any $T$-equivariant locally free
sheaf \( \cF\) \label{scriptF}
of rank \(r\) on a variety \(Y\) with trivial $T$-action the class
\[ \cF_0=\cF-r \cO_Y \in K_T(Y) \]
is nilpotent (c.f. \cite[Prop. 5.9.5]{chriss:rep}).  By taking
exterior powers with notation as in \eqref{bigwedge} we have
\[ {\textstyle \bigwedge_q} \cF = {\textstyle \bigwedge_q} (\cF_0 +r
\cO_Y) = (1+q)^{r} {\textstyle\bigwedge_{\frac{q}{1+q}}} \cF_0 \in
K_T(Y)[q].
\]
Using this formula on the weight \(\mu_i\) bundle
\(E_i\) gives 
\[
\Eul_{T}(E_i)= (1-z^{-\mu_i})^{r_i} (1 + N_i)
\]
where \(r_i=\rank E_i\),
and \(N_i \in K(Y)\otimes\Z[z,z^{-1}]\)
is a combination of nilpotent elements in \(K(Y)\)
whose coefficients are monomials of the form
\(\frac{z^{-\mu_i}}{1-z^{-\mu_i}}.\) Thus 
\begin{equation}
	\label{eq:eulerexp}	
        \Eul_{T}(E)= \prod_i (1-z^{-\mu_i})^{r_i}  (1 + N_i),
\end{equation}
where 
\begin{equation}\label{eq:sumnilpotents}
  N_i =\sum N_{k,i} s_k(z) \in K(Y)\otimes\Z[[z,z^{-1}]].
\end{equation}
The equation \eqref{eq:sumnilpotents} is a finite sum where
\(N_{k,i}\)
is a nilpotent element in $K(Y)$ and \(s_k(z)\)
is a rational function in \(z\)
that has no pole at \(z=0\)
nor \(z=\infty\).
Thus $\Eul_T(E)^{-1}$ has only poles at roots of unity.  In particular
the Euler classes can formally be inverted, since the leading term is
invertible.

\begin{example}\label{eg:eulerl}   We give the following example of
  formal power series associated to inverted Euler classes.  
  Let \(L \to Y\) be a \(T\)-equivariant line bundle
of weight \(\mu=\pm 1\).  Then 
\[  \Eul_{T} (L) =
1-z^{\mp 1}L^{\vee} \in K_T(Y).  \]  
This expression can be expanded, using the nilpotent element
\(L_0=L^{\vee}-1\), as
\[
  1-z^{\mp 1}L^{\vee} = 1-z^{\mp 1}(1+L_0)=
  (1-z^{ \mp 1})\left(1+\frac{z^{\mp 1}L_0}{1-z^{\mp
  1}}\right). 
\]
This ends the example.
\end{example}

The residue of a K-theoretic class is a difference between the
residues of the characters at zero and infinity.  For any class
\(\alpha\in K^{\loc}_T(Y)\) there exist unique expansions
\[
\alpha = \sum_{n \geq 0} \alpha_{0,n} z^n\in K(Y)[z^{-1},z]], \ \
\alpha_{0,n}\in K(Y)
\]
and
\[
  \alpha = \sum_{n\geq0} \alpha_{\infty,n} z^{-n}\in
  K(Y)[[z^{-1},z], \ \ \alpha_{\infty,n}\in
  K(Y).
\]
The \emph{residue} is the map
\begin{equation} \label{residuedef}
        \Resid: K^{\loc}_T(Y) \to K(Y); \ \alpha\mapsto
        \alpha_{\infty,0}- \alpha_{0,0}
	\end{equation}
	assigning the difference of the constant coefficients in the 
	power series expansions above.

        \begin{lemma}\label{rk:residue}
  Let \(\alpha\in K^{\loc}_T(Y)\) be a class in the
  image of the inclusion \(K_T(Y)\to K^{\loc}_T(Y)\).
  Then \(\Resid
  \alpha=0\).
\end{lemma}

\begin{proof} In this case the coefficients
  $\alpha_{0,i} = \alpha_{\infty,i}$ are the coefficients of $z^i$ in
  $\alpha$, hence in particular $\alpha_{0,0} =
  \alpha_{\infty,0}$.  \end{proof}

\begin{example} \label{firstex} Let $Y$ be a point and $E =
  \C_{(1)}$
  the representation with weight one.  The inverted Euler class of $E$
  is defined formally by the associated series
\[
 \Eul(\C_{(1)})^{-1}=(1-z^{-1})^{-1}
  =1 + z^{-1} +
  z^{-2}+\dots \] 
\[ 
 \Eul(\C_{(1)})^{-1}=(1-z^{-1})^{-1}=  \frac{-1}{z^{-1}(1-z)}=-z- z^2 - z^{3} +\dots
.\]
It follows that the residue of the inverted Euler class is 
\[
\Resid (\Eul (\C_{(1)})) = 1-0 =1 \in K(\pt) \cong \Z.
\]
Similarly 
\[
\Resid (\Eul (\C_{(-1)})) = 0-1 =-1 \in K(\pt) \cong \Z.
\]
More generally for any line bundle \(L\) of weight \(\pm
1\), by Example  \ref{eg:eulerl} we have 
\begin{equation}\label{eq:line1}
  \Resid \Eul(L)^{-1} = \Resid
  \left(\frac{1}{1-L^{\vee}}\right)=\pm 1.
\end{equation}
\end{example}

\begin{example}\label{ex:0res}
  Suppose that \(T\)
  acts trivially on \(Y\)
  and that \(E\to Y\)
  is a bundle with isotypical decomposition as in
  \eqref{eq:isotypical1}. Let
        \[ S_+,S_- \subset \{ 1, \ldots, k \} \]
denote the index sets for positive and negative weights \(\mu_i\)
respectively.  Using \eqref{eq:eulerexp}, we have
	\[
	  \Eul_{T}(E)^{-1}= \left(\prod_{i\in S_+}
	  (1-z^{-\mu_i})^{r_i}
	  (1+E_{+,i}) \prod_{i\in S_-} (1-z^{-\mu_i})^{r_i}
	  (1+E_{-,i})\right)^{-1}
	\]
	where both \(E_\pm \in K_T^{\on{loc}} (X)\)
        are as in Equation \eqref{eq:sumnilpotents}.
	To compute the residue, we may ignore the \(N_i\)
        terms. In this case the lowest order \label{lowestorder} terms
        in the power series expansions in both \(z\)
        and \(z^{-1}\)
        have degree given by \(\prod_{i\in S_+} \mu_i r_i \)
        and \(\prod_{i\in S_-} (-\mu_i) r_i\)
        respectively. If both \(S_\pm\)
        are non-empty, then \(\Resid \Eul_{T}(E)^{-1} =0\).
	\end{example}

In general one can make different choices of residues, and
the choice depends on the problem at hand, see \cite[Section
4]{metzler} for other examples.

\subsection{The Kalkman-Metzler formula}% (fold)

The formulas of Kalkman and Metzler describe the wall-crossing of
K-theoretic integrals under variation of geometric invariant theory
quotient.  We follow the same notations as in Section
\ref{ss:masterspace}.  Consider family of polarisations \(\LL_t\to X\)
with a simple wall crossing at the singular time \(t=0\)
with an associated master space \(\tilde{X}\).
Let
\[ X^{\zeta,0} = X^{\ss}(\LL_0) \cap X^{\C^\times} \]
be the fixed point component of the one-parameter subgroup generated
by $\zeta$ involved in the wall-crossing.

\begin{definition} For an element $\zeta \in \g$, define the
  \emph{fixed point contribution}
\begin{multline}
\tau_{X,\zeta,0}: K_G^0(X) \to \Z, \\
\alpha \mapsto \chi\left(X^{\zeta,0} \qu (G_\zeta/\C^\times_\zeta);\
  \Resid \left(\frac{\alpha_0}{
      \Eul_{\C^\times_\zeta}(\nu_{X^{\zeta,0}})}\right) \right ).
\end{multline}
Here \(\alpha_0\)
is the image of \(\alpha\)
under the map \(K_G^0(X) \to K^0_{\C^\times_\zeta} (X^{\zeta,0})\).
\end{definition}

A version of the following formula was proved by Metzler \cite[Theorem
1.1]{metzler} in the differential geometric setting of manifolds with
circle actions.

\begin{theorem} [Kalkman-Metzler wall-crossing]
  \label{kalk2} Let
  \(X\) be a smooth \(G\)-variety (projective or affine) and
  suppose that  \(\LL_\pm \to X\) are 
  polarizations inducing a simple wall-crossing. 
  Then
\begin{equation} \label{kalkman} \tau_{X \qu_+ G}
  \kappa_{X,+}^G - \tau_{X \qu_- G} \kappa_{X,-}^G =
  \tau_{X,\zeta,0}.
\end{equation}
\end{theorem}  

\begin{example} Our first example concerns the passage from projective
  space to the empty set by variation of git quotient.  Consider
  \(G=\Ct\)
  acting on \(X=\C^{n+1}\)
  diagonally with weight one. A polarisation is given by a choice of
  another character \(\ell\in \Z\).
  For \(\ell<0\)
  the quotient is empty and for \(\ell>0\),
  \(X\qu G=\P^n\).
  Let \(\C_{(k)}\)
  be the weight \(k\geq 0\)
  representation. The sheaf \(\cF=\cO_X\otimes \C_{(k)}\)
  descends to \(\cO(k)\to \P^n\).
  The fixed point set \(X^G\)
  is the single point \(0\in X\)
  whose normal bundle is identified with \(\C^{n+1}\)
  itself. The inverted Euler class is \( (1- z^{-1})^{-(n+1)}\)
  and can be identified with the class of the symmetric algebra
  \(\Sym (z^{-1} \C^{n+1} )\).
  The fixed point contribution of the wall crossing formula
  \eqref{kalkman} equals
\[ \cO(k)|_0 \otimes \Sym (z^{-1} \C^{n+1}) = z^k \Sym
  (z^{-1} \C^{n+1}) .\] 
  The residue, corresponding to the \(z^0\)
  coefficient is the degree \(k\)
  part \(\Sym^k(\C^{n+1})\).  The wall crossing formula now reads
  \[
	\chi(\P^n,\cO(k))= \chi( X^G, \Sym^k (\C^{n+1})) = \dim
	\Sym^k(\C^{n+1}) = {n+k \choose k}. 
  \]
  Similarly, if one considers \(\cO(-k)\) and the identity
  \[
	\Sym V  =(-1)^{n+1} \det V^{\vee} \otimes
	\Sym V^{\vee},  
  \]
  we get that the wall crossing contribution is 
  \[
	\Resid z^{-k} z^{n+1} \Sym (\C^{n+1}) 
  .\]
  This is the degree \(k-n - 1\)
  part \(\Sym^{k-n-1}(\C^{n+1})\), \label{degreepart} hence
  \[
  \chi(\P^n,\cO(-k))= (-1)^{n} \chi( X^G, \Sym^{k-n-1} (\C^{n+1})) =
  (-1)^n {k-1 \choose k-n-1}.
  \]
\end{example}

\begin{example} \label{cremona}  Our second example concerns the Cremona
  transformation.  Let \(X = (\P^1)^3\) with polarization
\[
  \LL = \mO_{\P^1}(c_1) \boxtimes \mO_{\P^1}(c_2) \boxtimes
  \mO_{\P^1}(c_3),
\quad c_1 < c_2 < c_3.
\] 
Consider \(G = \C^\times\) acting on each factor \(\P^1\) by
\( t[z_0,z_1] = [z_0,t z_1] \)  and the action on
\(\mO_{\P^1}(n)\) so that the weights at the fibres over the
fixed points \([1,0]\), \([0,1]\) are \(n/2,-n/2\). \label{signs} We
consider the family of polarizations \(\LL_t = \LL \otimes
\C_t\) obtained by shifting \(\LL\) by a trivial line bundle
weight weight \(t\). 
The singular values \(t\) are given by the weights \((\pm c_1
\pm c_2 \pm c_3)/2\) of the action on the fibre \(\LL_t|_p\) at
the fixed points. 
Hence, there are eight chambers for \(t\) on which the geometric invariant theory 
quotient \(X\qu_t G\) is constant. Two have empty git
quotients.  In the first and last chamber, we have \(X \qu_t
G \cong \P(\C^3)\) resp. \(\P((\C^3)^\dual)\), while the six
intermediate wall-crossings induce three blow-ups and three
blow-downs involved in the classic Cremona transformation of
\(\P^2\). 

We consider how the Euler characteristic of the structure sheaf
changes under wall-crossing.  Consider the structure sheaf \(\mO_X\),
whose restriction to any of the chambers is again the respective
structure sheaf \(\mO_t\).
We now describe the change in the Euler characteristic
\(\chi(X\qu_t G,\mO_t)\)
as we vary \(t\)
under the wall-crossing from the chamber \(t < -c_1 - c_2 - c_3\)
to the first non-empty chamber
\(t \in (c_1 - c_2 - c_3,-c_1 + c_2 - c_3)\),
corresponding to the wall defined by the fixed point
\(x = ([1,0],[1,0],[1,0]) \in X\).
Since all weights of the action on the tangent bundle at this fixed
point \(x \in X\)
are \(1\),
we have that the Euler class at this fixed point is \((1-z^{-1})^3\).
The wall crossing formula yields
\[
  \chi(\P^2, \mO) = 0 + \chi\left(x, \Resid
  \frac{\mO|_x}{(1-z^{-1})^{3}} \right) = 1
\]
as expected.

The next wall-crossing, corresponds to the blow-up
\(\Bl(\P^2)\) of \(\P^2\) at a fixed point. The normal
bundle at this point has weights \(-1,1,1\) and thus the
wall crossing formula yields
\begin{eqnarray*}
  \chi(\Bl(\P^2), \mO) &=& \chi(\P^2,\mO) + \chi\left(x, \Resid
  \frac{\mO|_x}{(1-z)(1-z^{-1})^{2}} \right) \\ &=&
  \chi(\P^2,\mO)+0=1
\end{eqnarray*}
since the residue is zero each time there are contributions with both
positive and negative weights (since the numerator is trivial) which
is consistent, since the Euler characteristic \(\chi(X,\mO_X)\)
is a birational invariant.  In fact the invariance of the Euler
characteristic \(\chi(X \qu G,\cO_{X \qu G})\)
under git birational transformation follows from Example
\ref{ex:0res}.
\end{example}

\begin{proof}[Proof of Theorem \ref{kalk2}]
  Consider the master space \(\ti{X}\) associated to the
  wall-crossing. Let \(\alpha\in K_G^0(X)\), and consider its
  image 
  \(\delta(\alpha)\in K^0_{\Ct}(\ti{X})\) of
  \eqref{eq:deltamap}. Using
  \(K\)-theoretic localization \eqref{eq:loc} 
  on the \(\Ct\)-space \(\ti{X}\),
  and the identification of the fixed point
  components and normal bundles with those in the ambient
  space we obtain the relation 
  \begin{equation*}\label{eq:aux0}
	\delta(\alpha) = \iota_{*} 
	\frac{\kappa_{X,-}^G(\alpha)}{\Eul_{\Ct}
  (\nu_-)}
  + \iota_{*} 
\frac{\kappa_{X,+}^G(\alpha)}{\Eul_{\Ct}(\nu_+)} +
\iota_{*} \frac{\alpha_0}{\Eul_{\Ct_{\zeta}}(\nu_{X^{\zeta,0}})},
\end{equation*}
where \(\iota:F\to \ti{X}^{\Ct}\) is the inclusion 
of (each) fixed point components. By applying residues and
Euler characteristics
  \begin{multline}\label{eq:aux1}
	\chi\left(\ti{X}; \Resid \delta(\alpha)\right) =
	\chi\left(X \qu_- G; 
	\Resid \frac{\kappa_{X,-}^G(\alpha)}{\Eul_{\Ct}
  (\nu_-)}\right)\\ 
  + \chi\left(X
\qu_+ G; \Resid
\frac{\kappa_{X,+}^G(\alpha)}{\Eul_{\Ct}(\nu_+)}\right)
+\chi\left(X^{\zeta,0} ;\Resid
\frac{\alpha_0}{\Eul_{\Ct}(\nu_{X^{\zeta,0}})}\right).
\end{multline}

  The left hand side of \eqref{eq:aux1} is
  zero by the definition of residue (c.f. Remark
  \ref{rk:residue}) . 
%  
%  As virtual
%  \(\Ct\)-representation, \(\chi\left(\ti{X};
%  \delta(\alpha)\right) \) is a \emph{finite} 
%  Laurent polynomial in \(z^{\pm 1}\), since \(\ti{X}\)
%  is proper. This means that it has poles only at
%  \(z^{\pm 1}=0\). On the other hand, the terms on the right
%  of \eqref{eq:aux1} arising from localisation have poles
%  only at roots of unity by the discussion after Equation
%  \eqref{}. Thus, both sides are independent
%  of \(a\) (but the independent terms are.)
%
  %
  The group \(\C^\times\)
  acts on \(\nu_\pm \)
  with weights \(\mp 1\)
  on the normal bundles \(\nu_\pm \to X \qu_\pm G\)
  of \(X \qu_\pm G\)
  in \(\ti{X}\)
  respectively, since they are canonically identified with
  \((\LL_+\otimes\LL_-)^{\pm 1}\).  Hence one obtains
  \begin{multline*}
	0 = - \tau_{X \qu_+ G} \kappa_{X,+}^G \alpha +
  \tau_{X \qu_- G} \kappa_{X,-}^G \alpha +
  \chi \left(
  X^{\zeta,0};\ \Resid \frac{\alpha_0}{
	\Eul_{\Ct_\zeta}(\nu_{X^{\zeta,0}})}\right)
\end{multline*}
  as claimed.
\end{proof}

We remark that Kalkman-Metzler wall-crossing formula Theorem
\ref{kalk2} can be generalised to more complicated wall-crossing, such
as when there are multiple singular times and fixed components.  The
details are left to the reader.

\subsection{The virtual wall-crossing formula} % (fold)

A virtual version of the Kalkman-Metzler formula follows from a
virtual version of localization in K-theory proved by Halpern-Leistner
\cite[Theorem 5.7]{hal:strat}.  Let \(\XX\)
be a Deligne-Mumford \(G\)-stack
with coarse moduli space \(X\).
Let \(\LL_\pm \to \XX\)
be \(G\)-polarizations
of \(\XX\).
That is, $\LL_\pm$ are \(G\)-line
bundles which are ample on the coarse moduli space \(X\).  Let
\[ \Pic_\Q(\XX) = \Pic(\XX) \otimes_\Z \Q \]  
denote the rational Picard group and
\begin{equation}
  \label{eq:tpolarv}
  \LL_t := \LL_-^{(1-t)/2}\otimes
  \LL_+^{(1+t)/2} \in \Pic_\Q(\XX), \quad t \in (-1,1)\cap \Q.
\end{equation}
The family \eqref{eq:tpolarv} is a one-parameter family of rational
polarizations given by interpolation.  We will also assume, for
simplicity that there is a simple wall-crossing:

\begin{assumption} \label{simple} There is only one \emph{singular
    value} \(t=0\),
  such that the semistable loci 
\[ \XX^{\ss}(+):=\XX^{\ss}(\LL_t), \quad
\XX^{\ss}(-):=\XX^{\ss}(\LL_t) \] 
are constant for \(\pm 1\leq t < 0\).
We assume that stable equals semi-stable for \(t\neq 0\).
Moreover
\(\XX^{\ss}(\LL_0) \setminus (\XX^{\ss}(+)\cup \XX^{\ss}(-)) \)
is connected. Assume that the stabiliser \(G_x\)
for
\(x\in \XX^{\ss}(\LL_0) \setminus (\XX^{\ss}(+)\cup \XX^{\ss}(-)) \)
is isomorphic to \(\Ct\).
\end{assumption} 

Consider the case that the given stack has an equivariant perfect
obstruction theory.  The \(G\)-action
on \(\XX\) induces a canonical morphism
\[a^\dual:L_{\XX} \to \ul{\g}^\dual \subset L_{\XX \times
G}=L_{\XX} \times \ul{\g}^\dual \]
that we call the {\em infinitesimal action}.   Composing
with the morphism \(E\to L_\XX\) gives a natural lift
\(\ti{a}^\dual: E \to \ul{\g}^\dual\). 
Let \(\XX^{\ss}\) be the semistable locus for some
polarization and assume stable=semistable, that is all the
stabilisers are finite. Let \(E^{\ss},
L_{\XX^{\ss}}\) etc. be the restrictions to the semistable
locus. 

\begin{lemma} \label{qot}
  The perfect obstruction theory \(E^{\ss} \to L_{\XX^{\ss}} :=
  L_{\XX}|_{\XX^{\ss}} \) descends to a perfect obstruction
  theory on the quotient \(\XX^{\ss} / G\).
\end{lemma}

\begin{proof}  
 From the fibration \(\pi:\XX^{\ss} \to \XX^{\ss}/G\) one obtains
 an exact triangle of cotangent complexes
 \begin{equation} \label{etg} 
   L_{\XX^{\ss}/ G} \to
   L_{\XX^{\ss}} \stackrel{a^\dual}{\to}  
   \ul{\g}^\dual\to  L_{\XX^{\ss}/G}[1],
 \end{equation}
thus we can consider \(L_{\XX^{\ss}} \to   \ul{\g}^\dual\)
as the cotangent complex of \(\XX^{\ss}/G\).  
Let \(\Cone(\ti{a}^\dual)\) denote the mapping cone of \(\ti{a}^\dual\).  Then the
exact triangle
\[ \Cone(\ti{a}^\dual) \to 
E^{\ss} \stackrel{\ti{a}^\dual}{\to} \ul{\g}^\dual \to
\Cone(\ti{a}^\dual)[1]\]
admits a morphism to \eqref{etg}, in particular making
\(\Cone(\ti{a}^\dual) \to L_{\XX^{\ss}/G}\) into an
obstruction theory with amplitude in \([-1,1]\).  By the
assumption on the finite stabilizers, this 
obstruction theory is perfect.
\end{proof}

The perfect obstruction theories on the stack induce perfect
obstruction theories on the fixed point components.  Assume again that
\(\XX\)
is equipped with a \(G\)
action. For any \(\zeta \in \g\),
consider the fixed point stacks \(\XX^{\zeta}\).
The restriction of the perfect obstruction theory
\(E^\bullet|_{\XX^{\zeta}}\) decomposes as
\[
  E^\bullet|_{\XX^{\zeta}}= E^{\bullet, \mov} + E^{\bullet,
  \fix}
\]
where $E^{\bullet, \mov} $ is the {\em moving part} and
$E^{\bullet, \fix}$ the {\em fixed part}.  By results of
\cite{gr:loc}, \(E^{\bullet, \fix}\)
yields an equivariant perfect obstruction theory for \( \XX^\zeta\),
which is compatible with that on \(\XX\).  Denote by
\[ a_\zeta^\dual: L_{\XX/G_\zeta} \to
(\g_\zeta/\C\zeta)^\dual \]  
the map given by the infinitesimal action of
\(G_\zeta/\C^\times_\zeta\)
on \(\XX/G_\zeta\).  Denote by 
\[ \nu_{\XX^\zeta} \in \Ob( D^b\Coh(\XX^\zeta/G_\zeta)) \]
the \emph{moving part} of the mapping cone \(\Cone(a_\zeta)\);
this is the conormal complex for the embedding
\(\XX^\zeta/G_\zeta \to \XX/G_\zeta\)
except for the factor \(\C \zeta\)
of automorphisms of the fixed point set.  Denote by \(\XX^{\zeta,0}\)
the locus of \(\XX^\zeta\)
which is \(\LL_0\)-semistable
and by \(\nu_{\XX^{\zeta,0}}\)
the restriction of \(\nu_{\XX^{\zeta}}\) to \(\XX^{\zeta,0}\).

A virtual version of localization in K-theory for stacks is proved in
Halpern-Leistner \cite[Theorem 5.7]{hal:strat}. 
 A special case is the
following: Assume that a torus \(T\)
acts on \(\XX\)
(e.g. the action of \(\Ct\)
on a master space for wall-crossing) and let \(\XX^T\)
denote the fixed locus. The virtual tangent space restricted to
\(\XX^T\) decomposes into
\[
  \on{Def}-\on{Obs}|_{\XX^T} =   \on{Def}^\fix
  -\on{Obs}^\fix+
  \on{Def}^\mov -\on{Obs}^\mov
\]
its fixed and moving parts consisting of \(T\)-modules
with trivial and non-trivial weights.  The K-class
\(\nu_{\XX^{T}}=(\on{Def}^\mov -\on{Obs}^\mov)\)
is the virtual \emph{normal bundle}.  In the case $T = \C^\times$ this
class splits into classes $\nu_{\XX^T,\pm}$ corresponding to negative
and positive weights and the inverted Euler class of $\nu_{\XX^T}$
maybe defined as the tensor product
\[ \Eul(\nu_{\XX^T})^{-1} = 
\Eul(\nu_{\XX^T,+})^{-1} 
\Eul(\nu_{\XX^T,-})^{-1} \] 
where as in \cite[Footnote 24]{hal:strat} \label{footnote}
\begin{eqnarray*}   \Eul(\nu_{\XX^T,-})^{-1} &=& \Sym(  \nu_{\XX^T,-} ) \\
\Eul(\nu_{\XX^T,+})^{-1} &=& \Sym(\nu_{\XX^T,+}^\dual) \otimes
\det(\nu_{\XX^T,+}) [-\rank(\nu_{\XX^T,+})]  \end{eqnarray*}
is an infinite sum of bundles such that each weight component is
finite; this suffices for the finiteness of the formulas below.  The
residue of such classes is defined as before in \eqref{residuedef}, by
taking the difference in expansion as formal power series in $z$ and
$z^{-1}$.

\begin{theorem} {\rm (Virtual localisation)} \cite[5.6,5.7]{hal:strat}
  Suppose that $\XX$ is the a proper global quotient of a
  quasiprojective scheme by a reductive group equipped with an
  equivariant $T = \C^\times$-action, and $E$ is a $T$-equivariant
  sheaf on $\XX$.  Then the Euler characteristic of $E$ is computed by
\begin{eqnarray*}\label{eq:vloc}
  \chi^{\vir}(\XX,E)  &=& \chi^{\vir}(\XX^T,E \otimes 
  \Eul_T (\nu_{\XX_{T}})^{-1}) \\
  &=& \sum_{F \subset \XX^T} \chi^{\vir}(F, E \otimes
  \Eul_T (\nu_{F})^{-1} ) \end{eqnarray*}
for the inclusion \(\iota:\XX^T \to \XX\).
Here the sum runs over all fixed components $F$ of $\XX^T$ and the
Euler class \(\Eul_T(\nu_{\XX_{T}})\)
are defined as in \ref{ss:euler}, and it is invertible in localised
equivariant K-theory. 
\end{theorem}

\begin{proof} This is stated for quasi-smooth schemes $Y$ in
  \cite[Theorem 5.7]{hal:strat}. The statement for quasi-smooth global
  quotient stacks $\XX = Y/G$ by reductive groups $G$ follows from the
  identification of the Euler characteristic $\chi(\XX,E/G)$ on $\XX$
  with the invariant part of the Euler characteristic $\chi(Y,E)^G$
  upstairs, where the Bialynicki-Birula decomposition satisfies the
  requirements of the stratification by \cite[Section 1.4.1]{hal:on}.
\end{proof}

The formula \eqref{eq:vloc} implies a formula for Euler
characteristics as a sum over fixed point components, in which the
fixed point contributions are of the following form.  Let
\(\Eul_{\Ct_\zeta}(\nu_{\XX^{\zeta,0}})\)
be the equivariant Euler class of the normal bundle in
\(K^0_{\Ct_\zeta}(\XX^{\zeta,0})\), \label{laurent}
Let \(\tau_{\XX,\zeta,0}\)
denote the equivariant virtual Euler characteristic twisted by
\(\Eul_{\C^\times_\zeta}^{-1}(\nu_{\XX^{\zeta,0}})\)
combined with the residue
\begin{multline*}
  \tau_{\XX,\zeta,0}:
  K^0_{\Ct_\zeta}(\XX^{\zeta,0})\to
  \Z; \quad \sigma\mapsto
\chi^{\vir}_{\Ct_\zeta}\left(\XX^{\zeta,0};
\Resid \frac{\sigma}{\Eul_{\C^\times_\zeta}(\nu_{\XX^{\zeta,0}}) }
\right). 
\end{multline*}
Let \(\tau_{\XX \qu_\pm G}\) denote the virtual Euler characteristic
\[
  \tau_{\XX \qu_\pm G}: K(\XX\qu_\pm G) \to \Z, \quad \alpha
  \mapsto \chi^{\vir}\left(\XX \qu_\pm G;\alpha\right),
\]
on the quotients as before. We have the following virtual
Kalkman-Metzler wall-crossing formula:

\begin{theorem} \label{vkalkman} Let \(\XX\)
  be a proper Deligne-Mumford global-quotient \(G\)-stack
  equipped with a \(G\)-equivariant
  perfect obstruction theory which admits a global resolution by
  vector bundles. Let \(\LL_\pm \to \XX\)
  be \(G\)-line
  bundles whose associated wall-crossing is simple.  Then
\begin{equation} \tau_{\XX \qu_+ G} \ \kappa_{\XX,+}^G -
  \tau_{\XX \qu_- G} \ \kappa_{\XX,-}^G =
  \tau_{\XX,\zeta,0} 
\end{equation}
\end{theorem}  

\begin{example} \label{nodal} 
  {\rm (Wall-crossing over a nodal fixed point)}
  Suppose that \(\XX=\P^1\cup \P^1\) is a nodal projective line 
  with a single node \(x_0\in\XX \), 
  equipped with the standard \(G=\C^\times\) action
  on each component, so that the weights of the action on
  the tangent spaces at the node are \(\pm 1\).  Equip \(\XX\)
  with a polarization so that the weights are \(\pm 1\) at the
  smooth fixed points, and \(0\) at the nodal
  point \(x_0=\XX^{\zeta,0}\). Then \(\XX
  \qu_t G\) is a point for \(t \in (-1,1)\), and is singular
  for \(t=0\). Since \(\XX\) is a complete intersection, \(\XX\)
  it has a perfect obstruction theory \cite[Example before
  Remark 5.4]{bf:in} and the virtual wall-crossing formula
  of Theorem \ref{vkalkman} applies. We examine the
  wall-crossing for the class \(\cO_\XX\) at the
  singular value \(t= 0\). The virtual normal complex at the
  nodal point is the quotient of \(\C_1 \oplus \C_{-1}\), the
  sum of one-dimensional representations with weights
  \(1,-1\), modulo their tensor product \(\C_1 \otimes
  \C_{-1}\), which has weight \(1-1=0\). Hence the normal complex
  has inverted Euler class
\[ 
  \Eul_{\Ct}( \nu_{\XX^{\zeta,0}})^{-1} = \frac{1}{
  (1-z^{-1})(z-1)}
\]
whose residue is
zero by Example \ref{ex:0res}.
The Euler characteristics on the left and right hand sides
are \(1\) while the wall-crossing term is
\begin{eqnarray*} 1 - 1 &=& \tau_{\XX \qu_+ G} \
  \kappa_{\XX,+}^G - \tau_{\XX \qu_- G} \ \kappa_{\XX,-}^G
  \\ &=& \tau_{\XX,\zeta,0} =
  \chi^\vir \left(x_0, \Eul_{\Ct}(
  \nu_{\XX^{\zeta,0}})^{-1}\right) =  0 \end{eqnarray*}
compatible with the wall-crossing formula.  \end{example}

The proof of Theorem \ref{vkalkman} uses the construction of a
master space for this set up. However, the same construction
as before with small modifications applies. 

\begin{lemma} \label{vmaster} 
There exists a proper Deligne-Mumford \(\C^\times\)-stack
\(\ti{\XX}\) equipped with a line bundle ample for the coarse
moduli space whose git quotients \( \ti{\XX} \qu_t \C^\times\)
are isomorphic to those \( \XX \qu_t G\) of \(\XX\) by the
action of \(G\) with respect to the polarization \(\LL_t\) and
whose fixed point set \(\ti{\XX}^{\C^\times}\) is given by the
union
\begin{equation}\label{eq:fixedset}
\ti{\XX}^{\C^\times} = (\XX \qu_- G) \cup (\XX \qu_+ G)
\cup \iota_\zeta( \XX^\zeta \qu_0
(G_\zeta/\C^\times_\zeta))
\end{equation}
where \(\iota_\zeta\)
is the natural map to \(\ti{\XX}\)
as before.  Furthermore, \(\ti{\XX}\)
has a perfect obstruction theory admitting a global resolution by
vector bundles with the property that the virtual normal complex of
\(\XX^\zeta \qu_0 (G_\zeta/\C^\times_\zeta)\)
is isomorphic to the image of \(\nu_{\XX^\zeta} / (\g/\C\zeta)\)
under the quotient map
\(\XX^\zeta \to \XX^\zeta \qu (G_\zeta / \C^\times_\zeta)\),
by an isomorphism that intertwines the action of \(\C^\times_\zeta\)
on
\((\nu_{\XX^\zeta} / (\g/\C\zeta) ) \qu (G_\zeta / \C^\times_\zeta)\)
with the action of \(\C^\times\) on \(\nu_{\ti{\XX}^{\C^\times}}\).
\end{lemma}

\begin{proof} The construction of the master space is the same as in
  the smooth case, that is, the master space is the stack-theoretic
  quotient \( \ti{\XX}=\P(\LL_- \oplus \LL_+) \qu G .\)
  The action of \(G\)
  on the semistable locus in \(\P(\LL_- \oplus \LL_+)\)
  is locally free by assumption.  It follows that \(\ti{\XX}\)
  is a proper Deligne-Mumford stack, and by Lemma \ref{qot} has a
  perfect obstruction theory induced from the natural obstruction
  theory on \(\P(\LL_- \oplus \LL_+)\)
  given by considering it as a bundle over \(\XX\).
  The quotient \(\ti{\XX}\)
  is such that \(\ti{\XX}\qu_t \Ct\)
  is isomorphic to \(\XX\qu_t G\)
  for \(t\neq 0\).
  In fact it contains the quotients of \(\P(\LL_\pm) \cong \XX\)
  with respect to the polarizations \(\LL_\pm\),
  that is, \(\XX \qu_\pm G\).

  The same argument as before describes the fixed point loci: 
  they correspond to fixed point loci in 
  \(\P(\LL_- \oplus \LL_+)\) for one-parameter
  subgroups of \(\C^\times \times G\). Given such a locus
  \(\P(\LL_-\oplus \LL_+)^{\xi
  +\zeta}\), the pull-back of the virtual normal complex is by
definition the moving part of \(\Cone(\ti{a}_{\P(\LL_-\oplus
  \LL_+)}^\dual)\), where 
\[\ti{a}^\dual_{\P(\LL_-\oplus \LL_+)}: E_{\P(\LL_- \oplus
\LL_+)} \to \ul{\g}^\dual_\zeta\]
is the lift of the infinitesimal action of \(G_\zeta\).
Consider the fibration \(\pi: \P(\LL_- \oplus \LL_+) \to
\XX\).  By definition \(E_{\P(\LL_- \oplus \LL_+)}\) fits into
an exact triangle
\[   E_\XX \to E_{\P(\LL_- \oplus \LL_+)} \to L_\pi \to
E_\XX[1] .\]
Over the complement \(\P(\LL_- \oplus \LL_+)\setminus
(D_0\cup D_\infty) \subset \P(\LL_- \oplus \LL_+)\) of the
sections at zero and infinity we may identify \(L_\pi \cong
\ul{\C}\) using the \(\C^\times\)-action on the fibers, by the
assumption on the weights of the \(\C_\zeta\) action on the
fiber. Thus the projection to \(\XX\) identifies
\[\Cone(\ti{a}^\dual_{\P(\LL_-\oplus \LL_+)} |_{\P(\LL_-
  \oplus \LL_+)^{\xi + \zeta}}) \to \pi^*
  \Cone(\ti{a}^\dual_{\XX^\zeta})
\]
where 
\[
\ti{a}^\dual_{\XX^\zeta}: E_{\XX} | \XX^\zeta \to (
\ul{\g_\zeta/\C\zeta} )^\dual
\]
is the lift of the infinitesimal action of
\(\g_\zeta/\C\zeta\).  Now the virtual normal complex is by
definition the \(\C^\times\)-moving part of the perfect
obstruction theory; the Lemma follows.  
\end{proof} 

\begin{proof}[Proof of Theorem \ref{vkalkman}] 
  For any equivariant class \(\alpha \in K^0_G(\XX)\),
  its pullback to \(\P(\LL_- \oplus \LL_+)\)
  descends to a class \(\ti{\alpha} \in K^0_{G}(\ti{\XX})\)
  whose restriction to \(\XX \qu_\pm G\)
  is \(\kappa_{X,\pm}^G(\alpha)\),
  and whose pullback under \(X^{\zeta,0} \to \ti{X}^{\C^\times}\)
  is \(\iota_{X^{\zeta,0}} \alpha\).
  By virtual localisation \eqref{eq:vloc} and the description of the
  fixed set, the normal bundles in Lemma \ref{vmaster}
  \begin{multline}
	0=\chi^{\vir}({\ti{\XX}}; \Resid \delta(\alpha)) =\\
	\chi^{\vir}\left(\XX \qu_- G;
	\Resid \frac{\kappa_{\XX,-}^G(\alpha)}{\Eul_{\Ct}
  (\nu_-)}\right) + 
  \chi^{\vir}\left(\XX \qu_+ G;
  \Resid \frac{\kappa_{\XX,+}^G(\alpha)}{\Eul_{\Ct}
(\nu_+)}\right)+\\
\chi^{\vir}\left({\XX^{\zeta,0}/(G_\zeta/\C^\times)};
\Resid \frac{\iota^*_{{\XX}^{\zeta,0}}
\alpha}{\Eul_{\Ct_\zeta}(\nu_{\XX^{\zeta,0}})}
\right).
\end{multline}
As in \eqref{firstex} we have 
\[
  \Resid \chi( \XX \qu_\mp G, 
\frac{\kappa_{\XX,\mp}^G(\alpha)}{\Eul_{\Ct}(\nu_\mp)}
=\pm
\chi(\XX \qu_\mp G, \kappa_{\XX,\mp}^G(\alpha) .) 
\] 
Indeed, by definition the normal bundle $\nu_F$ at $\XX \qu_\mp G$ has
virtual dimension one with positive resp. negative weight, and the
inverted Euler class $\Eul(\nu)^{-1}$ is the symmetric product
$\Sym( \nu ) $ for the Bialynicki-Birula decomposition for positive
weight, or $\Sym(\nu^\dual) \det(\nu^\dual)$ for the decomposition
with negative weight.  \footnote{That these classes both agree with
  the inverted Euler class in localized K-theory follows by inspection
  from Riemann-Roch \cite{to:rr}.  In fact, the agreement with the
  inverted Euler class is not necessary and one may take the
  difference in the Halpern-Leistner version of virtual localization
  for the Bialynicki-Birula decomposition and its opposite as the
  definition of residue.}  For $\chi(\XX \qu_+ G)$ the invariant part
of the first is the trivial line bundle, while for the second the
invariant part vanishes since the weights are positive so the
difference in \eqref{residuedef} of
$\frac{\kappa_{\XX,\mp}^G(\alpha)}{\Eul_{\Ct}(\nu_\mp)}$ is
$\kappa_{\XX,\mp}^G(\alpha)$.  For $\chi(\XX \qu_- G)$, the weight of
$\nu$ is negative and $\Sym(\nu)$ appears in the Bialynicki-Birula
decomposition for negative weight.  Thus the residue is
$\kappa_{\XX,\mp}^G(\alpha)$, which completes the proof.
\end{proof}

\section{Wall-crossing in quantum K-theory}
\label{qwall} 

The main result in this section, Theorem \ref{gwall} below, relates
the quantum K-theory pairings on both sides of a wall-crossing.  Let
\(X \qu_\pm G\)
denote the associated quotient stacks \([X^{\ss}(\pm)/G]\)
at times \(t=\pm 1\) and let
\[ \kappa_{X,\pm}^G : QK_G^0(X,\LL_\pm) \to QK_{\C^\times}^0(X \qu
_\pm G) \]
denote the associated quantum Kirwan maps. Consider the
graph potentials 
\[ \tau_{X \qu_\pm G}: QK_{\C^\times}^0(X \qu_\pm G) \to
\Lambda_{X,\LL_\pm}^G . \]
Denote by 
\[QK_G^{0,\on{fin}}(X) \subset
QK_G^0(X,\LL_-) \cap QK_G^0(X,\LL_+)\]
the subset of classes of sums lying in both completions; for example,
any finite sum lies in this intersection.  We want to establish a
formula for the difference
\[\tau_{X \qu_+ G} \ \kappa_{X,+}^G - \tau_{X \qu_- G}
\ \kappa_{X,-}^G: QK_G^{0,\on{fin}}(X) \to \Lambda_X^G \]
as a sum of fixed point contributions given by gauged
Gromov-Witten invariants with smaller structure groups
\(G_\lambda/\Ct_\lambda\).

\subsection{Master space for gauged maps and wall-crossing}
\label{master2}

We recall from \cite[Proposition 3.1]{wall} the construction of master
space whose quotients are the moduli stacks of Mundet stable gauged
maps.

\begin{proposition} [Existence of a master space]  
\label{masterprop}
Under assumptions of simple wall-crossing \ref{simple}, for each
equivariant degree \(d\in H_2^G(X)\)
there exists a proper Deligne-Mumford \(\C^\times\)-stack
\(\ol{\M}_n^G(C,X,\LL_-,\LL_+,d)\) with the following properties:
\begin{enumerate}
\item \label{mobs} 
  The stack \(\ol{\M}_n^{G}(C,X,\LL_-,\LL_+,d)\) has
  a \(\Ct\)-equivariant perfect 
  obstruction theory, relative over the moduli
  stack \(\ol{\MM}_n(C)\) of prestable maps to \(C\) of class
  \([C]\).  
\item the git quotients of \(\ol{\M}_n^G(C,X,\LL_-,\LL_+)\)
  by the \(\Ct\)-action
  are isomorphic to the moduli stacks
  \(\ol{\M}_n^G(C,X, \LL_-^{(1-t)/2} \otimes \LL_+^{(1+t)/2})\)
  for parameter \(t \in (-1,1)\);
\item the \(\C^\times\)-fixed
  substack includes \(\ol{\M}_n^G(C,X,\LL_-,d)\)
  and \(\ol{\M}_n^G(C,X,\LL_+,d)\);
  the other fixed point components are isomorphic to substacks of {\em
    reducible} elements of
  \(\ol{\M}_n^G(C,X, \LL_-^{(1-t)/2} \otimes \LL_+^{(1+t)/2})\)
  for $t \in (-1,1)$ consisting of gauged maps with
  $\C^\times$-automorphisms.
\item  \(\ol{\M}_n^G(C,X,\LL_-,\LL_+,d)\) admits an embedding
  in a non-singular Deligne-Mumford stack.  
  \end{enumerate}
\end{proposition} 

For any fixed point component
\(F \subset \ol{\M}_n^G(C,X,\LL_-,\LL_+,d)\)
denote by \(\nu_F\)
the normal complex, that is, the \(\C^\times\)-moving
part of the perfect obstruction theory of Proposition \ref{masterprop}
part \eqref{mobs}. The following is a direct application of virtual
wall crossing applied to the master space.  Denote the evaluation map
\[
\ev: \ol{\M}_n^G(C,X,\LL_-,\LL_+,d)\to (X/G)^n.\]

\begin{proposition}
For any class \(\alpha \in K_G^0(X)^n\)
\begin{multline}  \label{fpf}
\chi^{\vir}\left( \ol{\M}_n^G(C,X,\LL_+,d); \ev^*\alpha \right)  
 - \chi^{\vir}\left( \ol{\M}_n^G(C,X,\LL_-,d); 
 \ev^*\alpha \right) \\ = \sum_F 
 \chi^{\vir} \left(F ; \Resid \frac{\iota_F^* \ev^* \alpha}{ 
   \Eul_{\C^\times}(\nu_F)} \right)
\end{multline}
where \(F\) ranges over the fixed point components of
\(\C^\times\) on the moduli \(\ol{\M}^G_n(C,X,\LL_-,\LL_+,d)\) that are not equal to
\(\ol{\M}^G_n(C,X,\LL_\pm,d)\).
\end{proposition}

\subsection{Reducible gauged maps}

We analyze further the fixed point contributions in \eqref{fpf}, which
come from reducible gauged maps.  Let \(X\)
be a smooth projective \(G\)-variety.
Let \(Z \subset G\)
a central subgroup. For any principal \(G\)-bundle
\(P \to C\),
the right action of \(Z\)
on \(P\)
induces an action on the associated bundle \(P(X)\),
and so on the space of sections of \(P(X)\).
The fixed point set of \(Z\)
on \(P(X)\)
is \(P(X)^Z =P(X^Z)\),
the associated bundle with fiber the fixed point set
\(X^Z \subset X\).
The action of \(Z\)
on the space of sections of \(P(X)\)
preserves Mundet semistability, since the parabolic reductions are
invariant under the action and the Mundet weights are preserved. This
induces an action of \(Z\)
on \(\ol{\M}^G_n(C,X)\).

\begin{proposition} 
  Let \(Z \subset G\) be a central subgroup.  The fixed point
  locus for the action of \(Z\) on \(\ol{\M}_n^G(C,X)\) is the
  substack whose objects are tuples 
\[(p: P \to C, u: \hat{C} \to P(X),\ul{z})\] 
such that
\begin{enumerate} 
\item \(u\) takes values in \(P(X^Z)\) on the principal component \(C_0\);
\item for any bubble component \(C_i \subset \hat{C}\) mapping to a
  point in \(C\), \(u\) maps \(C_i\) to a one-dimensional orbit of \(Z\) on
  \(P(X)\); and
\item any node or marking of \(\hat{C}\) maps to the fixed point set
  \(P(X^Z)\).
\end{enumerate} 
\end{proposition} 

We introduce notation for the various substacks of reducible maps. 
Let \(\zeta \in \g\) generate a one-parameter 
subgroup \(\C^\times_\zeta \subset G\).  Recall that \(G_\zeta\) denotes
the centralizer in \(G\) and so it contains
\(\C^\times_\zeta\) as a central subgroup.  Let
\[\ol{\M}_n^{G_\zeta}(C,X,\LL_t,\zeta,d)  \]
denote the stack of \(\LL_t\)-Mundet-semistable morphisms from
\(C\) to \(X/G_\zeta\) that are \(\C^\times_\zeta\)-fixed and take
values in \(X^{\zeta}\) on the principal component and in
\(X/G_\zeta\) on the bubbles.  Because these gauged maps
correspond to the smaller group \(G_\zeta\), we call 
them as \emph{reducible} gauged maps.

Each component of reducibles has an equivariant perfect obstruction
theory.  Recall that the obstruction theory for the moduli of gauged
maps \(\ol{\M}^G_n(C,X,d)\)
is given by the complex \(Rp_*e^*T_{X/G}^\dual\),
which is relative with respect to \(\ol{\MM}_n(C)\).
The moduli \(\ol{\M}^{G_\zeta}_n(C,X,\LL_0,\zeta,d)\)
is an Artin stack, and if every automorphism group is finite modulo
\(\C^\times_\zeta\),
it is a proper Deligne-Mumford stack with a \(\C^\times\)-equivariant
relatively perfect obstruction theory over \(\ol{\MM}_n(C)\).
This follows from the fact that the relative perfect obstruction
theory for \( \ol{\M}^{G_\zeta}_n(C,X,\LL_t,\zeta,d)\)
is pulled back from that on the \(\C^\times\)-fixed
point set in the master space
\(\ol{\M}_n^G(C,X,\LL_-,\LL_+,d)^{\C^\times}\).
This coincides with the \(\C^\times_\zeta\)-equivariant
obstruction theory on the stack
\(\ol{\M}^{G_\zeta}_n(C,X,\LL_0,\zeta,d)\)
whose relative part is the \emph{cone} with target the trivial bundle
$\ul{\C}_\zeta^\dual$ with fiber the Lie algebra $\C_\zeta$ of
$\C^\times_\zeta$
\[\on{Cone}(Rp_*e^* T_{X/G}^\dual \to \C_\zeta^\dual)\]
given by the infinitesimal action of \(\C^\times_\zeta\).
The complex \(Rp_* e^* T_{X/G}^\dual\) itself is not perfect
because of the
\(\C^\times_\zeta\)-auto\-mor\-phisms; taking the cone has the
effect of cancelling this additional automorphisms. Denote by
\(\nu_0\) the virtual (co)normal complex of the morphism
\[
\ol{\M}^{G_\zeta}_n(C,X,\LL_t,\zeta) \to
\ol{\M}_n^G(C,X,\LL_-,\LL_+),
\]
and as before, denote by
\[ 
\Eul_{\Ct}(\nu_0) \in
K(\ol{\M}^G_n(C,X,\LL_0,\zeta))
\]
its Euler class. 

\label{ss:redgauge}
Virtual Euler characteristics over the reducible gauged maps gives
rise to the fixed point contributions in the wall-crossing formula:
these are \emph{\(\zeta\)-fixed
  K-theoretic gauged Gromov-Witten invariants}.  The \(\zeta\)-fixed
K-theoretic gauged Gromov-Witten invariants that appear in the
wall-crossing formula involve further twists by the inverse of Euler
classes of the virtual normal complex
\(\Eul_{\Ct}(\nu_0)^{-1}\).
Before we made this explicit, we need to allow a slightly larger
coefficient ring.  Denote by
\[
  \ti{\Lambda}_X^G := \Map(H_2^G(X,\Z),\Q) 
\]
the space of \(\Q\)-valued
functions on \(H_2^G(X,\Z)\) \label{switched}
(cf.  Equation \eqref{eq:novikov}).  Note that \(\ti{\Lambda}_X^G\)
has no ring structure extending that on \(\Lambda_X^G\).
The space \(\ti{\Lambda}_X^G\)
can be viewed as the space of distributions in the quantum parameter
\(q\),
and we use it as a master space interpolating Novikov parameters for
the quotients with respect to \(\LL_t\)
as \(t\)
varies.  Let \(\ti{\Lambda}^G_{X,\fin}\)
denote the subspace of finite sums.

\begin{definition}
  Let \(X,G,\LL_\pm, \zeta \in \g\) as above, such that there
  is simple wall-crossing at the unique singular time
  \(t=0\) and such that 
  \(X^{\zeta,0}\) is non-empty. The \emph{fixed point potential}
  associated to this data is the map 
\begin{multline} \tau_{X,\zeta,0}: QK_G^{0,\fin}(X) \to
  \ti{\Lambda}_X^G \\ \quad \alpha \mapsto \sum_{d
  \in H_2^G(X,\Z)} \sum_{n \ge 0}
  \chi\left(\ol{\M}_n^{G}(C,X,\LL_0,\zeta,d); \Resid \frac{\ev^* (\alpha,
  \ldots, \alpha) }{\Eul_{\Ct}(\nu_0)}\right)
  \frac{q^d}{n!},
\end{multline}
for \(\alpha \in K_G^0(X)\), extended to \(QK_G^{0,\fin}(X)\) by
linearity. Here we omit the restriction map \(K_G^{0,\fin}(X)
\to K_{G_\zeta}^0(X)\) to simplify notation.

\end{definition}

\begin{remark} The fixed point potential \(\tau_{X,\zeta,0}\)
  takes values in \(\ti{\Lambda}_X^G\)
  rather than in \(\Lambda^G_{X}(\LL_0)\)
  because Gromov compactness fails for gauged maps in the case that a
  central subgroup \(\C^\times_\zeta\)
  acts trivially.  Indeed, in this case, the energy
  \( \lan d , c_1(L) \ran\)
  of a gauged map of class \(d\)
  does not determine the isomorphism class of the bundle, since
  twisting by a character of \(\C^\times_\zeta\)
  does not change the energy.
\end{remark} 

\subsection{The wall-crossing formula} 
 
We may now prove the quantum version of the Kalkman-Metzler formula. 

\begin{theorem}[Wall-crossing for gauged potentials]
  \label{gwall} Let \(X\) be a smooth \(G\)-variety.
  Suppose that \(\LL_\pm \to X\) are polarizations such that
  there is simple wall-crossing.  Then the gauged
  Gromov-Witten potentials are related by
\begin{equation}
 \tau^G_{X,+} - \tau^G_{X,-} =
\tau_{X,\zeta,0} 
\end{equation}
where the same
Mundet semistability parameter should be used to define the potentials
on both sides of the equation.
\end{theorem}

\begin{proof}[Proof of Theorem \ref{gwall}] The statement follows from
  virtual Kalkman-Metzler formula \ref{vkalkman} applied to the master
  space \(\ol{\M}_n^G(C,X,\LL_-,\LL_+)\)
  and the identification of fixed point contributions as reducible
  gauged maps described in sections \ref{master2},\ref{ss:redgauge}.
\end{proof}

Combining Theorem \ref{gwall} with the adiabatic limit 
\eqref{largerel} yields:

\begin{theorem} [Quantum Kalkman-Metzler formula]
\label{kalk3}
Suppose that \(X\) is equipped with polarizations \(\LL_\pm\) so
that the wall crossing is simple (the only singular
polarisation is \(\LL_0\) ). Then the Gromov-Witten
invariants of \(X \qu_\pm G\) are related by twisted gauged
Gromov-Witten invariants with smaller subgroup \(G_\zeta
\subset G\)  
\begin{equation} 
  \tau_{X \qu_+ G}\ \kappa_{X,+}^G - \tau_{X
  \qu_- G}\ \kappa_{X,-}^G = \lim_{\rho \to 0}
  \tau_{X,\zeta,0}.
\end{equation}
\end{theorem}  

In other words, failure of the following square 
\begin{equation}\label{diag:wall}
\begin{tikzcd} 
  QK^0_G(X,\LL_-)  \arrow{d}[swap]{\kappa_{X,-}^G}
& {QK^{0,\fin}_G(X)}  
\arrow{l} \arrow{r}& 
{QK^0_G(X,\LL_+)}
 \arrow{d}{\kappa_{X,+}^G} \\ 
 {QK(X \qu _- G)} \arrow{d}[swap]{\tau_{X \qu_- G}} & &
  {QK(X \qu _+ G)} \arrow{d}{\tau_{X \qu_+ G}} \\
  \Lambda_{X,\LL_-}^G \arrow{r} &\ti{\Lambda}_X^G &
  \Lambda_{X,\LL_+}^G
\arrow{l}
\end{tikzcd}
\end{equation}
to commute is measured by an explicit sum of wall-crossing terms given
by the contribution of the fixed gauged potential.  We remark that if
the wall-crossing is not simple, the contributions on the right-hand
side of the wall-crossing formula might come from several singular
values \(t\in (-1,1)\)
as the polarisations \(\LL_t\)
varies; however a simple modification of the argument above proves it
as well.

\section{Crepant wall-crossing}%
\label{cy}

In this section we use the quantum Kalkman-Metzler formula to prove a
version of the crepant transformation conjecture for K-theoretic
Gromov-Witten invariants, under some rather strong assumptions on the
weights involved in the wall-crossing.  We assume the following
symmetry condition on the weights involved in the wall-crossing.
Suppose we have a birational transformation of git type
\[ \phi : X \qu_- G \dashrightarrow X \qu_+ G \] 
defined by a simple wall-crossing induced by two polarisations
\(\LL_+, \LL_-\)
as in the previous sections.  Suppose that for \(\zeta\in \g\),
the fixed point component \(X^{\zeta,0}\)
is the one contributing to the wall-crossing term, and let
\(\nu_{X^{\zeta,0}}\to X^{\zeta,0}\)
be its normal bundle in \(X\).  Let
\[
  \nu_{X^{\zeta,0}} = \bigoplus_{j} \nu_j 
\]
be the isotypical decomposition so that
\(\Ct_\zeta\) acts  on \(\nu_j\) with weight \(\mu_j\). Note as
before that all the \(\mu_j\neq 0\). Let
\(r_j=\rank \nu_j\). 

\begin{definition} \label{crepantdef} The birational transformation
  \(\phi: X \qu_- G \to X \qu_+ G \)
  is {\em simply crepant} if the set of weights $\mu_i$ of the
  normal bundle of $X^{\zeta,0}$ in $X$ is invariant under
  multiplication by $-1$, that is, whenever $\mu_j$ is a weight with
  multiplicity $r_j$ then so is $-\mu_j$ with the same multiplicity.
\end{definition}
\noindent If the wall-crossing is not simple, it is simply crepant
if the condition in \ref{crepantdef} holds for all fixed point
components contributing to the wall-crossing terms.

We show invariance for the gauged potentials under crepant
wall-crossing if a certain symmetrised version of the Euler
characteristics are used.  Let \(T\)
be a torus acting on a Deligne-Mumford stack \(\XX\),
endowed with a perfect obstruction theory.  Suppose \(x\in \XX^T\)
is an isolated fixed point. Locally the virtual tangent space
\[
  T^{\vir}_x := \on{Def}_x - \on{Obs}_x. 
\]
can be decomposed as
\[
  T^{\vir}_x = \bigoplus_i \C_{a_i} - \bigoplus_j \C_{b_j}
\]
where \(a_i,b_j\)
are the weights of the deformation and obstruction spaces
respectively.  Define
\[
\widehat{\cO}^{\vir}_{\XX}:= \cO^{\vir}_{\XX,x}\tensor
(K^{\vir}_{\XX})^{1/2}, \quad K^{\vir}_{\XX}: = (\det
T^{\vir}_{\XX})^{-1} \] 
where a square root can be defined in rational K-theory via the Chern
character \cite{ed:rr}.  The resulting K-theoretic Gromov-Witten
invariants obtained by replacing the virtual structure sheaf by this
shift quantum K-theory at level $-1/2$ in the language of Ruan-Zhang
\cite{rz:level}.  At an isolated fixed point $x$ we have
\[
\widehat{\cO}^{\vir}_x:=
\frac{\prod_j (b^{1/2}_j-b^{-1/2}_j)}{\prod_i (a^{1/2}_i - a^{-1/2}_i)}
\]
where \(a_i^{1/2}, b^{1/2}_j\)
are formal, since they represent weights only after passing to a cover
\(\hat{T}\to T\).

The virtual localization formula may be re-written in terms of the
shifted structure sheaves.  Let $\hat{A} (\cdot) $ be the denominator
of the A-hat genus, mapping $R(T)$ to the space of functions defined
on some cover
\[
  \widehat{A}(a_1+a_2)=\widehat{A}(a_1) \widehat{A}(a_2) ;\
  \ \  \widehat{A}(a)
  =\frac{1}{a^{1/2}-a^{-1/2}}.
\]
where \(a\) is a weight (representation) of \(T\).  Define
an extension to \(\cF\in K_T^0(X)\) by 
\[
  \widehat{A}(\cF)= \prod_j \widehat{A}(y_j), 
\]
where the product runs over the equivariant Chern roots
\(y_j \in K_T^0(X) \) of \(\cF\).  Then localization 
\eqref{eq:vloc} becomes
\begin{equation}\label{eq:symvlociso}
  \widehat{\cO}^\vir_\XX = \iota_*
  (\widehat{A}(T^{\vir}_{\XX^T})
  \ \widehat{\cO}^\vir_{\XX^T}).
\end{equation}
This can be made more explicit as follows.  For each component
\(F\subset \XX^{T}\)
we have a decomposition 
\[T^{\vir}_{\XX}|_{F}= T^{\vir}_F + \nu_F\]
and therefore 
\[ (K_{\XX}|_F)^{1/2}= K_F^{1/2}(\det \nu_F)^{-1/2} .\]
It follows that 
\begin{equation}
  \label{eq:restfix}
  \frac{\cO_F\otimes (K^{\vir}_\XX|_F)^{1/2}}{\Eul_T(\nu_F)}
  =\cO_F \otimes (K^{\vir}_F)^{1/2}\otimes \frac{(\det
	\nu_F)^{-1/2}}{\Eul_T(\nu_F)}.
\end{equation}
By considering the decomposition 
\begin{equation}\label{eq:isotypical}
\nu_F=\bigoplus_i z^{\mu_i} \nu_{F,i}, 
\end{equation}
in isotypical components, we have 
\[
  \Eul_{T}(\nu_F)= \prod_i \Eul_T( z^{\mu_i} \nu_{F,i})=
  \prod_{i,j}
  (1-z^{-\mu_i} x_{i,j}^{-1})
\]
where \(x_{i,j}\) are the Chern
roots of \(\nu_{F,i}\). Since \( (\det \nu_F)^{1/2}=\prod_{i,j}
(z^{\mu_i }x_{ij})^{1/2}\) we have 
\begin{equation}\label{eq:symaux}
  \Eul_T (\nu_F)^{-1}  (\det \nu_F)^{-1/2}=
  \prod_{i,j}  \widehat{A}(z^{\mu_i} x_{ij})^{-1} =
  \widehat{A}(\nu_{F})^{-1}.
\end{equation}

For our arguments below, we need to discuss 
the asymptotic behaviour of
\(\widehat{A}(\nu_F)\). Consider the decomposition of
\(\nu_F\) as in 
\eqref{eq:isotypical} and the Euler class expansion
\eqref{eq:eulerexp} for each of its isotypical components.
Thus
\[
  \widehat{A}(\nu_{F})=\frac{(\det \nu_F)^{-1/2}}{\prod_i
	(1-z^{-\mu_i})^{r_i}(1+N_i)} 
\]
with \(N_i\in K(F)\otimes K^{\loc}_T(\pt)\)
as in \eqref{eq:sumnilpotents}. Therefore
\begin{equation}
	\label{eq:ahat}
	\widehat{A}(\nu_F) = \prod_i \left(
	\frac{z^{-\mu_i/2}}{1-z^{-\mu_i}}\right)^{r_i} \cdot
	O(z)=\prod_i \widehat{A}(\C_{\mu_i})^{r_i} \cdot
	O(z),
\end{equation}
where \(\C_{\mu_i}\)
is the representation with weight \(\mu_i\)
and \(O(z)\) is a term that converges to zero as \(z^{\pm 1} \to 0 \).

\subsubsection{Symmetrised wall-crossing.}

We can define symmetric versions of the gauged K-theoretic
Gromov-Witten potentials previously studied by considering Euler
characteristics with respect to \(\widehat{\cO}^\vir\).
  In the following, we add a hat to any
expression whose definition now uses \(\widehat{\cO}^\vir\)
rather than \(\cO^\vir\).
The proof of the quantum Kalkman formula in Theorem \ref{kalk3} relied
on virtual localisation. If instead we use localisation for the
symmetrised virtual structure sheaf we obtain the following:
\begin{equation} 
  \widehat{\tau}_{X \qu_+ G}\ \widehat{\kappa}_{X,+}^G -
  \widehat{\tau}_{X
  \qu_- G}\ \widehat{\kappa}_{X,-}^G = \lim_{\rho \to 0}
  \widehat{\tau}_{X,\zeta,0}.
\end{equation}

The symmetrised virtual structure sheaves satisfy good properties
under the action of the Picard stack on the locus of reducible maps.
Let
\[
  \Pic(C) := \Hom(C, B\C^\times)
\]
denote the Picard stack of line bundles on \(C\).
The Lie algebra \(\g_\zeta\)
has a distinguished factor generated by \(\zeta\),
and using an invariant metric the weight lattice of \(\g_\zeta\)
has a distinguished factor \(\Z\)
given by its intersection with the Lie algebra of \(\C^\times_\zeta\).
After passing to a finite cover, we may assume that
\(G_\zeta \cong (G_\zeta/\Ct_\zeta) \times \Ct_\zeta\).
The Picard stack \(\Pic(C)\)
acts on the moduli stack of reducible gauged maps
\(\ol{\M}_n^{G_\zeta}(C,X,\LL_0,\zeta)\)
as follows.  Recall that a reducible gauged map \((P,\hat{C},u)\),
where \(P \to C\)
is a \(G\)-bundle
and \(u: \hat{C} \to P(X)\)
is \(\zeta\)-fixed.
The restriction of \(u\)
to the principal component of \(C\)
maps into the fixed point locus \(X^\zeta\).
For \(Q\)
an object of \(\Pic(C)\)
and \((P,\hat{C},u)\)
an object of \(\ol{\M}_n^{G_\zeta}(C,X,\LL_0,\zeta)\) define
\begin{equation} \label{picact}
 Q (P,\hat{C},u) := 
 (P \times_{\C^\times_\zeta} Q, \hat{C}, v) 
\end{equation} 
where the section \(v\)
is defined as follows: We have an identification of bundles
\((P \times_{\C^\times_\zeta} Q) (X^\zeta) \cong P(X^\zeta)\)
since the action of \(\C^\times_\zeta\)
on \(X^\zeta\)
is trivial. Hence the principal component \(u_0\),
which is a section of \(P(X^\zeta)\)
induces the corresponding section \(v_0\)
of \(( P \times_{\C^\times_\zeta} Q) (X^\zeta)\).
Each bubble component of \(u\)
maps into a fiber of \(P(X)\),
canonically identified with \(X\)
up to the action of \(G_\zeta\).
So $u$ induces the corresponding bubble map of \(v\)
into a fiber of \(( P \times_{\C^\times_\zeta} Q) (X)\),
well-defined up to isomorphism. Note that if the degree of
\((P,\hat{C},u)\)
is \(d\) the degree of \(Q(P,\hat{C},u)\) is \(d+c_1(Q)\).

The Picard action preserves semistable loci in the large area
limit. Indeed, because the Mundet weights \(\mu_M(\sigma,\lambda)\)
approach the Hilbert-Mumford weight \(\mu_{HM}(\sigma,\lambda)\)
as \(\rho \to 0\),
the limiting Mundet weight is unchanged by the shift by \(Q\)
in the limit \(\rho \to 0\)
and so Mundet semistability is preserved. Thus for \(\rho^{-1}\)
sufficiently large the action of an object \(Q\)
of \(\Pic(C)\) induces an isomorphism
\begin{equation} \label{translate} \cS^\delta:
  \ol{\M}_n^{G_\zeta}(C,X,\LL_0,\zeta,d) \to
  \ol{\M}_n^{G_\zeta}(C,X,\LL_0,\zeta,d + \delta)
\end{equation}
where \(\delta = c_1(Q)\). 

\begin{lemma}\label{pic} {\rm (Action of the Picard stack on fixed
    loci)} The action of \(Pic(C)\)
  in \eqref{translate} induces isomorphisms of the relative
  obstruction theories on $\ol{\M}_n^{G_\zeta}(C,X,\LL_0,\zeta,d) $
  preserving the restriction of symmetrised virtual structure sheaves
  $\hat{\cO}_{\ol{\M}_n(C,X,\LL_0)}^{\vir}$, and preserving the class
  \(\ev^* \alpha\) for any \(\alpha \in K_G^0(X)^n\).
\end{lemma}

\begin{proof} The action of \(\Pic(C)\) lifts to the universal
  curves, denoted by the same notation. Since the relative
  part of the obstruction theory on
  \(\ol{\M}_n^{G_\zeta}(C,X,\LL_0,\zeta,d)\) is the
  \(\C^\times_\zeta\)-invariant part of \(Rp_* e^*
  T_{X/G_\zeta}^\dual\) up to the factor \(\C\zeta\), 
  the isomorphism
  preserves the relative obstruction theories and so the
  virtual structure sheaves. (Note that
  on the principal component, the invariant part is \(Rp_*
  e^* T_{X^\zeta/G_\zeta}^\dual\) which is unchanged by the tensor
  product by \(\C^\times_\zeta\)-bundles, while on the bubble
  components \(Rp_* e^* T_{X/G}^\dual\) is unchanged by the tensor
  product by \(Q\).) Since the evaluation map is unchanged by
  pull-back by \(\cS^\delta\) (up to isomorphism given by
  twisting by \(Q\)), the class \(\ev^* \alpha\) is preserved.
\end{proof}

\begin{theorem}[Wall-crossing for crepant birational transformations of git type]
\label{cytype} 
Suppose that \(X,G, \LL_\pm\)
define a simple wall-crossing, and \(C\)
has genus zero. If the wall-crossings is simply crepant then
\[ 
  \widehat{\tau}_{X \qu_- G} 
\circ  \widehat{\kappa}_{X,-}^G  
 \underset{a.e.}{=} \widehat{\tau}_{X \qu_+ G}
\circ \widehat{\kappa}_{X,+}^G  
\]
almost everywhere (a.e.) in the quantum parameter \(q\).
\end{theorem}

The following remark explains precisely in what sense
\emph{a.e.} is used in Theorem \ref{cytype}.

\begin{remark} \label{distrib} In the Schwartz theory of distributions
  (H\"ormander \cite{ho:an}) denote by \(\mT(S^1) \)
  the space of {\em tempered distributions}.  Fourier transform
  identifies \(\mT(S^1)\)
  with the space of functions on \(\Z\)
  with polynomial growth.  The variable \(q\)
  is a coordinate on the punctured plane \(\C^\times\)
  and any formal power series in \(q,q^{-1}\)
  defines a distribution on \(S^1\),
  which is tempered if the coefficient of \(q^d\)
  has polynomial growth in \(d\).
  In particular the series \(\sum_{d \in \Z} q^d\)
  is the delta function at \(q = 1\),
  and its Fourier transform is the constant function with value \(1\).
 Any distribution of the form
\(\sum_{d \in \Z} f(d) q^d\), for \(f(d)\) polynomial, is
a sum of derivatives of the delta function
(since
Fourier transform takes multiplication to
differentiation) and so is almost everywhere zero.
\end{remark}

We study the dependence of the fixed point contributions
\(\tau_{X,\zeta, d,0} \) with respect to the Picard action
defined in \eqref{picact}.  Suppose that \(Q\) is a
\(\C^\times_\zeta\)-bundle of first Chern class the generator
of \(H^2(C)\), after the identification \(\C^\times_\zeta \to
\C^\times\).   Denote the corresponding class in
\(H_2^{G_\zeta}(X^\zeta)\) by \(\delta\).  Consider the action
of the \(\Z\)-subgroup \(\Z_Q\subset \Pic(C)\) generated by
\(Q\). The contribution of any component
\(\ol{\M}_n^{G_\zeta}(C,X,\LL_0,\zeta,d)\) of class \(d \in
H_2^G(X)\) differs from that from the component induced by
acting by \(Q^{\otimes r}, r\in \Z_Q\), of class \(d + r
\delta\), by the ratio of symmetrised Euler classes of the
moving parts of the virtual normal complexes
\begin{equation} \label{eulerdiff}
  \widehat{A}( (Rp_* e^*
  T_{X/G}^{\dual})^+) \widehat{A}(
  \cS^{r\delta,*} (Rp_* e^* T_{X/G}^{\dual})^+)^{-1} 
\end{equation}

As before, denote by \(\nu_i\)
be the subbundle of \(\nu_{X^{\zeta,0}}\) of weight \(\mu_i\).

\begin{lemma} The $\widehat{A}$ classes relate by
\[  \widehat{A}( (Rp_* e^* T_{X/G}^{\dual})^+) = \widehat{A}(
  \cS^{r\delta,*} (Rp_* e^* T_{X/G}^{\dual})^+) \left( \prod_{i}
    \widehat{A}\left(\nu_i\right)^{\mu_i}\right)^r.   \] 
\end{lemma}

\begin{proof}
  The Grothendieck-Riemann-Roch allows a computation of the Chern
  characters of the (representable) push-forwards.  Consider the isotypical
  decomposition into \(\C^\times\)-bundles
  of the normal bundle to the fixed component \(X^{\zeta,0}\) in \(X\)
\[ \nu_{X^{\zeta,0}}  = \bigoplus_{i=1}^k
\nu_{i} 
\]
where \(\C^\times\)
acts on \(\nu_{i}\)
with non-zero weight \(\mu_i \in \Z\).
By the discussion above \(e^* T^\dual_{X/G}\)
is canonically isomorphic to \(\cS^{r\delta} (e^* T^\dual_{X/G})\)
on the bubble components, since the \(G\)-bundles
are trivial on those components.  Because the pull-back complexes are
isomorphic on the bubble components, the difference
\((e^* T^\dual_{X/G})^+ - \cS^{r\delta,*} (e^* T^\dual_{X/G})^+ \)
is the pullback of the difference of the restrictions to the principal
part of the universal curve, that is, the projection on the second
factor
\[p_0: C \times \ol{\M}_n^{G_\zeta}(C,X,\LL_0,\zeta) \to
\ol{\M}_n^{G_\zeta}(C,X,\LL_0,\zeta) .\]
These restrictions are given  by 
\begin{eqnarray*}
  (e^* T_{X/G})^{+,\on{prin}} &\cong& \bigoplus_{i}^{}
e^* \nu_{X^{\zeta,t},i} \\   \cS^{r\delta,*} (e^*
T_{X/G})^{+,\on{prin}} &\cong& \bigoplus_{i} e^*
\nu_{X^{\zeta,t},i} \otimes (e_C^* Q
\times_{\C^\times_\zeta} \C_{r\mu_i}) 
\end{eqnarray*}
where \(e_C\) is the map from the universal curve to \(C\).  The
projection \(p_0\) is a representable morphism of stacks given
as global quotients.  Let 
\[\sigma: \ol{\M}_n^{G_\zeta}(C,X,\LL_0,\zeta) \to C \times  
\ol{\M}_n^{G_\zeta}(C,X,\LL_0,\zeta)\] 
be a constant section of \(p_0\),
so that \(c_1(\sigma^* e^* \nu_{X^{\zeta,0},i})\)
is the ``horizontal'' part of \(c_1(\nu_{X^{\zeta,0},i})\).
By Grothendieck-Riemann-Roch for such
stacks \cite{toen:rr}, \cite{ed:rr}
\begin{eqnarray*} \Td_\M  && \Ch(\cS^{r\delta,*} \Ind(
  T_{X/G})^+) = p_{0,*} (\Td_{C \times \M}
  \Ch(\cS^{r\delta,*}  T_{X/G})^+) \\ &=& (1-g) + \Td_\M
  p_{0,*} \Ch(\cS^{r\delta,*}  T_{X/G})^+) \\ &=& (1-g) +
  \Td_\M p_{0,*} \sum_{i}^{} \Ch(e^*
  \nu_{X^{\zeta,t},i}) \Ch( (e_C^* Q
  \times_{\C^\times_\zeta} \C_{r\mu_i})) \\ &=& (1-g) +
  \Td_\M p_{0,*} \sum_{i}^{} \Ch(e^*
  \nu_{X^{\zeta,t},i}) (1 + r \mu_i \omega_C) \\ &=&
  p_{0,*}( \Td_{C \times \M} \Ch\left(\Ind(T_{X/G})^+ \oplus
  \bigoplus_{i}^{} (\sigma^* e^*
  \nu_{X^{\zeta,t},i})^{\oplus r \mu_i}\right) \\ &=& \Td_\M
  \Ch\left(\Ind(T_{X/G})^+ \oplus \bigoplus_{i}^{} (\sigma^*
  e^* \nu_{X^{\zeta,0},i})^{\oplus r \mu_i}\right).
\end{eqnarray*}
Hence
\begin{equation} \label{diff} \Ch(\cS^{r\delta,*} \Ind(
  T_{X/G})^+) = \Ch\left(\Ind(T_{X/G})^+ \oplus
  \bigoplus_{i=1}^{m} (\sigma^* e^* \nu_{X^{\zeta,0},i})^{\oplus
  r \mu_i}\right) 
\end{equation}
 The equality of Chern characters above implies an
 isomorphism in rational topological \(K\)-theory. 
 The difference in Euler classes
 \eqref{eulerdiff} is therefore given by the Euler class of
 the last summand in \eqref{diff}
 \begin{eqnarray*} \frac{\widehat{A}(
  (Rp_* e^* T_{X/G}^\dual)^+)}{
	\widehat{A}( \cS^{r\delta,*}
	(Rp_* e^* T_{X/G}^\dual)^+)} 
&=&
\widehat{A}\left(
	\bigoplus_{i} (\sigma^* e^*
	(\nu_i))^{\oplus \mu_ir}
	\right) \\ 
&=& 
\left(
\prod_{i}
\widehat{A}\left(\nu_i\right)^{\mu_i}\right)^r. \qedhere
\end{eqnarray*}
\end{proof}

\begin{proof}[Proof of Theorem \ref{cytype}]
  Using the expansion of Euler classes as in \eqref{eq:ahat} and
  \eqref{eq:eulerexp}, we have that by setting \(r_i = \rank \nu_i\)
(on each component of the inertia stack, in the orbifold case)

\begin{eqnarray*}
\prod_{i}
\widehat{A}\left(\nu_i\right)^{-\mu_i }
&=& \prod_{i} ( \zeta_i^{1/2} z^{\mu_i/2} - \zeta_i^{-1/2} z^{-\mu_i/2})
^{-\mu_i r_i }  (1+N)
\end{eqnarray*}
where $\zeta_i$ are the roots of unity appearing in orbifold
Riemann-Roch \cite{to:rr} and $N$ is nilpotent.  Let \(S_-, S_+\)
denote the indices for which \(\mu_i\)
is negative and respectively positive. Define
\begin{equation}\label{eq:balanced}
  \Delta(z):=\frac{\prod_{j\in S_-} ( \zeta_j^{1/2}  z^{\mu_j/2} -
  \zeta_j^{-1/2} z^{-\mu_j/2})^{-\mu_j r_j }}{\prod_{i\in
    S_+}(\zeta_i^{1/2} z^{\mu_i/2} -
  \zeta_i^{-1/2}  z ^{-\mu_i/2})^{ \mu_i r_i} } 
\end{equation}
We can rewrite the difference
\begin{eqnarray*}
	\label{eq:ahatfinal}
  \left(
  \prod_{i}
  \widehat{A}\left(\nu_i\right)^{-\mu_i}\right)^r
  &=& (\Delta(z)(1+N))^r
\end{eqnarray*}
By the crepant wall-crossing assumption \ref{crepantdef}, the function
$\Delta(z)$ is a constant, denoted $\Delta$.  Summing the terms from
\(\cS^{r\delta}, r \in \Z\)
we obtain that the wall-crossing contribution is
\begin{equation} \label{type} \sum_{r \in \Z} q^{d + \delta r}
  \widehat{\tau}_{X,\zeta, d + \delta r,0}(\alpha) =
  \sum_{r \in \Z} q^{d + r \delta}.
  \chi_0(r)
\end{equation}
where $\chi_0(r)$ is polynomial in $r$, since each $N$ is nilpotent
and the binomial coefficients from the expansion of $(1+N)^r$ are
polynomial.  Now
\[
\sum_{r \in \Z} q^{\delta r} \in \mT(S^1)
\]
is a delta function and it vanishes almost everywhere in
\(q^{\delta}\),
see Remark \ref{distrib}.  Since $\chi_0(r)$ is polynomial,
\eqref{type} is the derivative of a delta function. Since
\[ \widehat{\kappa}_X^{G,+} \widehat{\tau}_{X \qu_+ G} -
\widehat{\kappa}_X^{G,-} \widehat{\tau}_{X \qu_- G} \] 
is a sum of wall-crossing terms of the type in \eqref{type}, this
completes the proof of Theorem \ref{cytype}.
\end{proof}

\section{Abelianization}% (fold) 

In this section we compare the K-theoretic Gromov-Witten invariants of
a git quotient with the quotient by a maximal torus, along the lines
of the case of quantum cohomology investigated by
Bertram-Ciocan-Fontanine-Kim \cite{be:qu} and
Ciocan-Fontanine-Kim-Sabbah \cite{ciocan:abnonab}.  The analogous
question for K-theoretic $I$-functions of git quotients was already
considered in Taipale \cite{taipale} as well as Wen \cite{wen} and,
around the same time as the first draft of this paper, Jockers, Mayr,
Ninad, and Tabler \cite{jockers}.

Recall that the equivariant cohomology may be identified with the 
Weyl-invariant equivariant cohomology for  the action of a maximal torus.  We assume
that \(G\)
is a connected complex reductive group. Choose a maximal torus \(T\)
and \(W\)
its Weyl group.  By a theorem of Harada-Landweber-Sjamaar
\cite[Theorem 4.9(ii), Section 6]{hls:kt} if \(X\)
is either a smooth projective \(G\)-variety
or a \(G\)-vector
space then restriction from the action of \(G\)
to the action of the torus \(T\)
defines an isomorphism onto the space of \(W\)-invariants
\[ 
  \on{\Restr}^G_T: K_G^0(X) \cong K_T^0(X)^W 
\]
for either the topological or algebraic \(K\)-cohomology.
Given a polarisation \(\LL\to X\)
of the \(G\)
action, consider the naturally induced \(T\)
polarisation on $X$ 
so that 
\[ X^{\ss,G}(\LL)\subset X^{\ss,T}(\LL). \]
We assume from now in this section that \(QK_G^0(X)\)
denotes the algebraic equivariant quantum \(K\)-cohomology.
We relate the K-theoretic potentials of the two geometric invariant
theory quotients \(X\qu G,\)
and the \emph{abelian} quotient \(X\qu T\).
Let \(\nu_{\g/\t}\)
denote the bundle over \(X\qu T\)
induced from the trivial bundle over \(X\)
with fibre \(\g/\t\).  Consider the graph potential
\[ \tau_{X\qu T}:QK_T^0(X)\to \Lambda_X^T \] 
twisted by the Euler class of the index bundle associated to
\(\g/\t\):
\begin{multline} 
  \tau_{X\qu T} (\alpha,q) := \\ \sum_{  n \ge 0 } \sum_{d \in
                                      H_2^G(X,\Q) }  
  \chi^{\vir}\left( \ol{\M}_n(C,X\qu T,d);  
  \ev^*\alpha^n \Eul (\Ind(\g/\t)) \right)  
  \frac{q^d}{n!}.
\end{multline}
Similarly \(T\)-gauged
potential \(\tau^T_X\)
and the Kirwan map \(\kappa_{X,T}\)
will from now on denote the maps with the Euler twist above.  The
natural map \(H_2^T(X,\Q) \to H_2^G(X,\Q)\) induces a map of Novikov rings
\[
  \pi_T^{G}: \Lambda_X^T\to \Lambda_X^G,
  \sum_{d\in H_2^T(X)} c_q q^d\mapsto \sum_{d\in H_2^G(X)}
  c_q q^{\pi(d)}.
\]
By abuse of notation, denote again by
\[
  \Restr: QK_G^0(X)\to QK_T^0(X)
\]
the map induced by the restriction map above and the inclusion of the
Weyl invariants
\(\Lambda_X^G\cong (\Lambda_X^T)^W\subset \Lambda_X^T.\)
As in Martin \cite{mar:sy} the restriction map
\[
  \Restr_G^T: K(X\qu T)^W \to K(X\qu G)
\]
is surjective and has kernel is the
annihilator of \(\Eul(\g/\t)\),  the set
of classes that vanish when multiplied by \(\Eul(\g/\t)\).
This map naturally  extends to a map 
\[ \Restr_G^T:   QK(X \qu T)^W \to QK(X \qu G) \] 
by a similar definition on the twisted sectors.  On the main sector
$\Restr_G^T$ is given by restriction of a class
\[ \alpha\in K(X\qu T)=K(X^{\ss}(T)/T) \]  
to \(X^{\ss}(G)/T\)
then followed by the identification of the Weyl invariant part with
\(K(X\qu G)\)
\cite[Theorem A]{mar:sy}.  With these notations we have the following
result.

\begin{theorem} {\rm (Quantum Martin formula in quantum
    \(K\)-theory)}
  \label{qkabel} Let \(C\)
  be a smooth connected projective genus \(0\)
  curve, \(X\)
  a smooth projective or convex quasiprojective \(G\)-variety,
  and suppose that stable=semistable for \(T\)
  and \(G\)
  actions on \(X\).
  The following equality holds on the topological quantum K-theory
  \(QK_G^0(X)\):
\[ \tau_{X \qu G} \circ \kappa_{X,G} = |W|^{-1} \pi_T^G
\circ \tau_{X \qu T}^{\g/\t} \circ \kappa_{X,T}^{\g/\t} \circ \Restr_T^G
.\]
Similarly for  (J-functions) localised graph potentials
\[ \begin{array}{l}  \tau_{X\qu G,-} : QK_G^0(X)\to QK(X\qu G)[z, {z}^{-1}]] \\
 \tau_{X\qu T,-} : QK_T^0(X)\to QK(X\qu T)[z, {z}^{-1}
     ]] \end{array} \] 
we have 
 \begin{eqnarray} \nonumber
	\tau_{X \qu G,-} \circ 
\kappa_{X,G} &=&
                                                                    \tau_{X,-}^G \\ 
&=&  \label{lab}  \Restr_G^T  \circ 
\Eul(\g/\t)^{-1} \tau_{X,-}^{T,\g/\t} \circ   \Restr_T^G \\ 
& = &  \nonumber 	\Restr_G^T  \circ \Eul(\g/\t)^{-1} 
\tau_{X \qu T,-}^{\g/\t}
      \circ \kappa_{X,T}^{\g/\t} \circ \Restr_T^G. 
 \end{eqnarray} 
\end{theorem}
 
In particular commutativity of the following diagram holds: 
\begin{equation*}
\begin{tikzcd}[every arrow/.append style={-latex}]
  QK_G^0(X) \arrow[d,"\kappa_{X,G}"'] \arrow[r,"\Restr_T^G"] &
  QK_T^0(X)^{W} \arrow[d,"\kappa_{X,T}"]\\ 
  QK(X \qu G) \arrow[d,"\tau_{X \qu G}"'] & QK(X
  \qu T) \arrow[d,"|W|^{-1} \tau_{X \qu T}"] \\
  \Lambda_X^G & \Lambda_X^T  \arrow[l,"\pi_T^G"]
\end{tikzcd}	
\end{equation*}

\begin{proof}[Sketch of proof] The argument is the same as that for
  cohomology in \cite[Section 4]{cross}.  In the case of projective
  target $X$, one can vary the vortex parameter $\rho \in \R_{> 0 }$
  until one reaches the small-area chamber in which every bundle
  $P \to C$ appearing in the vortex moduli space is trivial (in genus
  zero).  Indeed, for $\rho^{-1}$ sufficiently large the Mundet weight
  is dominated by the Ramanathan weight, and this forces the bundle to
  be semistable of vanishing Chern class and so trivial.  It follows
  that both $\ol{\M}^G(C,X,d)$ and $\ol{\M}^T(C,X,d)$ \label{CXD} are quotients of
  open loci in the moduli stacks of stable maps $\ol{\M}_{0,n}(X,d)$
  by $G$ resp. $T$.  In the small-area chamber abelianization holds
  for the localized potentials $\tau_{X,G}$ and $\tau_{X,T}$ by
  virtual non-abelian localization \cite{hal:strat}: For sufficiently
  positive equivariant vector bundles $V$ over $\ol{\M}_{0,n}(X,d)$
  denote by $V \qu G$ the restriction to $\ol{\M}^G(C,X,d)$.  Then \label{CX}
 \begin{eqnarray*} 
 \chi( \ol{\M}^G(C,X,d), V \qu G) &=& \chi(\ol{\M}_{0,n}(X,d),V)^G  \\
&=&
  |W|^{-1} \chi(\ol{\M}_{0,n}(X,d),V \otimes \Eul(\g/\t))^T \\
&=& |W|^{-1}
  \chi(\ol{\M}^T(C,X,d), ( V \otimes \Eul(\g/\t)) \qu T) \end{eqnarray*}
where the second equality holds by the Weyl character formula.  If the
stabilizers are at most one-dimensional then the wall-crossing formula
of \cite{wall} implies that the variation in the gauged Gromov-Witten
invariants $\tau_X^G$ with respect to the vortex parameter $\rho$ is
given by wall-crossing terms $\tau_{X,\zeta}^{G_\zeta}$ involving
smaller-dimensional structure group given by the centralizers
$G_\zeta, \zeta \in \g$ of one-parameter subgroups generated by
$\zeta$: For any singular value $\rho_0$ and
$\rho_\pm = \rho_0 \pm \eps$ for $\eps$ small we have
  \begin{equation} \label{wcross} \tau_{X,d}^G(\alpha,\rho_+) -
    \tau_{X,d}^G(\alpha,\rho_-) = \sum_\zeta
    \tau_{X,\zeta}^{G_\zeta}(\alpha,\rho_0) \end{equation}
  where $\tau_{X,\zeta}^{G_\zeta}$ is a moduli stack of
  $\rho_0$-semistable gauged maps fixed by the one-parameter subgroup
  generated by $\zeta$ as in Section \ref{qwall}.  After possibly
  adding a parabolic structure the stabilizers of gauged maps that are
  $\rho_0$-semistable are one-dimensional and so the wall-crossing
  formula \eqref{wcross} holds.  Furthermore, the fixed point
  components have structure group that reduces to
  $G_\zeta/\C^\times_\zeta$, which as such that objects in the fixed
  point components have trivial stabilizer.  By induction we may
  assume that the abelianization formula holds for structure groups
  $G_\zeta/\C^\times_\zeta$ of lower dimension, and in particular for
  the invariants $\tau_{X,\zeta}^{G_\zeta}(\alpha,\rho_0)$.  The
  result for other chambers $\rho \in (\rho_i,\rho_{i+1})$ holds by
  the wall-crossing formula Theorem \ref{gwall} since, by the
  inductive hypothesis, the wall-crossing terms are equal.  The
  conclusion for the git quotients then follows from the adiabatic
  limit theorem \ref{Jresults}.

  In the case of quasiprojective $X$ we assume that $G$ has a central
  factor $\C^\times \subset G$ and the moment map $\Phi: X \to \R$ for
  this factor on $X$ is bounded from below.  \label{quasiprojective}
  Then a similar wall-crossing argument obtained by varying the
  polarization $\lambda(t) \in H^2_G(X)$ from $\lambda(0) = \omega$ to
  to a chamber corresponding to an equivariant symplectic class
  $\lambda \in H^2_G(X)$ where $X \qu_{\lambda(1)} G$ is
  empty, produces the same result \cite{cross}. Indeed the moduli space
  of gauged maps $\ol{\M}_n^G(\P^1,X)$ for the polarization
  $\lambda(1)$ is also empty, since for $\rho$ small elements of
  $\ol{\M}_n^G(\P^1,X)$ must be generically semistable.  On the other
  hand, the wall-crossing terms correspond to integrals over gauged
  maps whose structure group $G_\zeta$ is the centralizer of some
  one-parameter subgroup generated by a rational element
  $\zeta \in \t$.  By induction on the dimension of $G_\zeta$ we may
  assume that the wall-crossing terms for $\ol{\M}_n^G(\P^1,X)$ and
  $\ol{\M}_n^T(\P^1,X)$ are equal and we obtain the abelianization
  formula by induction.

  For the localized potentials the same wall-crossing argument applied
  to the $\C^\times$-fixed point components
  $\ol{\M}_n^G(\P^1,X)^{\C^\times}$ and
  $\ol{\M}_n^T(\P^1,X)^{\C^\times}$ produces the abelianization
  formula
 \begin{multline}
 \chi( \ol{\M}_n^G(\P^1,X)^{\C^\times}, V \otimes \Eul(\nu_G)^{-1})
  \\= 
|W|^{-1} \chi( \ol{\M}_n^T(\P^1,X)^{\C^\times}, V \otimes \Eul(\Ind(\g/\t)))  \end{multline}
for any equivariant $K$-class $V$, where $\nu_G,\nu_T$ are the normal
bundles for the inclusion of the fixed point sets of the action of
$\C^\times$.  \label{nu} This formula holds as well after restricting
to fixed point components with markings $z_1,\ldots,z_n$ mapping to
$0 \in \P^1$ and $z_{n+1}$ mapping to $\infty$, and taking $V$ to be a
class of the form
$ \ev_1^* \alpha \otimes \ldots \ev_n^* \alpha \otimes \ev_{n+1}^*
\alpha_0$.
One obtains the formula (with superscript $\on{class}$ denoting the
classical Kirwan map)
 \begin{eqnarray*}
 \lan \tau_{X,-}^{G}(\alpha), \kappa_{X,G}^{\on{class}}( \alpha_0)  \ran
&=& 
|W|^{-1} \lan 
  \tau_{X,-}^{T,\g/\t}(\Restr^G_T \alpha), \kappa_{X,T}^{\on{class}}
    (\alpha_0) \ran \\
&=& \lan 
\Restr_G^T  \Eul(\g/\t)^{-1}  \tau_{X,-}^{T,\g/\t}(\Restr_T^G \alpha) ,
                                                                       \kappa_{X,G}^{\on{class}} (\alpha_0 )
                                                \ran 
\end{eqnarray*}
(the second by Martin's formula \cite{mar:sy}) from which the localized
abelianization formula \eqref{lab} follows.
\end{proof}

\begin{example} 
\label{grassex}
The fundamental solution in quantum K-theory for the Grassmannian
\(G(r,n)\)
is studied in Taipale \cite[Theorem 1]{taipale}, Wen \cite{wen}, and
Jockers, Mayr, Ninad, and Tabler \cite{jockers}.  Let
\[ X = \Hom(\C^r,\C^n), \quad G = GL(r) . \]
The group $G$ acts on $X$ by composition on the right:
$gx = x \circ g$.    Choose a polarization $\cL = X \times \C$
corresponding a positive central character of $G$. The semistable
locus is then 
\[ X^{\ss} = \{ x \in X \ | \ \rank(x) = r \} \]
and the git quotient
\[ X \qu G \cong G(r,n) .\]
The torus $T = (\C^\times)^r$ is a maximal torus of 
$G$ and the git quotient by the maximal torus is 
\[ X \qu T \cong (\P^{n-1})^r . \]
We claim that the localized gauged potential $\tau^G_{X,-}$ is
  the restriction of the $\Ind(\g/\t)$-twisted potential
  $\tau_{X,-}^{T,\g/\t}$ given by
  \begin{multline} \label{taipale} \tau^{T,\g/\t}_{X,-}(\alpha,q,z) =
    \sum_d q^d \tau^{T,g/\t}_{X,-,d}(\alpha,q,z) \prod_{1 \leq i < j
            \leq r} (1 - X_i X^\dual_j) X^{2\rho} 
    \\
    \tau^{T,\g/\t}_{X,-,d}(\alpha,q,z) := \sum_{d_1 + \ldots + d_r =
      d} 
\exp \left( \frac{ \Psi_d(\alpha)}{1 -
        {z}^{-1}} \right) (-1)^{d(r-1)} z^{ \lan d + \rho ,d + \rho \ran  - \lan \rho,\rho \ran} \\ \left( \frac{ \prod_{i < j}
        (1 - X_i X^\dual_j {z}^{d_i - d_j})}{ \prod_{i=1}^r \prod_{l =
          1}^{d_i} (1 - X_i {z}^l)^n } \right) .\end{multline}
  Without the factors $1 - X_i X^\dual_j {z}^{d_i - d_j}$ and
  $1 - X_i X^\dual_j$ the expression \eqref{taipale} in the Lemma
  would be the formula \eqref{Ifun} for $\tau_{X,-}^T$ discussed
  previously in the toric setting.  The additional factors are given
  by the Euler class of the index bundle
  \begin{eqnarray} 
 \label{eulind} \Eul (\Ind(\g/\t)) &=&
\frac{
 \prod_{i < j } \prod_{k= 0}^{d_j - d_i} ( 1 - X_i X^\dual_j 
    {z}^k ) }
{ \prod_{i < j  } \prod_{k= 1}^{d_j - d_i - 1} (  1 - X_j X^\dual_i 
    {z}^k ) } 
 \\
  &=&\nonumber  
\frac{
 \prod_{i < j } \prod_{k= 0}^{d_j - d_i} (1 - X_i X^\dual_j 
    {z}^{-k} )}
{ \prod_{i < j } (-1)^{d_j - d_i - 2}
\prod_{k=1}^{d_j - d_i - 1} z^{-2k} X_j X_i^\dual ( 1 - X_i X_j^\dual z^k  )} 
\\ 
&=&  \nonumber  (-1)^{d(r-1)} z^{ \sum_{i < j}  -2 (d_j - d_i) (d_j - d_i - 1)/2} 
 \\ && \prod_{i < j}   X^{2 \rho} (1 - X_i X_j^\dual) (1 - X_i X_j^\dual 
    z^{d_j - d_i}).  \\
&=&  \nonumber  (-1)^{d(r-1)} z^{ \lan d,d \ran  + 2 \lan d,\rho \ran} 
\\ && \prod_{i < j}  
  X^{2 \rho} (1 - X_i X_j^\dual) (1 - X_i X_j^\dual 
    z^{d_j - d_i})\\
&=& \nonumber
 (-1)^{d(r-1)} z^{ \lan d + \rho ,d + \rho \ran  - \lan \rho,\rho
    \ran}  \\&&
\prod_{i < j}  
  X^{2 \rho} (1 - X_i X_j^\dual) (1 - X_i X_j^\dual 
    z^{d_j - d_i})
\end{eqnarray}
where $\lan \cdot, \cdot \ran$ is the Killing form.  Note the missing
factor of $X^{2\rho}$ in \cite[(20)]{taipale}; this factor re-appears
in \cite[(31)]{taipale} but without the powers of $z$.

As pointed out to us by M. Zhang, the additional factors arising from
the $\rho$-shift in \eqref{eulind} vanish when one uses the level
$-1/2$-theory introduced by Ruan-Zhang \cite{rz:level}. It would be
interesting to know how the relations depend on the level structure,
and whether at level $-1/2$ the relations can be found using the
difference module structure in \eqref{ddrel}; see Jockers, Mayr,
Ninad, and Tabler \cite{jockers} for further developments.
\end{example}

\def\cprime{\('\)} \def\cprime{\('\)} \def\cprime{\('\)}
\def\cprime{\('\)} \def\cprime{\('\)} \def\cprime{\('\)}
\def\polhk#1{\setbox0=\hbox{#1}{\ooalign{\hidewidth
  \lower1.5ex\hbox{`}\hidewidth\crcr\unhbox0}}}
  \def\cprime{\('\)} \def\cprime{\('\)} \def\cprime{\('\)}
  \def\cprime{\('\)}

\end{document}